\documentclass[12pt, reqno]{amsart}
\usepackage{txfonts}    
\usepackage{amssymb}
\usepackage{eucal}
\usepackage{amsmath}
\usepackage{amsthm}
\usepackage{amscd}
\usepackage{amsxtra}
\usepackage[dvips]{color}
\usepackage{multicol}
\usepackage[all,pdf]{xy}        
\usepackage{graphicx}
\usepackage{color}
\usepackage{colordvi}
\usepackage{xspace}
\usepackage{rotating}
\usepackage{tikz}
\usepackage{appendix}
\usepackage{xcolor}
\usepackage{ifpdf}
\ifpdf
    \usepackage[colorlinks,final,backref=page,hyperindex]{hyperref}
\else
    \usepackage[colorlinks,final,backref=page,hyperindex,hypertex]{hyperref}
\fi

\allowdisplaybreaks
\newcommand {\emptycomment}[1]{} 

\newcommand{\nc}{\newcommand}
\newcommand{\delete}[1]{}

\nc{\mlabel}[1]{\label{#1}}  
\nc{\mcite}[1]{\cite{#1}}  
\nc{\mref}[1]{\ref{#1}}  
\nc{\meqref}[1]{\eqref{#1}} 
\nc{\mbibitem}[1]{\bibitem{#1}} 

\delete{
\nc{\mlabel}[1]{\label{#1}  
{\hfill \hspace{1cm}{\bf{{\ }\hfill(#1)}}}}
\nc{\mcite}[1]{\cite{#1}{{\bf{{\ }(#1)}}}}  
\nc{\mref}[1]{\ref{#1}{{\bf{{\ }(#1)}}}}  
\nc{\meqref}[1]{\eqref{#1}{{\bf{{\ }(#1)}}}} 
\nc{\mbibitem}[1]{\bibitem[\bf #1]{#1}} 
}

\makeatletter
\@namedef{subjclassname@2020}{\textup{2020} Mathematics Subject Classification}
\makeatother
\newtheorem{thm}{Theorem}[section]
\newtheorem{lem}[thm]{Lemma}
\newtheorem{cor}[thm]{Corollary}
\newtheorem{pro}[thm]{Proposition}
\theoremstyle{definition}
\newtheorem{defi}[thm]{Definition}
\newtheorem{ex}[thm]{Example}
\newtheorem{rmk}[thm]{Remark}

\newcommand{\dc}{{d\mkern-5mu\mathchar '26}}
\nc{\tred}[1]{\textcolor{red}{#1}}
\nc{\tblue}[1]{\textcolor{blue}{#1}}
\nc{\tgreen}[1]{\textcolor{green}{#1}}
\nc{\tpurple}[1]{\textcolor{purple}{#1}}
\nc{\btred}[1]{\textcolor{red}{\bf #1}}
\nc{\btblue}[1]{\textcolor{blue}{\bf #1}}
\nc{\btgreen}[1]{\textcolor{green}{\bf #1}}
\nc{\btpurple}[1]{\textcolor{purple}{\bf #1}}
\nc{\zis}[1]{\textcolor{purple}{#1}}

\nc{\cm}[1]{\textcolor{red}{Chengming:#1}}
\nc{\li}[1]{\textcolor{blue}{#1}}
\nc{\lir}[1]{\textcolor{blue}{Li:#1}}
\nc{\sy}[1]{\textcolor{purple}{Siyuan:#1}}

\nc{\name}[1]{{\bf #1}}
\nc{\acd}{associative crossed datum\xspace}
\nc{\acds}{associative crossed data\xspace}

\nc{\mmod}[1]{{\ (\mathrm{mod}~{#1})}}

\nc{\qcl}{QCL\xspace} 
\nc{\qcls}{QCLs\xspace}

\nc{\bfK}{\mathbf{K}}
\nc{\dera}{{d_1}}
\nc{\derb}{{d_2}}

\nc{\vspa}{\vspace{-.1cm}}
\nc{\vspb}{\vspace{-.2cm}}
\nc{\vspc}{\vspace{-.3cm}}
\nc{\vspd}{\vspace{-.4cm}}
\nc{\vspe}{\vspace{-.5cm}}

\nc{\twovec}[2]{\left(\begin{array}{c} #1 \\ #2\end{array} \right )}
\nc{\threevec}[3]{\left(\begin{array}{c} #1 \\ #2 \\ #3 \end{array}\right )}
\nc{\twomatrix}[4]{\left(\begin{array}{cc} #1 & #2\\ #3 & #4 \end{array} \right)}
\nc{\threematrix}[9]{{\left(\begin{matrix} #1 & #2 & #3\\ #4 & #5 & #6 \\ #7 & #8 & #9 \end{matrix} \right)}}
\nc{\twodet}[4]{\left|\begin{array}{cc} #1 & #2\\ #3 & #4 \end{array} \right|}

\nc{\rk}{\mathrm{r}}
\newcommand{\g}{\mathfrak g}

\newcommand{\ad}{\mathrm{ad}}

\newtheoremstyle{redtitle}%
{3pt}{3pt}
\itshape
{}
{\color{red}\bfseries}
{.}
{ }
{\thmname{#1}\thmnumber{ #2}\thmnote{ (#3)}}

\theoremstyle{redtitle}

\nc{\tforall}{\text{ for all }}

\nc{\svec}[2]{{\tiny\left(\begin{matrix}#1\\
#2\end{matrix}\right)\,}}  
\nc{\ssvec}[2]{{\tiny\left(\begin{matrix}#1\\
#2\end{matrix}\right)\,}} 

\nc{\typeI}{local cocycle $3$-Lie bialgebra\xspace}
\nc{\typeIs}{local cocycle $3$-Lie bialgebras\xspace}
\nc{\typeII}{double construction $3$-Lie bialgebra\xspace}
\nc{\typeIIs}{double construction $3$-Lie bialgebras\xspace}

\nc{\bia}{{$\mathcal{P}$-bimodule ${\bf k}$-algebra}\xspace}
\nc{\bias}{{$\mathcal{P}$-bimodule ${\bf k}$-algebras}\xspace}

\nc{\rmi}{{\mathrm{I}}}
\nc{\rmii}{{\mathrm{II}}}
\nc{\rmiii}{{\mathrm{III}}}
\nc{\pr}{{\mathrm{pr}}}

\nc{\OT}{constant $\theta$-}
\nc{\T}{$\theta$-}
\nc{\IT}{inverse $\theta$-}
\nc{\rad}{\mathrm{ad}}
\nc{\bt}{\tilde{T}}
\nc{\bi}{\tilde{id}}
\nc{\asi}{ASI\xspace}
\nc{\qadm}{$Q$-admissible\xspace}
\nc{\aybe}{AYBE\xspace}
\nc{\admset}{\{\pm x\}\cup (-x+K^{\times}) \cup K^{\times} x^{-1}}
\nc{\hr}{\hat{r}}
\nc{\dualrep}{gives a dual representation\xspace}
\nc{\admt}{admissible to\xspace}

\nc{\opa}{\cdot_A}
\nc{\opb}{\cdot_B}

\nc{\post}{positive type\xspace}
\nc{\negt}{negative type\xspace}
\nc{\invt}{inverse type\xspace}

\nc{\pll}{\beta}
\nc{\plc}{\epsilon}

\nc{\ass}{{\mathit{Ass}}}
\nc{\lie}{{\mathit{Lie}}}
\nc{\comm}{{\mathit{Comm}}}
\nc{\dend}{{\mathit{Dend}}}
\nc{\zinb}{{\mathit{Zinb}}}
\nc{\tdend}{{\mathit{TDend}}}
\nc{\prelie}{{\mathit{preLie}}}
\nc{\postlie}{{\mathit{PostLie}}}
\nc{\quado}{{\mathit{Quad}}}
\nc{\octo}{{\mathit{Octo}}}
\nc{\ldend}{{\mathit{ldend}}}
\nc{\lquad}{{\mathit{LQuad}}}

 \nc{\adec}{\check{;}} \nc{\aop}{\alpha}
\nc{\dftimes}{\widetilde{\otimes}} \nc{\dfl}{\succ} \nc{\dfr}{\prec}
\nc{\dfc}{\circ} \nc{\dfb}{\bullet} \nc{\dft}{\star}
\nc{\dfcf}{{\mathbf k}} \nc{\apr}{\ast} \nc{\spr}{\cdot}
\nc{\twopr}{\circ} \nc{\tspr}{\star} \nc{\sempr}{\ast}
\nc{\disp}[1]{\displaystyle{#1}}
\nc{\bin}[2]{ (_{\stackrel{\scs{#1}}{\scs{#2}}})}  
\nc{\binc}[2]{ \left (\!\! \begin{array}{c} \scs{#1}\\
    \scs{#2} \end{array}\!\! \right )}  
\nc{\bincc}[2]{  \left ( {\scs{#1} \atop
    \vspace{-.5cm}\scs{#2}} \right )}  
\nc{\sarray}[2]{\begin{array}{c}#1 \vspace{.1cm}\\ \hline
    \vspace{-.35cm} \\ #2 \end{array}}
\nc{\bs}{\bar{S}} \nc{\dcup}{\stackrel{\bullet}{\cup}}
\nc{\dbigcup}{\stackrel{\bullet}{\bigcup}} \nc{\etree}{\big |}
\nc{\la}{\longrightarrow} \nc{\fe}{\'{e}} \nc{\rar}{\rightarrow}
\nc{\dar}{\downarrow} \nc{\dap}[1]{\downarrow
\rlap{$\scriptstyle{#1}$}} \nc{\uap}[1]{\uparrow
\rlap{$\scriptstyle{#1}$}} \nc{\defeq}{\stackrel{\rm def}{=}}
\nc{\dis}[1]{\displaystyle{#1}} \nc{\dotcup}{\,
\displaystyle{\bigcup^\bullet}\ } \nc{\sdotcup}{\tiny{
\displaystyle{\bigcup^\bullet}\ }} \nc{\hcm}{\ \hat{,}\ }
\nc{\hcirc}{\hat{\circ}} \nc{\hts}{\hat{\shpr}}
\nc{\lts}{\stackrel{\leftarrow}{\shpr}}
\nc{\rts}{\stackrel{\rightarrow}{\shpr}} \nc{\lleft}{[}
\nc{\lright}{]} \nc{\uni}[1]{\tilde{#1}} \nc{\wor}[1]{\check{#1}}
\nc{\free}[1]{\bar{#1}} \nc{\den}[1]{\check{#1}} \nc{\lrpa}{\wr}
\nc{\curlyl}{\left \{ \begin{array}{c} {} \\ {} \end{array}
    \right .  \!\!\!\!\!\!\!}
\nc{\curlyr}{ \!\!\!\!\!\!\!
    \left . \begin{array}{c} {} \\ {} \end{array}
    \right \} }
\nc{\leaf}{\ell}       
\nc{\longmid}{\left | \begin{array}{c} {} \\ {} \end{array}
    \right . \!\!\!\!\!\!\!}
\nc{\ot}{\otimes} \nc{\sot}{{\scriptstyle{\ot}}}
\nc{\otm}{\overline{\ot}}
\nc{\ora}[1]{\stackrel{#1}{\rar}}
\nc{\ola}[1]{\stackrel{#1}{\la}}
\nc{\pltree}{\calt^\pl}
\nc{\epltree}{\calt^{\pl,\NC}}
\nc{\rbpltree}{\calt^r}
\nc{\scs}[1]{\scriptstyle{#1}} \nc{\mrm}[1]{{\rm #1}}
\nc{\dirlim}{\displaystyle{\lim_{\longrightarrow}}\,}
\nc{\invlim}{\displaystyle{\lim_{\longleftarrow}}\,}
\nc{\mvp}{\vspace{0.5cm}} \nc{\svp}{\vspace{2cm}}
\nc{\vp}{\vspace{8cm}} \nc{\proofbegin}{\noindent{\bf Proof: }}
\nc{\proofend}{$\blacksquare$ \vspace{0.5cm}}
\nc{\freerbpl}{{F^{\mathrm RBPL}}}
\nc{\sha}{{\mbox{\cyr X}}}  
\nc{\ncsha}{{\mbox{\cyr X}^{\mathrm NC}}} \nc{\ncshao}{{\mbox{\cyr
X}^{\mathrm NC,\,0}}}
\nc{\shpr}{\diamond}    
\nc{\shprm}{\overline{\diamond}}    
\nc{\shpro}{\diamond^0} 
\nc{\shprr}{\diamond^r}  
\nc{\shpra}{\overline{\diamond}^r}
\nc{\shpru}{\check{\diamond}} \nc{\catpr}{\diamond_l}
\nc{\rcatpr}{\diamond_r} \nc{\lapr}{\diamond_a}
\nc{\sqcupm}{\ot}
\nc{\lepr}{\diamond_e} \nc{\vep}{\varepsilon} \nc{\labs}{\mid\!}
\nc{\rabs}{\!\mid} \nc{\hsha}{\widehat{\sha}}
\nc{\lsha}{\stackrel{\leftarrow}{\sha}}
\nc{\rsha}{\stackrel{\rightarrow}{\sha}} \nc{\lc}{\lfloor}
\nc{\rc}{\rfloor}
\nc{\tpr}{\sqcup}
\nc{\nctpr}{\vee}
\nc{\plpr}{\star}
\nc{\rbplpr}{\bar{\plpr}}
\nc{\sqmon}[1]{\langle #1\rangle}
\nc{\forest}{\calf}
\nc{\altx}{\Lambda_X} \nc{\vecT}{\vec{T}} \nc{\onetree}{\bullet}
\nc{\Ao}{\check{A}}
\nc{\seta}{\underline{\Ao}}
\nc{\deltaa}{\overline{\delta}}
\nc{\trho}{\tilde{\rho}}

\nc{\rpr}{\circ}
\nc{\dpr}{{\tiny\diamond}}
\nc{\rprpm}{{\rpr}}

\nc{\mmbox}[1]{\mbox{\ #1\ }} \nc{\ann}{\mrm{ann}}
\nc{\Aut}{\mrm{Aut}} \nc{\can}{\mrm{can}}
\nc{\twoalg}{{two-sided algebra}\xspace}
\nc{\colim}{\mrm{colim}}
\nc{\Cont}{\mrm{Cont}} \nc{\rchar}{\mrm{char}}
\nc{\cok}{\mrm{coker}} \nc{\dtf}{{R-{\rm tf}}} \nc{\dtor}{{R-{\rm
tor}}}

\nc{\depth}{{\mrm d}}
\nc{\Div}{{\mrm Div}} \nc{\End}{\mrm{End}} \nc{\Ext}{\mrm{Ext}}
\nc{\Fil}{\mrm{Fil}} \nc{\Frob}{\mrm{Frob}} \nc{\Gal}{\mrm{Gal}}
\nc{\GL}{\mrm{GL}} \nc{\Hom}{\mrm{Hom}} \nc{\hsr}{\mrm{H}}
\nc{\hpol}{\mrm{HP}} \nc{\id}{\mrm{id}} \nc{\im}{\mrm{im}}
\nc{\incl}{\mrm{incl}} \nc{\length}{\mrm{length}}
\nc{\LR}{\mrm{LR}} \nc{\mchar}{\rm char} \nc{\NC}{\mrm{NC}}
\nc{\mpart}{\mrm{part}} \nc{\pl}{\mrm{PL}}
\nc{\ql}{{\QQ_\ell}} \nc{\qp}{{\QQ_p}}
\nc{\rank}{\mrm{rank}} \nc{\rba}{\rm{RBA }} \nc{\rbas}{\rm{RBAs }}
\nc{\rbpl}{\mrm{RBPL}}
\nc{\rdef}{\mrm{def}} \nc{\rdiv}{{\rm div}} \nc{\rtf}{{\rm tf}}
\nc{\rtor}{{\rm tor}} \nc{\res}{\mrm{res}} \nc{\SL}{\mrm{SL}}
\nc{\Spec}{\mrm{Spec}} \nc{\tor}{\mrm{tor}} \nc{\Tr}{\mrm{Tr}}
\nc{\mtr}{\mrm{sk}}
\nc{\rbw}{\rm{RBW }} \nc{\rbws}{\rm{RBWs }} \nc{\rcot}{\mrm{cot}}
\nc{\rest}{\rm{controlled}\xspace}

\nc{\ab}{\mathbf{Ab}} \nc{\Alg}{\mathbf{Alg}}
\nc{\Dend}{\mathbf{DD}} \nc{\bfk}{{\bf k}} \nc{\bfone}{{\bf 1}}
\nc{\rb}{\mathrm{RB}}

\nc{\BA}{{\mathbb A}} \nc{\CC}{{\mathbb C}} \nc{\DD}{{\mathbb D}}
\nc{\EE}{{\mathbb E}} \nc{\FF}{{\mathbb F}} \nc{\GG}{{\mathbb G}}
\nc{\HH}{{\mathbb H}} \nc{\LL}{{\mathbb L}} \nc{\NN}{{\mathbb N}}
\nc{\QQ}{{\mathbb Q}} \nc{\RR}{{\mathbb R}} \nc{\BS}{{\mathbb{S}}} \nc{\TT}{{\mathbb T}}
\nc{\VV}{{\mathbb V}} \nc{\ZZ}{{\mathbb Z}}

\nc{\calao}{{\mathcal A}} \nc{\cala}{{\mathcal A}}
\nc{\calc}{{\mathcal C}} \nc{\cald}{{\mathcal D}}
\nc{\cale}{{\mathcal E}} \nc{\calf}{{\mathcal F}}
\nc{\calfr}{{{\mathcal F}^{\,r}}} \nc{\calfo}{{\mathcal F}^0}
\nc{\calfro}{{\mathcal F}^{\,r,0}} \nc{\oF}{\overline{F}}
\nc{\calg}{{\mathcal G}} \nc{\calh}{{\mathcal H}}
\nc{\cali}{{\mathcal I}} \nc{\calj}{{\mathcal J}}
\nc{\call}{{\mathcal L}} \nc{\calm}{{\mathcal M}}
\nc{\caln}{{\mathcal N}} \nc{\calo}{{\mathcal O}}
\nc{\calp}{{\mathcal P}} \nc{\calq}{{\mathcal Q}} \nc{\calr}{{\mathcal R}}
\nc{\calt}{{\mathcal T}} \nc{\caltr}{{\mathcal T}^{\,r}}
\nc{\calu}{{\mathcal U}} \nc{\calv}{{\mathcal V}}
\nc{\calw}{{\mathcal W}} \nc{\calx}{{\mathcal X}}
\nc{\CA}{\mathcal{A}}

\nc{\fraka}{{\mathfrak a}} \nc{\frakB}{{\mathfrak B}}
\nc{\frakb}{{\mathfrak b}} \nc{\frakd}{{\mathfrak d}}
\nc{\oD}{\overline{D}}
\nc{\frakF}{{\mathfrak F}} \nc{\frakg}{{\mathfrak g}}
\nc{\frakm}{{\mathfrak m}} \nc{\frakM}{{\mathfrak M}}
\nc{\frakMo}{{\mathfrak M}^0} \nc{\frakp}{{\mathfrak p}}
\nc{\frakS}{{\mathfrak S}} \nc{\frakSo}{{\mathfrak S}^0}
\nc{\fraks}{{\mathfrak s}} \nc{\os}{\overline{\fraks}}
\nc{\frakT}{{\mathfrak T}}
\nc{\oT}{\overline{T}}
\nc{\frakX}{{\mathfrak X}} \nc{\frakXo}{{\mathfrak X}^0}
\nc{\frakx}{{\mathbf x}}
\nc{\frakTx}{\frakT}      
\nc{\frakTa}{\frakT^a}    
\nc{\frakTxo}{\frakTx^0}   
\nc{\caltao}{\calt^{a,0}}   
\nc{\ox}{\overline{\frakx}} \nc{\fraky}{{\mathfrak y}}
\nc{\frakz}{{\mathfrak z}} \nc{\oX}{\overline{X}}
\nc{\ho}{\hat{\otimes}}
\font\cyr=wncyr10

\nc{\al}{\alpha}
\nc{\lam}{\lambda}
\nc{\lr}{\longrightarrow}

\begin{document}

\title[Poisson bialgebras by deformations-to-quasiclassical limits]{Poisson bialgebras by deformations-to-quasiclassical limits}

\author{Siyuan Chen}
\address{School of mathematics, Hunan Institute of Science and Technology, Yueyang, Hunan Province 414006, China}
\email{csy016@hnist.edu.cn}

\author{Chengming Bai}
\address{Chern Institute of Mathematics \& LPMC, Nankai University, Tianjin 300071 China}
\email{baicm@nankai.edu.cn}

\begin{abstract}
 Poisson algebras are the quasiclassical limits of associative algebra deformations of commutative associative algebras.
 This paper extends this process to the level of bialgebras. We derive Poisson bialgebras as the quasiclassical limits of antisymmetric infinitesimal bialgebra deformations of commutative and cocommutative antisymmetric
 infinitesimal bialgebras. It might be regarded as the ``infinitesimal" version of the
 quantization process of Lie bialgebras in terms of Hopf algebras.  Such deformations-quasiclassical limits
 process for a Poisson bialgebra is equivalently characterized in
 terms of the introduced notions of
 deformations of a matched pair of associative algebras as well as a standard Manin triple of associative
 algebras, whose corresponding quasiclassical limits
 are a matched pair of Poisson algebras and a standard Manin triple of Poisson algebras,
 respectively. We illustrate these equivalent deformations-quasiclassical
 limits processes via coherent derivations.
\end{abstract}

\subjclass[2020]{
    13D10, 
    16W60, 
    17B63, 
    53D55,  
    13N15   
}

\keywords{Deformation, quasiclassical limit,
    antisymmetric infinitesimal bialgebra, Poisson bialgebra, derivation}


\maketitle

\vspace{-1.5cm}

\tableofcontents

\vspace{-1.7cm}


\section{Introduction}
Poisson algebras originate from analytical mechanics and have
found applications in various fields of mathematics and physics
\cite{Hue,Ko, LG, RS}. It is well-known that Poisson algebras are
quasiclassical limits of associative algebra deformations of
commutative associative algebras. The purpose of this paper is to
lift the deformation-quasiclassical limit process for Poisson
algebras to the level of Poisson bialgebras.
\subsection{Antisymmetric infinitesimal bialgebra deformations and the corresponding quasiclassical limits}

A bialgebra structure consists of an algebra structure and a
coalgebra structure that satisfy certain compatibility conditions.
Typical examples of bialgebra structures on associative algebras
are the ones on Hopf algebras, in which comultiplications are
homomorphisms. Hopf algebras arise from topology and also appear
in many other areas \cite{Ab, CK, Ga, Hen}. 
The
deformation theory of Hopf algebras plays an important role in the
quantization process of Lie bialgebras \cite{Drq}, where Lie bialgebras are the algebraic
structures corresponding to Poisson-Lie groups
\cite{chari1995guide,Drh}. Note that in this quantization process
of Lie bialgebras, a Lie bialgebra might be thought to be
``inside" a commutative and cocommutative Hopf algebra such that
the bracket and cobracket are compatible in the Poisson sense with
the product and coproduct, respectively and then one studies
deformations of the resulting Hopf algebra in the ``direction" of
the bracket and cobracket.

There is the ``infinitesimal" version of the bialgebra structures
on associative algebras on Hopf algebras, where comultiplications
are derivations, namely infinitesimal bialgebras. They were
introduced by Join and Rota in order to provide an algebraic
framework for the calculus of divided difference \cite{JR}. As a
special class, the notion of antisymmetric infinitesimal
bialgebras was introduced in \cite{Bai} as the equivalent
structures of double constructions of Frobenius algebras, whereas
they were also called ``associative D-bialgebras" in
\cite{zhelyabian1997jordan} and ``balanced infinitesimal
bialgebras" in \cite{A3}. Recently, the theory of antisymmetric
infinitesimal bialgebras has been further developed. For example,
\cite{Lin} studied derivations on antisymmetric infinitesimal
bialgebras and the induced structures. In \cite{BGM,HB},
antisymmetric infinitesimal bialgebras were extended to the
context of Rota-Baxter algebras and  associative conformal
algebras.

Therefore it is natural to consider the ``infinitesimal" version
of the aforementioned quantization process of Lie bialgebras in
terms of Hopf algebras, by replacing ``Hopf" with
``infinitesimal". There a commutative and cocommutative infinitesimal bialgebra is automatically
antisymmetric, and one then constructs the corresponding
deformation theory. So one should consider the deformations of
antisymmetric infinitesimal bialgebras and the corresponding
quasiclassical limits. Moreover, such a study extends the involved
results at the level of algebras to the context of bialgebras,
establishing relationships between different bialgebras with
potential applications in mathematical physics.

Indeed we show that the quasiclassical limits of antisymmetric
infinitesimal bialgebra deformations of commutative and
cocommutative antisymmetric infinitesimal bialgebras are Poisson
bialgebras introduced in \cite{NB}. It is exactly the bialgebra
version of the fact that Poisson algebras are the quasiclassical
limits of associative algebra deformations of commutative
associative algebras. Moreover, inspired by the construction of
associative algebra deformations from derivations \cite{G3}, we
give a similar construction of antisymmetric infinitesimal
bialgebra deformations, which is related to the theory of
differential antisymmetric infinitesimal bialgebras \cite{Lin},
illustrating explicitly the deformations-quasiclassical limits
 process for Poisson bialgebras.

\subsection{Equivalent characterizations of the deformations-quasiclassical limits process for Poisson bialgebras}
An antisymmetric infinitesimal bialgebra can be equivalently
characterized by a matched pair of associative algebras or a
standard Manin triple of associative algebras. The latter is also
called a double construction of Frobenius algebra \cite{Bai},
whereas Frobenius algebras 
have important applications in representation theory, $2$D
topological quantum field theories and string theory \cite{Koc,
La, SY}. Note that a matched pair of associative algebras consists
of a pair of associative algebras that give an associative algebra
structure on the direct sum of the underlying vector spaces via
interacting actions. It generalizes the notions of bimodules and
bimodule algebras of associative algebras, appearing in the
extension and deformation theory of associative algebras
\cite{Ag}.

These equivalences can be summarized as follows.
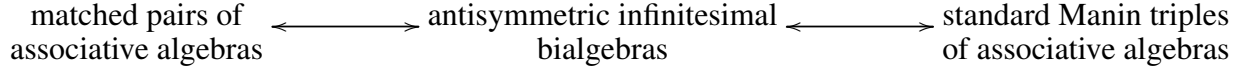
\begin{equation*}
\xymatrix{
\txt{matched pairs of\\ associative algebras}\ar@{<->}[rr] &&\txt{antisymmetric infinitesimal\\ bialgebras} \ar@{<->}[rr]  &&\txt{standard Manin triples\\ of associative algebras}
}
\end{equation*}
There are similar characterizations of Poisson bialgebras
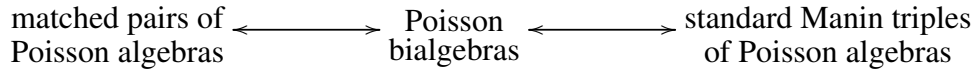
\cite{NB}.
\begin{equation*}
\xymatrix{
\txt{matched pairs of\\ Poisson algebras}\ar@{<->}[rr] &&\txt{Poisson\\ bialgebras} \ar@{<->}[rr]  &&\txt{standard Manin triples\\ of Poisson algebras}
}
\end{equation*}
So it is natural to consider whether the
deformations-quasiclassical limits
 processes can be fit into these equivalent relationships, that is, to consider deformations  of matched pairs of associative algebras as well as
standard Manin triples of associative algebras, and the
corresponding quasiclassical limits.

Explicitly, the quasiclassical limit of a deformation of a matched
pair of commutative associative algebras is a pair of Poisson
algebras giving a Poisson algebra structure on the direct sum of
the underlying vector spaces via interacting actions, that is, a
matched pair of Poisson algebras. Therefore, the
deformations-quasiclassical limits process for a Poisson bialgebra
is equivalently characterized by the deformations-quasiclassical
limits process for a matched pair of Poisson algebras (see (\ref{diag1}) for a summary commutative diagram). Similarly,
 the quasiclassical limit of a
deformation of a standard Manin triple of commutative associative
algebras is a standard Manin triple of Poisson algebras, giving
another equivalent characterization of the
deformations-quasiclassical limits process for a Poisson
bialgebra (see (\ref{diag2}) for a summary commutative diagram).

\subsection{Organization of the paper}
The paper is outlined as follows.

In Section \mref{S2}, we give the notion of an antisymmetric
infinitesimal bialgebra deformation and determine the
corresponding quasiclassical limit. We construct antisymmetric
infinitesimal bialgebra deformations via coherent derivations.

In Section \mref{S3}, we introduce the notion of an associative
matched pair deformation and determine the corresponding
quasiclassical limit, characterizing the
deformations-quasiclassical limits process for a Poisson bialgebra
equivalently. We also study the correspondence of the
antisymmetric infinitesimal bialgebra deformations induced by
coherent derivations with the associative matched pair deformations
induced by derivations.

In Section \mref{S4}, we introduce the notion of an associative
Manin triple deformation and determine the corresponding
quasiclassical limit, which gives another equivalent
characterizations of the deformations-quasiclassical limits
process for a Poisson bialgebra.
\smallskip

\noindent {\bf Notations.} Throughout this paper, we fix a ground
field $\bfk$ of characteristic $0$ for vector spaces, algebras and
tensor products. We also fix the power series ring
$\mathbf{K}=\bfk[[h]]$ as the ground algebra for topological
modules, topological algebras and topological
 tensor products.
For a vector space $A$ with an operation $\circ$, the linear maps
$L_\circ, R_\circ: A\rightarrow {\rm End}_{\bf k}(A)$ are
respectively defined by \vspb
$$L_\circ(x)y:=x\circ y,\;\;R_\circ(x)y:=y\circ x,\;\; x,y\in A.
\vspa
$$
In particular, when $(A,\circ=[,])$ is a Lie algebra, we use the
usual notion of the adjoint operator ${\rm ad}_{[,]}(x)(y):=[x,y]$
for all $x,y\in A$.
Let $\mathbb{N}$ and $\mathbb{Z}^{+}$ denote the set of natural numbers and the set of positive integers respectively. 
\vspb

\section{Antisymmetric infinitesimal bialgebra deformations }\mlabel{S2}
We introduce the notions of an antisymmetric infinitesimal bialgebra deformation and the corresponding
quasiclassical limit. Then we show that the quasiclassical limit is a Poisson bialgebra. Finally, we construct antisymmetric infinitesimal bialgebra deformations via coherent derivations.

\subsection{Associative deformations, coassociative deformations and the corresponding quasiclassical limits}
\mlabel{ss:deformgen}

We  recall the needed notations for (formal) deformations as follows \cite{ES}. Let
 $V$ be a vector space. Consider the set of formal series
\[V_h:=V[[h]]:= \Big\{\sum_{s=0}^{\infty}v_{s}~h^{s}~\Big|~v_{s}\in V\Big\}.
\]
It is clear that $V[[h]]$ is a $\mathbf{K}$-module.
For a fixed real number $C>1$, define the \textbf{$h$-adic norm} $\|~\|$ on $V[[h]]$ by
\[\Big\|\sum_{s=0}^{\infty}v_{s}~h^{s}\Big\|:= C^{-m},\]
where $m$ is the smallest integer such that $v_{m}\neq 0$.
A $\mathbf{K}$-module that is $\bf K$-linearly isomorphic to $V[[h]]$ for a vector space $V$ is called a \textbf{topologically free $\mathbf{K}$-module}. In particular, $V_h$ is a topologically  free $\bfK$-module.

 Let $A_{h}$ be a topologically free $\mathbf{K}$-module. If there is a $\mathbf{K}$-bilinear operation $\circ_{h}$ on $A_{h}$ such that
 \vspb
  \[(x_h\circ_{h} y_h)\circ_{h} z_h=x_h\circ_{h}(y_h\circ_{h} z_h),\;\; \forall x_h,y_h,z_h\in A_h,\]
  then $(A_{h},\circ_{h})$ is called a \textbf{topologically free associative algebra}.
\begin{defi}
  A \textbf{deformation of an associative algebra}$(A,\circ)$ is a topologically free associative algebra $(A_{h},\circ_{h})$ such that
    \begin{equation}
    x\circ_{h}y \equiv x\circ y\pmod h,\;\;\forall x,y\in A.\mlabel{defa2}
    \end{equation}
\end{defi}
A deformation $(A_{h},\circ_{h})$ of an associative algebra
$(A,\circ)$ is also called an \textbf{associative algebra
deformation} or simply an \textbf{associative deformation} of
$(A,\circ)$.

\begin{defi}
Let $(A_{h},\circ_{h})$ be an associative deformation of a commutative associative algebra $(A,\circ)$. Let $\{,\}$ be an operation on $A$ such that
\[\{x,y\}\equiv\frac{x\circ_{h}y-y\circ_{h}x}{h}\pmod{h},\;\;\forall  x, y\in A.\]
Then $(A,\{,\},\circ)$ is called the
\textbf{quasiclassical limit (\qcl) of the associative
deformation} $(A_{h},\circ_{h})$.
\end{defi}

Recall that a \textbf{Poisson algebra}
$(P,\{,\},\circ)$ consists of a Lie algebra $(P,\{,\})$ and a
commutative associative algebra $(P,\circ)$ such that
\vspb
\begin{equation}\mlabel{PostP1}
    \{x,y\circ z\}=\{x,y\}\circ z+y\circ\{x,z\},\;\;\forall x,y,z\in P.
\end{equation}
The following result is well known.

\begin{thm}{\rm(\mcite{ES})}\label{thacl}
Let $(A_{h},\circ_{h})$ be an associative deformation of a commutative associative algebra $(A,\circ)$. Then its \qcl $(A, \{,\},\circ)$ is a Poisson algebra.
\end{thm}

Let $V[[h]]$ and $W[[h]]$ be topologically free $\mathbf{K}$-modules. The \textbf{topological tensor product} $V[[h]]\hat{\otimes}W[[h]]$ is the completion of
$V[[h]]\otimes_{\mathbf{K}}W[[h]]$ under the $h$-adic norm. So the topological tensor product is the completion of the $\mathbf{K}$-bilinear extension of the usual (algebraic) tensor product.

Let $A$ be a vector space. Define the \textbf{twisting operator} $\sigma:A\ot A\to A\ot A$ by
$$\sigma (x\otimes y)=y\otimes x, \quad\forall x,y\in A.
$$
Similarly, define the \textbf{topological twisting operator} $\sigma_{\bf K}:A_h\hat{\ot} A_h\to A_h\hat{\ot} A_h$ by
$$\sigma_{\bf K}(x_h{\hat \otimes} y_h)=y_h{\hat \otimes} x_h,\;\;\forall x_h,y_h\in A_h.
$$
\begin{defi}
   Let $A$ be a vector space, $\Delta:A\to A\ot A$ be a linear map. If
    \[(\Delta\ot \id_{A})\Delta=(\id_{A}\otimes\Delta)\Delta,\]
    then we call $(A,\Delta)$ a\textbf{ coassociative coalgebra}. Moreover, if $\Delta=\sigma\Delta$, then we say that $(A,\Delta)$ is
    \textbf{cocommutative}.

    Let $A_h$ be a topologically free $\bf K$-module, $\Delta_h:A_h\to A_h\hat{\ot}A_h$ be a $\bf K$-linear map.
    If $(A_h,\Delta_h)$ satisfies the following identity, then we call $(A_h,\Delta_h)$ a \textbf{topologically free coassociative coalgebra}.
    \begin{equation}\label{eq:3}
    (\Delta_h\hat{\otimes}\id_{A_h})\Delta_h=(\id_{A_h}\hat{\otimes}\Delta_h)\Delta_h.
    \end{equation}

    Let $(A,\Delta)$ be a coassociative coalgebra. If a topologically free coassociative coalgebra $(A_h,\Delta_h)$ satisfies
    \begin{eqnarray}
    \Delta_{h}(x) &\equiv& \Delta (x)\pmod{h},\;\;\forall x\in A,\label{cmh}
    \end{eqnarray}
    then we call $(A_{h},\Delta_{h})$ a \textbf{coassociative deformation} of $(A,\Delta)$.
    Moreover, if $(A,\Delta)$ is cocommutative, then $(A,\delta, \Delta)$ is called the \textbf{quasiclassical limit {\rm(}\qcl{\rm)} of the coassociative
deformation} $(A_{h},\Delta_{h})$, where
    $\delta :A\to A\ot A$ is defined by
    \begin{eqnarray}
    \delta(x)\equiv\frac{\Delta_h(x ) - \sigma_{\mathbf{K}}(\Delta_h(x))}{h} \pmod{h},\;\;\forall x\in A.
    \end{eqnarray}
\end{defi}
Let $\g$ be a vector space, $\delta: \g \to \g \otimes \g$ be a linear map, if $\delta$ satisfies $\delta = -\sigma \delta$ and the following identity
\begin{equation}
(\id_{\g\otimes\g\otimes\g} + \tau + \tau^2)(\id_{\g} \otimes \delta)\delta = 0,\label{lieco}
\end{equation}
where $$\tau:\g \otimes \g \otimes \g\to\g \otimes \g \otimes
\g,\; \tau(x \otimes y\otimes z):= z\otimes x\otimes y, \;\;
\forall x, y, z \in \g,$$ then $(\g, \delta)$ is called a
\textbf{Lie coalgebra}. Let $(\g, [\ ,\ ])$ be a Lie algebra and
$(\g, \delta)$ be a Lie coalgebra. If
\begin{equation}
\delta([x, y])= (\ad_{[,]}(x) \otimes\id_{\g} +\id_{\g}
\otimes\ad_{[,]}(x) )\delta(y) -(\ad_{[,]}(y) \otimes\id_{\g}
+\id_{\g} \otimes\ad_{[,]}(y) )\delta(x), \;\;\forall x,y\in
\g,\label{liebi}
\end{equation}
then we call $(\g, [\ ,\ ], \delta)$ a \textbf{Lie bialgebra}.

    Let $(A,\delta)$ be a Lie coalgebra and $(A,\Delta)$ be a cocommutative coassociative coalgebra. If
    \begin{equation}
    (\mathrm {id}_{A} \otimes \Delta) \delta(x)=(\delta \otimes \mathrm {id}_{A}) \Delta(x)+(\sigma \otimes \mathrm {id}_{A})(\mathrm {id}_{A} \otimes \delta) \Delta(x), \quad \forall x \in A,\label{pco}
    \end{equation}
    then we call $(A, \delta, \Delta)$ a \textbf{Poisson coalgebra}.

\begin{thm}
    \label{qclco}
    Let $(A,\Delta)$ be a cocommutative coassociative coalgebra, $(A_{h},\Delta_{h})$ be the coassociative deformation and $(A,\delta, \Delta)$ be the QCL. Then $(A,\delta, \Delta)$ is a Poisson coalgebra.
\end{thm}
\begin{proof}
   Let $x\in A$. Set
    \begin{equation}
(\id_{A_h}\hat{\otimes}\Delta_h)\Delta_h(x)=\sum_{s}a_{h,s}\hat{\otimes}b_{h,s}\hat{\otimes}c_{h,s},\label{cox1.1}
    \end{equation}
    where $a_{h,s},b_{h,s},c_{h,s}\in A_h$. Since
    \begin{equation}
    (\Delta_h\hat{\otimes}\id_{A_h})\Delta_h(x)=(\id_{A_h}\hat{\otimes}\Delta_h)\Delta_h(x),\label{coa}
    \end{equation}
    we obtain
    \begin{equation}(\Delta_h\hat{\otimes}\id_{A_h})\Delta_h(x)=\sum_{s}a_{h,s}\hat{\otimes}b_{h,s}\hat{\otimes}c_{h,s}.\label{cox2}
    \end{equation}
    We deduce from (\ref{cox1.1}) that
    \begin{equation}
    (\id_{A_h}\hat{\otimes}\sigma_{\bf K}\Delta_h)\Delta_h(x)=\sum_{s}a_{h,s}\hat{\otimes}c_{h,s}\hat{\otimes}b_{h,s}.\label{cox1.2}
    \end{equation}
    By (\ref{cox2}), we have
    \begin{eqnarray}
    (\id_{A_h}\hat{\otimes}\Delta_h)\sigma_{\bf K}\Delta_h(x)&=&\sum_{s}c_{h,s}\hat{\otimes}a_{h,s}\hat{\otimes}b_{h,s},\label{cox1.3}\\
    (\id_{A_h}\hat{\otimes}\sigma_{\bf K}\Delta_h)\sigma_{\bf K}\Delta_h(x)&=&\sum_{s}c_{h,s}\hat{\otimes}b_{h,s}\hat{\otimes}a_{h,s}.\label{cox1.4}
    \end{eqnarray}
    Computing $(\ref{cox1.1})-(\ref{cox1.2})-(\ref{cox1.3})+(\ref{cox1.4})$, we get
    \begin{equation}
    \begin{split}
        &(\id_{A_h}\hat{\otimes}(\Delta_h-\sigma_{\bf K}\Delta_h))(\Delta_h(x)-\sigma_{\bf K}\Delta_h(x))=\sum_{s}a_{h,s}\hat{\otimes}b_{h,s}\hat{\otimes}c_{h,s}\\
    &-\sum_{s}a_{h,s}\hat{\otimes}c_{h,s}\hat{\otimes}b_{h,s}-\sum_{s}c_{h,s}\hat{\otimes}a_{h,s}\hat{\otimes}b_{h,s}+\sum_{s}c_{h,s}\hat{\otimes}b_{h,s}\hat{\otimes}a_{h,s}.
    \end{split}
\label{cox3.1}
    \end{equation}
    Define a $\bf K$-linear map $\tau_{\bf K}:A_h\hat{\otimes} A_h\hat{\otimes} A_h\to A_h\ho A_h\ho A_h$ by
    \[\tau_{\bf K}(x_h\hat{\otimes} y_h\hat{\otimes} z_h):=z_h\hat{\otimes} x_h\hat{\otimes} y_h,\;\forall x_h,y_h,z_h\in A_h.\]
    By applying $\tau_{\bf K}$ to (\ref{cox3.1}), we obtain
    \begin{equation}
    \begin{split}
        &\tau_{\bf K}((\id_{A_h}\hat{\otimes}(\Delta_h-\sigma_{\bf K}\Delta_h))(\Delta_h(x)-\sigma_{\bf K}\Delta_h(x)))=\sum_{s}c_{h,s}\hat{\otimes}a_{h,s}\hat{\otimes}b_{h,s}\\
    &-\sum_{s}b_{h,s}\hat{\otimes}a_{h,s}\hat{\otimes}c_{h,s}-\sum_{s}b_{h,s}\hat{\otimes}c_{h,s}\hat{\otimes}a_{h,s}+\sum_{s}a_{h,s}\hat{\otimes}c_{h,s}\hat{\otimes}b_{h,s}.
    \end{split}
    \label{cox3.2}
    \end{equation}
    By applying $\tau_{\bf K}$ to (\ref{cox3.2}), we have
    \begin{equation}
    \begin{split}
    &\tau_{\bf K}^2((\id_{A_h}\hat{\otimes}(\Delta_h-\sigma_{\bf K}\Delta_h))(\Delta_h(x)-\sigma_{\bf K}\Delta_h(x))=\sum_{s}b_{h,s}\hat{\otimes}c_{h,s}\hat{\otimes}a_{h,s}\\
&-\sum_{s}c_{h,s}\hat{\otimes}b_{h,s}\hat{\otimes}a_{h,s}-\sum_{s}a_{h,s}\hat{\otimes}b_{h,s}\hat{\otimes}c_{h,s}+\sum_{s}b_{h,s}\hat{\otimes}a_{h,s}\hat{\otimes}c_{h,s}.
    \end{split}
\label{cox3.3}
    \end{equation}
    Adding (\ref{cox3.1}), (\ref{cox3.2}) and (\ref{cox3.3}) together yields
    \begin{eqnarray}
    (\id_{A_h\ho A_h\ho A_h}+\tau_{\bf K}+\tau_{\bf K}^2)(\id_{A_h}\hat{\otimes}(\Delta_h-\sigma_{\bf K}\Delta_h))(\Delta_h(x)-\sigma_{\bf K}\Delta_h(x))=0.\label{cox4}
    \end{eqnarray}
    Taking the second order terms of $h$ in (\ref{cox4}), we conclude that $(A,\delta)$ is a Lie coalgebra.

    Note that
    \begin{align}
    (\id_{A_h}\hat{\otimes}\Delta_h)(\Delta_h(x)-\sigma_{\bf K}\Delta_h(x))
    &\overset{(\ref{coa})}{=}(\Delta_h\ho\id_{A_h})\Delta_h(x)-(\id_{A_h}\hat{\otimes}\Delta_h)\sigma_{\bf K}\Delta_h(x)\nonumber\\
    &\overset{\hphantom{(\ref{coa})}}{=}(\Delta_h\ho\id_{A_h})\Delta_h(x)-(\sigma_{\bf K}\Delta_h\ho\id_{A_h})\Delta_h(x)\nonumber\\
    &\qquad+(\sigma_{\bf K}\Delta_h\ho\id_{A_h})\Delta_h(x)-(\id_{A_h}\hat{\otimes}\Delta_h)\sigma_{\bf K}\Delta_h(x)\nonumber\\
    &\overset{\hphantom{(\ref{coa})}}{=}(\Delta_h\ho\id_{A_h})\Delta_h(x)-(\sigma_{\bf K}\Delta_h\ho\id_{A_h})\Delta_h(x)\nonumber\\
    &\qquad+(\sigma_{\bf K}\ho\id_{A_h})(\Delta_h\ho\id_{A_h})\Delta_h(x)-(\id_{A_h}\hat{\otimes}\Delta_h)\sigma_{\bf K}\Delta_h(x)\nonumber\\
    &\overset{(\ref{coa})}{=}(\Delta_h\ho\id_{A_h})\Delta_h(x)-(\sigma_{\bf K}\Delta_h\ho\id_{A_h})\Delta_h(x)\nonumber\\
    &\qquad+(\sigma_{\bf K}\ho\id_{A_h})(\id_{A_h}\ho\Delta_h)\Delta_h(x)-(\id_{A_h}\hat{\otimes}\Delta_h)\sigma_{\bf K}\Delta_h(x).\label{dpcob1}
    \end{align}
    Comparing (\ref{cox1.2}) and (\ref{cox1.3}), we show that
    \begin{equation}
    (\id_{A_h}\hat{\otimes}\Delta_h)\sigma_{\bf K}\Delta_h(x)=(\sigma_{\bf K}\ho\id_{A_h})(\id_{A_h}\ho\sigma_{\bf K}\Delta_h)\Delta_h(x).\label{dpcob2}
    \end{equation}
    Inserting (\ref{dpcob2}) to (\ref{dpcob1}), we obtain
    \begin{equation}
    \begin{split}
    (\id_{A_h}\hat{\otimes}\Delta_h)(\Delta_h(x)-\sigma_{\bf K}\Delta_h(x))=&
    ((\Delta_h-\sigma_{\bf K}\Delta_h)\ho\id_{A_h})\Delta_h(x)\\
    &+(\sigma_{\bf K}\ho\id_{A_h})(\id_{A_h}\ho(\Delta_h-\sigma_{\bf K}\Delta_h))\Delta_h(x).
    \end{split}
    \label{dpcob}
    \end{equation}
    So (\ref{pco}) follows from taking the first order terms of $h$ in
    (\ref{dpcob}). Thus $(A, \delta, \Delta)$ is a Poisson coalgebra.
\end{proof}

\subsection{Antisymmetric infinitesimal bialgebra deformations and the quasiclassical limits}
\begin{defi}
    {\rm (\cite{Bai})}
    Let $(A, \circ)$ be an associative algebra and $(A, \Delta)$ be a coassociative coalgebra. Then the triple $(A, \circ, \Delta)$ is called an \textbf{ antisymmetric infinitesimal (ASI) bialgebra}
    if  for all $x,y\in A$,
    \begin{eqnarray}
&\Delta (x \circ y)=(\mathrm{id}_{A}\otimes L_{\circ}(x)) \Delta(y)+(R_{\circ}(y) \otimes \mathrm{id}_{A}) \Delta(x),&\label{asi1}\\
&(L_{\circ}(y) \otimes \mathrm{id}_{A}- \mathrm{id}_{A}\otimes
R_{\circ}(y)) \Delta(x)+\sigma((L_{\circ}(x) \otimes
\mathrm{id}_{A}- \mathrm{id}_{A} \otimes R_{\circ}(x))
\Delta(y))=0.&\label{asi2}
\end{eqnarray}
\end{defi}

\begin{defi}\label{defipb}
    {\rm (\cite{NB})}
     Let $(A,\{,\}, \circ)$ be a Poisson algebra, $(A,\{,\}, \delta)$ be a Lie bialgebra, $(A,\delta,\Delta)$ be a Poisson coalgebra, and $(A, \circ, \Delta)$ be a commutative and cocommutative ASI bialgebra.
     Then the quintuple $(A,\{,\}, \circ, \delta, \Delta)$ is called a \textbf{Poisson bialgebra}
     if for all $x,y\in A$,
     \begin{eqnarray}
    \delta(x \circ y)=(L_{\circ}(y) \otimes \mathrm{id}_{A}) \delta(x)+(L_{\circ}(x) \otimes \mathrm{id}_{A}) \delta(y)
        +(\mathrm{id}_{A} \otimes \ad_{\{,\}}(x)) \Delta(y)+(\mathrm{id}_{A} \otimes \ad_{\{,\}}(y)) \Delta(x),
     \label{pb1}    \end{eqnarray}
\begin{eqnarray}
        \Delta(\{x,y\})=(\ad_{\{,\}}(x)\ot\mathrm{id}_{A}+\mathrm{id}_{A} \otimes \ad_{\{,\}}(x)) \Delta(y)
     +(L_{\circ}(x) \otimes \mathrm{id}_{A}-\mathrm{id}_{A} \otimes L_{\circ}(x))
     \delta(y).
        \label{pb2}
     \end{eqnarray}
\end{defi}

\begin{rmk}
    Note that there is also a notion of Poisson bialgebras in \cite{KT}, in which the
    comultiplications are homomorphisms of Poisson algebras.
    Obviously, it is different from Definition~\ref{defipb}.
\end{rmk}

Next, we introduce the notions about ASI bialgebra deformations.
\begin{defi}
    Let $(A_{h},\circ_{h})$ be a topologically free associative algebra and $(A_h,\Delta_{h})$ be a topologically free coassociative coalgebra. Then we call the triple $(A_{h},\circ_{h},\Delta_{h})$ a \textbf{topologically free ASI bialgebra} if for all
    $x_h,y_h\in A_h$,
    \begin{eqnarray}
&\Delta_h (x_h \circ_h y_h)=(\mathrm{id}_{A_h}\ho L_{\circ_h}(x_h)) \Delta_h(y_h)+(R_{\circ_h}(y_h) \ho \mathrm{id}_{A_h}) \Delta_h(x_h),&\label{tib1}
    \\
   & (L_{\circ_h}(y_h) \ho \mathrm{id}_{A_h}- \mathrm{id}_{A_h}\otimes R_{\circ_h}(y_h)) \Delta_h(x)+\sigma_{\bf K}((L_{\circ_h}(x_h) \ho\mathrm{id}_{A_h}- \mathrm{id}_{A_h} \ho R_{\circ_h}(x_h)) \Delta_h(y_h))=0.&\label{tasib1}
    \end{eqnarray}
Let $(A,\circ,\Delta)$ be an ASI bialgebra and
$(A_{h},\circ_{h},\Delta_{h})$ be a topologically free ASI
bialgebra. If $(A_{h},\circ_{h})$ and $(A_{h},\Delta_{h})$ are an
associative deformation and a coassociative deformation of
$(A,\circ)$ and $(A,\Delta)$, respectively, then we call
$(A_{h},\circ_{h},\Delta_{h})$ an \textbf{ASI bialgebra
deformation} of $(A,\circ,\Delta)$. Moreover, if
$(A,\circ,\Delta)$ is a commutative and cocommutative ASI
bialgebra, $(A,\{,\},\circ)$ and $(A,\delta,\Delta)$ are QCLs of
$(A_{h},\circ_{h})$ and $(A_{h},\Delta_{h})$, respectively, then
the quintuple $(A,\{,\}, \circ,\delta,\Delta)$ is called the
\textbf{quasiclassical limit (QCL) of the ASI bialgebra
deformation} $(A_{h},\circ_{h},\Delta_{h})$.
\end{defi}
\begin{thm}\label{thm-limasi}
    Let $(A_{h},\circ_{h},\Delta_{h})$ be an ASI bialgebra deformation of a  commutative and cocommutative ASI bialgebra $(A,\circ,\Delta)$ and $(A,\{,\}, \circ,\delta,\Delta)$ be the QCL. Then $(A,\{,\}, \circ,\delta,\Delta)$ is a Poisson bialgebra.
\end{thm}
\begin{proof}
    By Theorems~\ref{thacl}~and~\ref{qclco}, $(A,\{,\}, \circ)$ is a Poisson algebra and $(A,\delta, \Delta)$ is a Poisson coalgebra.
Let $x,y\in A$. By (\ref{tib1}) and (\ref{tasib1}), we have
\begin{eqnarray}
\Delta_h (x \circ_h y)=(\mathrm{id}_{A_h}\ho L_{\circ_h}(x))
\Delta_h(y)+(R_{\circ_h}(y) \ho \mathrm{id}_{A_h})
\Delta_h(x),\label{tib}\end{eqnarray}
\begin{eqnarray}
0=(L_{\circ_h}(y) \ho \mathrm{id}_{A_h}- \mathrm{id}_{A_h}\otimes
R_{\circ_h}(y)) \Delta_h(x)+\sigma_{\bf K}((L_{\circ_h}(x)
\ho\mathrm{id}_{A_h}- \mathrm{id}_{A_h} \ho R_{\circ_h}(x))
\Delta_h(y)).\label{tasib}
\end{eqnarray}
Applying $\sigma_{\bf K}$ to (\ref{tib}), we obtain
\begin{equation}
\sigma_{\bf K}(\Delta_h (x \circ_h y))=(L_{\circ_h}(x) \ho \mathrm{id}_{A_h}) \sigma_{\bf K}(\Delta_h(y))+(\mathrm{id}_{A_h}\ho R_{\circ_h}(y)) \sigma_{\bf K}(\Delta_h(x)).\label{ttib}
\end{equation}
Subtracting (\ref{ttib}) from (\ref{tib}), we have
\begin{equation}
\begin{split}
\Delta_h (x \circ_h y)-\sigma_{\bf K}(\Delta_h (x \circ_h y))&=(\mathrm{id}_{A_h}\ho L_{\circ_h}(x)) \Delta_h(y)-(L_{\circ_h}(x) \ho \mathrm{id}_{A_h}) \sigma_{\bf K}(\Delta_h(y))\\
&\quad+(R_{\circ_h}(y) \ho \mathrm{id}_{A_h}) \Delta_h(x)-(\mathrm{id}_{A_h}\ho R_{\circ_h}(y)) \sigma_{\bf K}(\Delta_h(x)).
\end{split}
\label{g1}
\end{equation}
Subtracting (\ref{tasib}) from (\ref{g1}), we have
\begin{eqnarray}
\begin{split}
&\Delta_h (x \circ_h y)-\sigma_{\bf K}(\Delta_h (x \circ_h y))\\&=(\mathrm{id}_{A_h}\ho L_{\circ_h}(x)) (\Delta_h(y)-\sigma_{\bf K}(\Delta_h(y))-(L_{\circ_h}(x) \ho \mathrm{id}_{A_h}-R_{\circ_h}(x) \ho \mathrm{id}_{A_h}) \sigma_{\bf K}(\Delta_h(y))\\
&\quad+(R_{\circ_h}(y) \ho \mathrm{id}_{A_h}-L_{\circ_h}(y) \ho \mathrm{id}_{A_h}) \Delta_h(x)+(\mathrm{id}_{A_h}\ho R_{\circ_h}(y)) (\Delta_h(x)-\sigma_{\bf K}(\Delta_h(x))).
\end{split}
\label{g2}
\end{eqnarray}
Taking the first order terms of $h$ in (\ref{g2}), we obtain (\ref{pb1}).

Exchanging $x,y$ in (\ref{tib}), we have
\begin{equation}
        \Delta_h (y \circ_h x)=(\mathrm{id}_{A_h}\ho L_{\circ_h}(y)) \Delta_h(x)+(R_{\circ_h}(x) \ho \mathrm{id}_{A_h}) \Delta_h(y),\label{htib}
\end{equation}
By the difference between (\ref{tib}) and (\ref{htib}), we show that
\begin{eqnarray}
\Delta_h (x \circ_h y-y\circ_h x)=(\mathrm{id}_{A_h}\ho L_{\circ_h}(x)-R_{\circ_h}(x) \ho \mathrm{id}_{A_h})\Delta_h(y)+(R_{\circ_h}(y) \ho \mathrm{id}_{A_h}-\mathrm{id}_{A_h}\ho L_{\circ_h}(y)) \Delta_h(x).
\label{g3}
\end{eqnarray}
By the difference between (\ref{g3}) and (\ref{tasib}), we get
\begin{equation}
\begin{split}
\Delta_h (x \circ_h y-y\circ_h x)=&(\mathrm{id}_{A_h}\ho L_{\circ_h}(x)-R_{\circ_h}(x) \ho \mathrm{id}_{A_h}) (\Delta_h(y)-\sigma_{\bf K}(\Delta_h(y)))\\
&-(L_{\circ_h}(y) \ho \mathrm{id}_{A_h}-R_{\circ_h}(y) \ho \mathrm{id}_{A_h}-\mathrm{id}_{A_h}\ho L_{\circ_h}(y)+\mathrm{id}_{A_h}\ho R_{\circ_h}(y)) \Delta_h(x).
\end{split}
\label{g4}
\end{equation}
Taking the first order terms of $h$ in (\ref{g4}), we deduce that
\begin{equation}
\Delta(\{x, y\})=(\mathrm{id}_{A} \otimes L_{\circ}(x)-L_{\circ}(x) \otimes \mathrm{id}_{A}) \delta(y)-(\ad_{\{,\}}(y) \otimes \mathrm{id}_{A}+\mathrm{id}_{A} \otimes \ad_{\{,\}}(y)) \Delta(x).
\label{g5}
\end{equation}
Exchanging $x,y$ in ($\ref{g5}$) yields (\ref{pb2}).

Applying $\sigma_{\bf K}$ to ($\ref{g3}$), we have
\begin{equation}
    \begin{split}
    \sigma_{\bf K}(\Delta_h (x \circ_h y-y\circ_h x))&=
    (L_{\circ_h}(x) \ho \mathrm{id}_{A_h})\sigma_{\bf K}( \Delta_h(y))-(\mathrm{id}_{A_h}\ho R_{\circ_h}(x))\sigma_{\bf K}( \Delta_h(y))\\
    &\quad+(\mathrm{id}_{A_h}\ho R_{\circ_h}(y))\sigma_{\bf K}( \Delta_h(x))-(R_{\circ_h}(y) \ho \mathrm{id}_{A_h})\sigma_{\bf K}( \Delta_h(x)).
    \end{split}
    \label{g6}
\end{equation}
Interchanging $x,y$ in ($\ref{tasib}$), we have
\begin{equation}
        0=(L_{\circ_h}(x) \ho \mathrm{id}_{A_h}- \mathrm{id}_{A_h}\otimes R_{\circ_h}(y)) \Delta_h(y)+\sigma_{\bf K}((L_{\circ_h}(y) \ho\mathrm{id}_{A_h}- \mathrm{id}_{A_h} \ho R_{\circ_h}(y)) \Delta_h(x)).
    \label{htasib}
\end{equation}
Computing $(\ref{g3})-(\ref{g6})+(\ref{htasib})-(\ref{tasib})$, we obtain
\begin{equation}
\begin{split}
    &(\Delta_h-\sigma_{\bf K}(\Delta_h)) (x \circ_h y-y\circ_h x)\\&=
((L_{\circ_h}(x)-R_{\circ_h}(x)) \ho \mathrm{id}_{A_h}+\mathrm{id}_{A_h}\ho (L_{\circ_h}(x)-R_{\circ_h}(x)))\Delta_h(y)\\
&\quad-((L_{\circ_h}(x)-R_{\circ_h}(x)) \ho \mathrm{id}_{A_h}+\mathrm{id}_{A_h}\ho (L_{\circ_h}(x)-R_{\circ_h}(x)))\sigma_{\bf K}(\Delta_h(y))\\
&\quad+((L_{\circ_h}(y)-R_{\circ_h}(y)) \ho \mathrm{id}_{A_h}+\mathrm{id}_{A_h}\ho (L_{\circ_h}(y)-R_{\circ_h}(y)))\sigma_{\bf K}(\Delta_h(x))\\
&\quad-((L_{\circ_h}(y)-R_{\circ_h}(y)) \ho \mathrm{id}_{A_h}+\mathrm{id}_{A_h}\ho (L_{\circ_h}(y)-R_{\circ_h}(y)))\Delta_h(x).
\end{split}\label{g7}
\end{equation}
We show that $(A,\{,\},\delta)$ is a Lie bialgebra by taking the second order terms of $h$ in (\ref{g7}). Finally, we conclude that $(A,\{,\}, \circ,\delta,\Delta)$ is a Poisson bialgebra.
\end{proof}
    \subsection{Antisymmetric infinitesimal bialgebra deformations via coherent derivations}\label{sec2.3}
\begin{lem}{\rm{(\mcite{G3})}}\label{Dda}
Let $d_1$ and $d_2$ be commuting derivations on an associative
algebra $(A,\circ)$. Set
\begin{equation}\label{UDF}
    x\circ_{h}y:=\sum_{s=0}^{\infty}(d_1^{s}(x)\circ d_2^{s}(y))\frac{h^{s}}{s!}, \quad \forall x, y\in A,
  \end{equation}
and extend it to $A_{h}$ by $\mathbf{K}$-bilinearity. Then $(A_h,\circ_h)$ is
an associative deformation of $(A,\circ)$, called the \textbf{associative deformation induced by} $(A,\circ, d_1,d_2)$. \mlabel{lemG}
\end{lem}

By Lemma~\ref{Dda}~and Theorem~\ref{thacl}, we obtain the following result.
\begin{cor}\label{daop}
    Let $d_1,d_2$ be commuting derivations on a commutative associative
algebra $(A,\circ)$ and $(A_h,\circ_h)$ be the associative deformation induced by $(A,\circ, d_1,d_2)$.
    Define a binary operation $\{,\}$ on $A$ by
    \begin{equation}
    \{x,y\}:= d_1(x)\circ d_2(y)-d_2(x)\circ d_1(y),\;\;
    \forall x, y \in A. \label{eq-dp}
    \end{equation}
    Then $(A,\{,\},\circ)$ coincides with the QCL of $(A_h,\circ_h)$. Moreover, $(A,\{,\},\circ)$ is a Poisson algebra, called the \textbf{Poisson algebra induced by} $(A, \circ, d_1,d_2)$.
\end{cor}

Next, we give a similar construction for a coassociative deformation. Let $(A,\Delta)$ be a coassociative coalgebra. Then a linear map $\dc:A\to A$ satisfying \[\Delta\dc=(\dc\otimes\id_{A})\Delta+(\id_{A}\otimes\dc)\Delta\]
is called a \textbf{coderivation} on $(A,\Delta)$.
It is well-known that a derivation $d$ on an associative algebra $(A,\circ)$ satisfies
\begin{equation}
    d^m(x\circ y)=\sum_{j=0}^{m}\binom{m}{j}d^{j}(x)\circ d^{m-j}(y),\;\;\forall x,y\in A, m\in\mathbb{N}.\label{leb}
\end{equation}
Similarly, the following equation for a coderivation $\dc$ on a
coassociative coalgebra $(A,\Delta)$ can be obtained by
 induction.
\begin{equation}
\Delta\dc^{m}=\sum_{j=0}^{m}\binom{m}{j}(\dc^{j}\otimes\dc^{m-j})\Delta,\;\;\forall m\in\mathbb{N}.\label{coleb}
\end{equation}
\begin{pro}\label{DFC}
        Let $\dc_1$ and $\dc_2$ be commuting coderivations on a coassociative coalgebra $(A,\Delta)$. Set
    \begin{equation}\label{UDFC}
    \Delta_h(x):=\sum_{s=0}^{\infty}\dc_1^{s}(x_{(1)})\hat{\ot}\dc^{s}_{2}(x_{(2)})\frac{h^s}{s!},\;\;\forall x\in A,
    \end{equation}
    where we use the Sweedler notation $\Delta(x) = x_{(1)} \otimes x_{(2)}$,
    and extend it to $A_{h}$ by $\mathbf{K}$-linearity. Then $(A_h,\Delta_h)$ is a coassociative deformation of $(A,\Delta)$, called the \textbf{coassociative deformation induced by} $(A,\Delta,\dc_1,\dc_2)$.
\end{pro}
\begin{proof}
    Let $x\in A$. Then we have
    \begin{align}
            &\sum_{k=0}^{\infty}\sum_{s=0}^{\infty}((\dc_1^{k}\ot \dc^{k}_{2})\Delta\otimes\id_{A})(\dc_1^{s}\ot \dc^{s}_{2})\Delta(x)\frac{h^{k+s}}{k!s!}\nonumber\\
                &\overset{\hphantom{(\ref{coa})}}{=}\sum_{k=0}^{\infty}\sum_{s=0}^{\infty}((\dc_1^{k}\ot \dc^{k}_{2})\Delta\dc_1^{s}\ot \dc^{s}_{2})\Delta(x)\frac{h^{k+s}}{k!s!}\nonumber\\
            &\overset{(\ref{coleb})}{=}\sum_{k=0}^{\infty}\sum_{s=0}^{\infty}\sum_{j=0}^{s}((\dc_1^{k+j}\ot \dc^{k}_{2}\dc_1^{s-j})\Delta\ot \dc^{s}_{2})\Delta(x)\binom{s}{j}\frac{h^{k+s}}{k!s!}\nonumber\\
                &\overset{\hphantom{(\ref{coa})}}{=}\sum_{k=0}^{\infty}\sum_{s=0}^{\infty}\sum_{j=0}^{s}(\dc_1^{k+j}\ot \dc^{k}_{2}\dc_1^{s-j}\ot \dc^{s}_{2})(\Delta\otimes\id_{A})\Delta(x)\frac{h^{k+s}}{j!k!(s-j)!}.\label{css1}
    \end{align}
    Similarly, we obtain
        \begin{eqnarray}
    &&\sum_{k=0}^{\infty}\sum_{s=0}^{\infty}(\id_{A}\ot(\dc_1^{k}\ot \dc^{k}_{2})\Delta)(\dc_1^{s}\ot \dc^{s}_{2})\Delta(x)\frac{h^{k+s}}{k!s!}\nonumber\\
    &&=\sum_{k=0}^{\infty}\sum_{s=0}^{\infty}\sum_{j=0}^{s}(\dc_1^{s}\ot \dc_1^{k}\dc^{j}_{2}\ot \dc^{s-j}_{2})(\id_{A}\otimes\Delta)\Delta(x)\frac{h^{k+s}}{j!k!(s-j)!}.\label{cs2}
    \end{eqnarray}
    Comparing the summation indices of (\ref{css1}) and (\ref{cs2}), and then exploiting the commutativity of $\dc_{1},\dc_{2}$ and the coassociativity of $\Delta$, we see that
        \begin{equation}
        \begin{split}
            &\sum_{k=0}^{\infty}\sum_{s=0}^{\infty}((\dc_1^{k}\ot \dc^{k}_{2})\Delta\otimes\id_{A})(\dc_1^{s}\ot \dc^{s}_{2})\Delta(x)\frac{h^{k+s}}{k!s!}\\
        &=\sum_{k=0}^{\infty}\sum_{s=0}^{\infty}(\id_{A}\ot(\dc_1^{k}\ot \dc^{k}_{2})\Delta)(\dc_1^{s}\ot \dc^{s}_{2})\Delta(x)\frac{h^{k+s}}{k!s!},\;\;\forall x\in A.
        \end{split}\label{dcce}
    \end{equation}
    Embedding (\ref{dcce}) into $A_h\hat{\ot}A_h\hat{\ot}A_h$
    via the $\bfK$-linear isomorphism $(A\ot A\ot A)[[h]]\cong A_h\hat{\ot}A_h\hat{\ot}A_h$, we have (\ref{coa}). Thus we conclude that
$(A_h,\Delta_h)$ is a topologically free coassociative coalgebra.
\end{proof}
The following result follows from Proposition~\ref{DFC}~and Theorem~\ref{qclco}.
\begin{cor}
    Let $\dc_1$ and $\dc_2$ be commuting coderivations on a cocommutative coassociative coalgebra $(A,\Delta)$ and $(A_h,\Delta_h)$ be the coassociative deformation  induced by $(A,\Delta,\dc_1,\dc_2)$.
   Define a linear map $\delta:A\to A\ot A$ by
    \begin{equation}\label{UDFCA}
    \delta(x):=(\dc_1\ot\dc_{2}-\dc_2\ot\dc_{1})\Delta(x),\;\;\forall x\in A.
    \end{equation}
    Then $(A,\delta,\Delta)$ coincides with the QCL of $(A_h,\Delta_h)$.
    Moreover, $(A,\delta,\Delta)$ is a Poisson coalgebra.
\end{cor}

\begin{defi}{\rm{(\cite{Lin})}}
    Let $(A, \circ, \Delta)$ be an ASI bialgebra,
    $d$ be a derivation on $(A, \circ)$, and $\dc$ be a coderivation on $(A, \Delta)$. Then the pair $(d,\dc)$ is called a \textbf{coherent derivation}
    on the ASI bialgebra
    $(A, \circ, \Delta)$ if for all $x,y\in A$,
    \begin{eqnarray*}
    &&\dc(x)\circ y=x\circ d(y)+\dc(x\circ y),\;\;x\circ\dc(y)=d(x)\circ y+\dc(x\circ y),\\
    &&(d\otimes\id_{A})\Delta=(\id_{A}\otimes\dc)\Delta+\Delta d,\;\;(\id_{A}\otimes\dc)\Delta=(\dc\otimes\id_{A})\Delta+\Delta d.
    \end{eqnarray*}
\end{defi}
The following result can be proved by induction on $m$.
\begin{pro}
    Let $(d,\dc)$ be a coherent derivation on an ASI bialgebra $(A, \circ, \Delta)$. Then for all $m\in\mathbb{N}, x,y\in A$, the following
    equations hold.
        \begin{eqnarray}
        \dc^{m}(x)\circ y&=&\sum_{j=0}^{m}\binom{m}{j}\dc^{j}(x\circ d^{m-j}(y)),\label{px1}\\
        x\circ\dc^{m}(y)&=&\sum_{j=0}^{m}\binom{m}{j}\dc^{j}(d^{m-j}(x)\circ y),\label{px2}\\
        (d^{m}\otimes\id_{A})\Delta&=&\sum_{j=0}^{m}\binom{m}{j}(\id_{A}\otimes\dc^{j})\Delta d^{m-j},\label{px3}\\
        (\id_{A}\otimes\dc^{m})\Delta&=&\sum_{j=0}^{m}\binom{m}{j}(\dc^{j}\otimes\id_{A})\Delta d^{m-j}\label{px4}.
    \end{eqnarray}
\end{pro}

It is straightforward to obtain the following result.
\begin{lem}
    Let $A$ be a vector space, $\circ$ be a binary operation on $A$, and $\Delta:A\to A\ot A$, $d_1,d_2,\dc_1,\dc_2:A\to A$ be linear maps. Assume that $d_1,d_2$ commute and $\dc_1,\dc_2$ also commute.
    \begin{enumerate}
        \item If the following equation holds,
        \begin{equation}
            d_{1}(x)\circ d_{2}(y)=\dc_{2}(x\circ d_1(y)),\;\;\forall x,y\in A, \label{ed1}
        \end{equation}
        then we have
        \begin{equation}
            d_{1}^{m}(x)\circ d_{2}^{m}(y)=\dc_{2}^{m}(x\circ
            d_1^{m}(y)),\;\;\forall m\in\mathbb{N},\;x,y\in A.
            \label{ped1}
        \end{equation}
            \item If the following equation holds,
        \begin{equation}
    d_{1}(x)\circ d_{2}(y)=\dc_1(d_2(x)\circ y),\;\;\forall x,y\in A,\label{ed2}
        \end{equation}
    then we have
        \begin{equation}
    d_{1}^{m}(x)\circ d_{2}^{m}(y)=\dc_{1}^{m}(d_{2}^{m}(x)\circ y),\;\;\forall m\in\mathbb{N},\;x,y\in A.\label{ped2}
        \end{equation}
            \item If the following equation holds,
        \begin{equation}
    (\dc_1\otimes\dc_2)\Delta=(\id_{A}\otimes\dc_1)\Delta d_2,\label{ed3}
        \end{equation}
        then we have
        \begin{equation}
    (\dc_1^{m}\otimes\dc_2^{m})\Delta=(\id_{A}\otimes\dc_1^{m})\Delta d_2^{m},\;\;\forall m\in\mathbb{N}.\label{ped3}
        \end{equation}
        \item If the following equation holds,
        \begin{equation}
    (\dc_1\otimes\dc_2)\Delta=(\dc_2\otimes\id_{A})\Delta d_1,\label{ed4}
        \end{equation}
        then we have
        \begin{equation}
    (\dc_1^{m}\otimes\dc_2^{m})\Delta=(\dc_2^{m}\otimes\id_{A})\Delta d_1^{m},\;\;\forall m\in\mathbb{N}.\label{ped4}
        \end{equation}
    \end{enumerate}
\end{lem}

\begin{thm}
    \label{ddasi}
Let $(d_1,\dc_1)$ and $(d_2,\dc_2)$ be coherent derivations on an
ASI bialgebra $(A, \circ, \Delta)$. Suppose that $d_1,d_2$ commute
and $\dc_1,\dc_2$ also commute. Let $(A_h,\circ_h)$ be the
associative deformation induced by $(A,\circ, d_1,d_2)$ and
$(A_h,\Delta_h)$ be the coassociative deformation induced by
$(A,\Delta,\dc_1,\dc_2)$. If {\rm(\ref{ed1})}, {\rm(\ref{ed2})},
{\rm(\ref{ed3})} and {\rm(\ref{ed4})} hold, then $(A_h, \circ_h,
\Delta_h)$ is an ASI bialgebra deformation of $(A, \circ,
\Delta)$, called the \textbf{ASI bialgebra deformation induced by}
$(A,\circ, d_1,d_2,\Delta,\dc_1,\dc_2)$.
\end{thm}
\begin{proof}
       Let $x,y\in A$. Then we have
    \begin{align}
      &\sum_{s=0}^{\infty}\sum_{m=0}^{s}\frac{\left(\dc_{1}^{m}\otimes\dc_{2}^{m}\right)\Delta\left(d_1^{s-m}\left(x\right)\circ d_2^{s-m}\left(y\right)\right)}{m!\left(s-m\right)!}h^s\nonumber\\
      &\overset{\left(\ref{ped3}\right)}{=}  \sum_{s=0}^{\infty}\sum_{m=0}^{s}\frac{\left(\id_A\otimes\dc_{1}^{m}\right)\Delta d_2^{m}\left(d_1^{s-m}\left(x\right)\circ d_2^{s-m}\left(y\right)\right)}{m!\left(s-m\right)!}h^s\nonumber\\
      &\overset{\left(\ref{leb}\right)}{=}\sum_{s=0}^{\infty}\sum_{m=0}^{s}\sum_{j=0}^{m}\frac{\binom{m}{j}\left(\id_A\otimes\dc_{1}^{m}\right)\Delta \left(d_2^{j}d_1^{s-m}\left(x\right)\circ d_2^{s-j}\left(y\right)\right)}{m!\left(s-m\right)!}h^s\nonumber\\
      &\overset{\left(\ref{asi1}\right)}{=}
      \sum_{s=0}^{\infty}\sum_{m=0}^{s}\sum_{j=0}^{m}
      \frac{\left(R_{\circ}\left(d_2^{s-j}\left(y\right)\right)\otimes\dc_{1}^{m}\right)\Delta \left(d_2^{j}d_1^{s-m}\left(x\right)\right)}{j!\left(m-j\right)!\left(s-m\right)!}h^s\label{guoduib1}\\
      &\qquad+        \sum_{s=0}^{\infty}\sum_{m=0}^{s}\sum_{j=0}^{m}\frac{\left(\id_{A}\ot\dc_1^{m}L_{\circ}\left(d_2^{j}d_1^{s-m}\left(x\right)\right)\right)\Delta \left(d_2^{s-j}\left(y\right)\right)}{j!\left(m-j\right)!\left(s-m\right)!}h^s\label{guoduib2},\\
      &\sum_{s=0}^{\infty}\sum_{m=0}^{s}\frac{\left(R_{\circ}\left(d_2^{s-m}\left(y\right)\right)d_1^{s-m}\otimes\id_A\right)\left(\dc_{1}^{m}\otimes\dc_{2}^{m}\right)\Delta\left(x\right)}{m!\left(s-m\right)!}h^s\nonumber\\
      &\overset{\left(\ref{ped3}\right)}{=}  \sum_{s=0}^{\infty}\sum_{m=0}^{s}\frac{\left(R_{\circ}\left(d_2^{s-m}\left(y\right)\right)d_1^{s-m}\otimes\dc_{1}^{m}\right)\Delta d_2^{m}\left(x\right)}{m!\left(s-m\right)!}h^s\nonumber\\
      &\overset{\hphantom{\left(\ref{px3}\right)}}{=}\sum_{s=0}^{\infty}\sum_{m=0}^{s}\frac{\left(R_{\circ}\left(d_2^{s-m}\left(y\right)\right)\otimes\dc_{1}^{m}\right)\left(d_1^{s-m}\otimes\id_{A}\right)\Delta d_2^{m}\left(x\right)}{m!\left(s-m\right)!}h^s\nonumber\\
      &\overset{\left(\ref{px3}\right)}{=}
      \sum_{s=0}^{\infty}\sum_{m=0}^{s}\sum_{j=0}^{m}
      \frac{\binom{s-m}{j}\left(R_{\circ}\left(d_2^{s-j}\left(y\right)\right)\otimes\dc_{1}^{m+j}\right)\Delta \left(d_1^{s-m-j}d_2^{m}\left(x\right)\right)}{\left(m-j\right)!\left(s-m\right)!}h^s\nonumber\\
      &\overset{\hphantom{\left(\ref{px3}\right)}}{=}    \sum_{s=0}^{\infty}\sum_{m=0}^{s}\sum_{j=0}^{m}
      \frac{\left(R_{\circ}\left(d_2^{s-j}\left(y\right)\right)\otimes\dc_{1}^{m+j}\right)\Delta \left(d_1^{s-m-j}d_2^{m}\left(x\right)\right)}{j!\left(s-m-j\right)!\left(m-j\right)!}h^s.\label{guoduib3}
    \end{align}
    On the other hand, we have
        \begin{align}
       &\sum_{s=0}^{\infty}\sum_{m=0}^{s}\frac{\left(\id_{A}\ot L_{\circ}\left(d_1^{s-m}\left(x\right)\right) d_2^{s-m}\right)\left(\dc_{1}^{m}\otimes\dc_{2}^{m}\right)\Delta\left(y\right)}{m!\left(s-m\right)!}h^s\nonumber\\
       &\overset{\left(\ref{ped3}\right)}{=}  \sum_{s=0}^{\infty}\sum_{m=0}^{s}\frac{\left(\id_{A}\ot L_{\circ}\left(d_1^{s-m}\left(x\right)\right)d_2^{s-m}\dc_{1}^{m}\right)\Delta d_2^{m}\left(y\right)}{m!\left(s-m\right)!}h^s\nonumber\\
       &\overset{\left(\ref{ped2}\right)}{=}\sum_{s=0}^{\infty}\sum_{m=0}^{s}\frac{\left(\id_{A}\ot \dc_{1}^{s-m}L_{\circ}\left(d_2^{s-m}\left(x\right)\right)\dc_{1}^{m}\right)\Delta d_2^{m}\left(y\right)}{m!\left(s-m\right)!}h^s\label{guoduib4}\\
       &\overset{\left(\ref{px2}\right)}{=}\sum_{s=0}^{\infty}\sum_{m=0}^{s}\sum_{j=0}^{m}
       \frac{\left(\id_{A}\ot \dc_{1}^{s-m+j}L_{\circ}\left(d_1^{m-j}d_2^{s-m}\left(x\right)\right)\right)\Delta d_2^{m}\left(y\right)}{j!\left(m-j\right)!\left(s-m\right)!}h^s\nonumber\\
       &\qquad+    \sum_{s=0}^{\infty}\sum_{m=0}^{s}\sum_{j=0}^{m}
       \frac{\left(\id_{A}\ot \dc_{1}^{s-m+j}L_{\circ}\left(d_1^{m-j}d_2^{s-m}\left(x\right)\right)\right)\Delta d_2^{m}\left(y\right)}{j!\left(m-j\right)!\left(s-m\right)!}h^s.\nonumber
    \end{align}
    Comparing (\ref{guoduib1}) with (\ref{guoduib3}) and (\ref{guoduib2}) with (\ref{guoduib4}) respectively, by the commutativity of $d_1,d_2$, we obtain 
    \begin{equation*}
   \begin{split}
   \sum_{s=0}^{\infty}\sum_{m=0}^{s}\frac{\left(\dc_{1}^{m}\otimes\dc_{2}^{m}\right)\Delta\left(d_1^{s-m}\left(x\right)\circ d_2^{s-m}\left(y\right)\right)}{m!\left(s-m\right)!}h^s=&\sum_{s=0}^{\infty}\sum_{m=0}^{s}\frac{\left(R_{\circ}\left(d_2^{s-m}\left(y\right)\right)d_1^{s-m}\otimes\id_A\right)\left(\dc_{1}^{m}\otimes\dc_{2}^{m}\right)\Delta\left(x\right)}{m!\left(s-m\right)!}h^s\\
   &+\sum_{s=0}^{\infty}\sum_{m=0}^{s}\frac{\left(\id_{A}\ot L_{\circ}\left(d_1^{s-m}\left(x\right) d_2^{s-m}\right)\right)\left(\dc_{1}^{m}\otimes\dc_{2}^{m}\right)\Delta\left(y\right)}{m!\left(s-m\right)!}h^s.
   \end{split}
    \end{equation*}
    Embedding the above equation into $A_h\hat{\ot}A_h$
    by the $\bfK$-linear isomorphism $(A\ot A)[[h]]\cong A_h\hat{\ot}A_h$
    yields (\ref{tib}).

   Moreover, we have
\begin{align}
&\sum_{s=0}^{\infty}\sum_{m=0}^{s}\frac{\left(L_{\circ}\left(d_1^{m}\left(x\right)\right)d_2^{m}\otimes\id_A\right)\left(\dc_{1}^{s-m}\otimes\dc_{2}^{s-m}\right)\Delta\left(y\right)}{m!\left(s-m\right)!}h^s\nonumber\\
&\overset{\left(\ref{ped2}\right)}{=}  \sum_{s=0}^{\infty}\sum_{m=0}^{s}\frac{\left(\dc_{1}^{m}L_{\circ}\left(d_2^{m}\left(x\right)\right)\dc_{1}^{s-m}\otimes\dc_{2}^{s-m}\right)\Delta\left(y\right)}{m!\left(s-m\right)!}h^s\nonumber\\
&\overset{\left(\ref{px2}\right)}{=}   \sum_{s=0}^{\infty}\sum_{m=0}^{s}\sum_{j=0}^{s-m}\frac{\binom{s-m}{j}\left(\dc_{1}^{m+j}L_{\circ}\left(d_1^{s-m-j}d_2^{m}\left(x\right)\right)\otimes\dc_{2}^{s-m}\right)\Delta\left(y\right)}{m!\left(s-m\right)!}h^s\nonumber\\
&\overset{\hphantom{\left(\ref{px3}\right)}}{=}\sum_{s=0}^{\infty}\sum_{m=0}^{s}\sum_{j=0}^{s-m}\frac{\left(\dc_{1}^{m+j}L_{\circ}\left(d_1^{s-m-j}d_2^{m}\left(x\right)\right)\otimes\dc_{2}^{s-m}\right)\Delta\left(y\right)}{m!j!\left(s-m-j\right)!}h^s,
\label{guoduaib1}
\end{align}
    \begin{align}
&\sum_{s=0}^{\infty}\sum_{m=0}^{s}\frac{\left(\id_A\otimes R_{\circ}\left(d_2^{m}\left(x\right)\right)d_1^{m}\right)\left(\dc_{1}^{s-m}\otimes\dc_{2}^{s-m}\right)\Delta\left(y\right)}{m!\left(s-m\right)!}h^s\nonumber\\
&\overset{\left(\ref{ped1}\right)}{=}  \sum_{s=0}^{\infty}\sum_{m=0}^{s}\frac{\left(\dc_1^{s-m}\otimes \dc_{2}^{m}R_{\circ}\left(d_1^{m}\left(y\right)\right)\dc_{2}^{s-m}\right)\Delta\left(y\right)}{m!\left(s-m\right)!}h^s\nonumber\\
&\overset{\left(\ref{px1}\right)}{=}   \sum_{s=0}^{\infty}\sum_{m=0}^{s}\sum_{j=0}^{s-m}\frac{\binom{s-m}{j}\left(\dc_1^{s-m}\otimes \dc_{2}^{m+j}R_{\circ}\left(d_2^{s-m-j}d_1^{m}\left(y\right)\right)\right)\Delta\left(y\right)}{m!\left(s-m\right)!}h^s\nonumber\\
&\overset{\hphantom{\left(\ref{px3}\right)}}{=}\sum_{s=0}^{\infty}\sum_{m=0}^{s}\sum_{j=0}^{s-m}\frac{\left(\dc_1^{s-m}\otimes
    \dc_{2}^{m+j}R_{\circ}\left(d_2^{s-m-j}d_1^{m}\left(y\right)\right)\right)\Delta\left(y\right)}{j!m!\left(s-m-j\right)!}h^s,
\label{guoduaib2}
\end{align}
\begin{align}
&\sum_{s=0}^{\infty}\sum_{m=0}^{s}\sigma\left(\frac{\left(\id_A\otimes R_{\circ}\left(d_2^{m}\left(y\right)\right)d_1^{m}\right)\left(\dc_{1}^{s-m}\otimes\dc_{2}^{s-m}\right)\Delta\left(x\right)}{m!\left(s-m\right)!}\right)h^s\nonumber\\
&\overset{\left(\ref{ped4}\right)}{=}  \sigma\left(\frac{\left(\dc_2^{s-m}\otimes R_{\circ}\left(d_2^{m}\left(y\right)\right)d_1^{m}\right)\Delta\left(d_1^{s-m}\left(x\right)\right)}{m!\left(s-m\right)!}\right)h^s\nonumber\\
&\overset{\left(\ref{ped2}\right)}{=}  \sigma\left(\frac{\left(\dc_2^{s-m}\otimes \dc_{1}^{m}R_{\circ}\left(y\right)d_2^{m}\right)\Delta\left(d_1^{s-m}\left(x\right)\right)}{m!\left(s-m\right)!}\right)h^s\nonumber\\
&\overset{\hphantom{\left(\ref{px3}\right)}}{=}
\sigma\left(\frac{\left(\dc_2^{s-m}\otimes
    \dc_{1}^{m}R_{\circ}\left(y\right)\right)\left(\id_{A}\ot
    d_2^{m}\right)\Delta\left(d_1^{s-m}\left(x\right)\right)}{m!\left(s-m\right)!}\right)h^s\nonumber\\
&\overset{\left(\ref{ped2}\right)}{=}\sigma\left(\frac{\left(\dc_2^{s-m}\otimes \dc_{1}^{m}R_{\circ}\left(y\right)d_2^{m}\right)\Delta\left(d_1^{s-m}\left(x\right)\right)}{m!\left(s-m\right)!}\right)h^s\nonumber\\
&\overset{\left(\ref{px4}\right)}{=}   \sigma\left(\frac{\left(\dc_2^{s-m+j}\otimes
    \dc_{1}^{m}R_{\circ}\left(y\right)\right)\Delta\left(d_2^{m-j}d_1^{s-m}\left(x\right)\right)}{\left(m-j\right)!j!\left(s-m\right)!}\right)h^s,
\label{guoduaib3}
\end{align}
\begin{align}
&\sum_{s=0}^{\infty}\sum_{m=0}^{s}\sigma\left(\frac{\left(L_{\circ}\left(d_1^{m}\left(y\right)\right)d_2^{m}\otimes\id_A\right)\left(\dc_{1}^{s-m}\otimes\dc_{2}^{s-m}\right)\Delta\left(x\right)}{m!\left(s-m\right)!}\right)h^s\nonumber\\
&\overset{\left(\ref{ped3}\right)}{=}  \sum_{s=0}^{\infty}\sum_{m=0}^{s}\sigma\left(\frac{\left(L_{\circ}\left(d_1^{m}\left(y\right)\right)d_2^{m}\otimes\dc_{1}^{s-m}\right)\Delta\left(d_2^{s-m}\left(x\right)\right)}{m!\left(s-m\right)!}\right)h^s\nonumber\\
&\overset{\left(\ref{ped1}\right)}{=}  \sum_{s=0}^{\infty}\sum_{m=0}^{s}\sigma\left(\frac{\left(\dc_{2}^{m}L_{\circ}\left(y\right)d_1^{m}\otimes\dc_{1}^{s-m}\right)\Delta\left(d_2^{s-m}\left(x\right)\right)}{m!\left(s-m\right)!}\right)h^s\nonumber\\
&\overset{\hphantom{\left(\ref{px3}\right)}}{=}\sum_{s=0}^{\infty}\sum_{m=0}^{s}\sigma\left(\frac{\left(\dc_{2}^{m}L_{\circ}\left(y\right)\otimes\dc_{1}^{s-m}\right)\left(d_1^{m}\otimes\id_{A}\right)\Delta\left(d_2^{s-m}\left(x\right)\right)}{m!\left(s-m\right)!}\right)h^s\nonumber\\
&\overset{\left(\ref{px3}\right)}{=}   \sum_{s=0}^{\infty}\sum_{m=0}^{s}\sum_{j=0}^{m}\sigma\left(\frac{\left(\dc_{2}^{m}L_{\circ}\left(y\right)\otimes\dc_{1}^{s-m+j}\right)\Delta\left(d_1^{m-j}d_2^{s-m}\left(x\right)\right)}{j!\left(m-j\right)!\left(s-m\right)!}\right)h^s.
\label{guoduaib4}
\end{align}
Adding (\ref{guoduaib1}), (\ref{guoduaib2}), (\ref{guoduaib3}), and (\ref{guoduaib4}) together, comparing the summation indices, using the commutativity of $d_1,d_2$ and (\ref{asi2}), we have
\begin{equation}
\begin{split}
&\sum_{s=0}^{\infty}\sum_{m=0}^{s}\frac{\left(L_{\circ}\left(d_1^{m}\left(x\right)\right)d_2^{m}\otimes\id_A\right)\left(\dc_{1}^{s-m}\otimes\dc_{2}^{s-m}\right)\Delta\left(y\right)}{m!\left(s-m\right)!}h^s\\
&+\sum_{s=0}^{\infty}\sum_{m=0}^{s}\frac{\left(\id_A\otimes R_{\circ}\left(d_2^{m}\left(x\right)\right)d_1^{m}\right)\left(\dc_{1}^{s-m}\otimes\dc_{2}^{s-m}\right)\Delta\left(y\right)}{m!\left(s-m\right)!}h^s\\
&+\sum_{s=0}^{\infty}\sum_{m=0}^{s}\sigma\left(\frac{\left(\id_A\otimes R_{\circ}\left(d_2^{m}\left(y\right)\right)d_1^{m}\right)\left(\dc_{1}^{s-m}\otimes\dc_{2}^{s-m}\right)\Delta\left(x\right)}{m!\left(s-m\right)!}\right)h^s\\
&+\sum_{s=0}^{\infty}\sum_{m=0}^{s}\sigma\left(\frac{\left(L_{\circ}\left(d_1^{m}\left(y\right)\right)d_2^{m}\otimes\id_A\right)\left(\dc_{1}^{s-m}\otimes\dc_{2}^{s-m}\right)\Delta\left(x\right)}{m!\left(s-m\right)!}\right)h^s=0.
\end{split}
    \label{sb2}
\end{equation}
    Embedding (\ref{sb2}) into $A_h\hat{\ot}A_h$
    by the $\bfK$-linear isomorphism $(A\ot A)[[h]]\cong A_h\hat{\ot}A_h$
    yields (\ref{tasib}).
\end{proof}
By Theorems~\ref{ddasi} and~\ref{thm-limasi}, we obtain the following result.
\begin{cor}\label{corpb}
    Let $(d_1,\dc_1)$ and $(d_2,\dc_2)$ be coherent derivations on a commutative and cocommutative ASI bialgebra $(A, \circ, \Delta)$. Suppose that $d_1,d_2$ commute, $\dc_1,\dc_2$ also commute and  {\rm(\ref{ed1})}, {\rm(\ref{ed2})}, {\rm(\ref{ed3})} and {\rm(\ref{ed4})} hold. Let $(A_h, \circ_h, \Delta_h)$ be the ASI bialgebra deformation induced by $(A,\circ, d_1,d_2,\Delta,\dc_1,\dc_2)$ and define $\{,\}$ and $\delta$ by {\rm(\ref{eq-dp})} and {\rm(\ref{UDFCA})}, respectively. Then $(A,\{,\},\circ,\delta,\Delta)$ coincides with the QCL of $(A_h, \circ_h, \Delta_h)$. Moreover, $(A,\{,\},\circ,\delta,\Delta)$ is a Poisson bialgebra, called the \textbf{Poisson bialgebra induced by} $(A,\circ, d_1,d_2,\Delta,\dc_1,\dc_2)$.
\end{cor}
\begin{rmk}
Conditions~{\rm(\ref{ed1})}, {\rm(\ref{ed2})}, {\rm(\ref{ed3})}
and {\rm(\ref{ed4})} are only sufficient conditions for
$(A,\{,\},\circ,\delta,\Delta)$ being a Poisson bialgebra, while
{\rm \cite{Lin}} gives the sufficient and necessary conditions.
\end{rmk}

We give an explicit example of the deformation-QCL process for a
Poisson bialgebra by coherent derivations.
\begin{ex}
    Let $(A, \circ, \Delta)$ be the $6$-dimensional commutative and cocommutative ASI bialgebra with a basis $\{e_1,e_2,e_3,e_4,e_5,e_6\}$
    whose non-zero product and coproduct are given respectively by
    \begin{eqnarray*}
         &&e_1\circ e_2=e_2\circ e_1=e_3,\;\;e_1\circ e_6=e_6\circ e_1=e_5,\;\;e_2\circ e_6=e_6\circ e_2=e_4;\\
         &&\Delta(e_6)=e_4\ot(e_3+e_5)+(e_3+e_5)\ot e_4.
    \end{eqnarray*}
  Assume that $\mathrm{e}^{i\frac{2\pi}{3}}\in\mathbf{k}$. Set $\zeta=\mathrm{e}^{i\frac{2\pi}{3}}$ and define linear maps $d_1,d_2\in\End_{\mathbf{k}}(A)$ respectively by
       \begin{eqnarray*}
         d_1(e_1)=\zeta e_1,\;\; d_1(e_2)=e_2,\;\;d_1(e_3)=-\zeta^2 e_3,\;\;d_1(e_4)=-\zeta e_4,\;\;d_1(e_5)=-e_5,\;\;d_1(e_6)=\zeta^2 e_6;\\
         d_2(e_1)=e_1,\;\; d_2(e_2)=\zeta e_2,\;\;d_2(e_3)=-\zeta^2 e_3,\;\;d_2(e_4)=-e_4,\;\;d_2(e_5)=-\zeta e_5,\;\;d_2(e_6)=\zeta^2 e_6.
         \end{eqnarray*}
   Take $\dc_1=-d_1,\dc_2=-d_2$. Then it is straightforward to check that $(d_1,\dc_1)$ and $(d_2,\dc_2)$ are coherent derivations on $(A, \circ, \Delta)$, $d_1,d_2$ commute, $\dc_1,\dc_2$ also commute and  {\rm(\ref{ed1})}, {\rm(\ref{ed2})}, {\rm(\ref{ed3})} and {\rm(\ref{ed4})} hold. Thus $(A,\circ, d_1,d_2,\Delta,\dc_1,\dc_2)$ induces an ASI bialgebra
   deformation $(A_h, \circ_h, \Delta_h)$ whose non-zero product and coproduct are given respectively by
    \begin{eqnarray*}
         &&e_1\circ_h e_2=\sum^{\infty}_{s=0}\zeta^{2s} e_3~\frac{h^{s}}{s!},\;\;
         e_2\circ_h e_1=\sum^{\infty}_{s=0}\zeta^{s}e_3~\frac{h^{s}}{s!},\;\;e_1\circ_h e_6=\sum^{\infty}_{s=0}e_5~\frac{h^{s}}{s!},\\
         &&e_6\circ_h e_1=\sum^{\infty}_{s=0}\zeta^{2s}e_5~\frac{h^{s}}{s!},\;\;e_2\circ_h e_6=\sum^{\infty}_{s=0}\zeta^{2s}e_4~\frac{h^{s}}{s!},\;\;e_6\circ_h e_2=\sum^{\infty}_{s=0}e_4~\frac{h^{s}}{s!};\\
         &&\Delta_h(e_6)=\sum^{\infty}_{s=0}\zeta^{2s}e_4\hat{\ot}(e_5+(-1)^s e_3)~\frac{h^{s}}{s!}+\sum^{\infty}_{s=0}(\zeta^{2s}e_3+e_5)\hat{\ot} e_4~\frac{h^{s}}{s!}.
    \end{eqnarray*}
    Moreover, in the QCL Poisson bialgebra $(A,\{,\},\circ,\delta,\Delta)$, the non-zero bracket and cobracket are given respectively by
    {\small \begin{eqnarray*}
         &&\{e_1,e_2\}=-\{e_2,e_1\}=(\zeta^2-\zeta)e_3,
        \;\{e_1,e_6\}=-\{e_6,e_1\}=(1-\zeta^2)e_5,\;\{e_2,e_6\}=-\{e_6,e_2\}=(\zeta^2-1)e_4;\\
         &&\delta(e_6)=(\zeta^{2}-1)(e_4\ot e_5-e_5\ot e_4)+2\zeta^{2}(e_3\ot e_4-e_4\ot e_3).
    \end{eqnarray*}}
\end{ex}
\section{Associative matched pair deformations}\label{S3}
We introduce the notion of an associative matched pair deformation
and show that the corresponding quasiclassical limit is a matched
pair of Poisson algebras, characterizing the
deformations-quasiclassical limits process for a Poisson
bialgebra. The antisymmetric infinitesimal bialgebra deformations
induced by coherent derivations are revisited via associative
matched pair deformations induced by derivations.
\subsection{Associative matched pair deformations and the quasiclassical limits}
\begin{defi}
    \label{defiam}
        Let $(A, \circ)$ and $(B, \cdot)$ be associative algebras,
    $l_\circ, r_\circ: A \to \End_\bfk(B)$ and $l_\cdot, r_\cdot: B \to \End_\bfk(A)$ be linear maps.
    Define a binary operation $\odot$ on the direct sum of vector spaces $A \oplus B$ by
    \begin{equation}
    (x,u)\odot(y,v):=(x\circ y+l_\cdot(u)y+r_\cdot(v)x,l_\circ(x)v+r_\circ(y)u+u\cdot v), \;\;\forall x, y\in A, u, v\in B\label{odotpipei}.
    \end{equation}
    Then $((A,\circ), (B, \cdot), l_\circ,r_\circ, l_\cdot, r_\cdot)$ is called a \textbf{matched pair of associative algebras} if $(A \oplus B, \odot)$ is an associative algebra. In this case, we denote the associative algebra $(A \oplus B,\odot)$ by $A \bowtie_{l_\cdot, r_\cdot}^{l_\circ,r_\circ} B$.
\end{defi}
Let $((A,\circ), (B, \cdot), l_\circ,r_\circ, l_\cdot, r_\cdot)$
be a matched pair of associative algebras. If $(A, \circ)$,
$(B,\cdot)$ and $A \bowtie_{l_\cdot, r_\cdot}^{l_\circ,r_\circ} B$
are commutative associative algebras, then $ l_\circ=r_\circ$ and
$l_\cdot=r_\cdot$. Thus, we set $\rho_{\circ}:=l_\circ=r_\circ$
and $\rho_{\cdot}:=l_\cdot=r_\cdot$, and then denote $((A,\circ),
(B, \cdot), l_\circ,r_\circ, l_\cdot, r_\cdot)$ and $A
\bowtie_{l_\cdot, r_\cdot}^{l_\circ,r_\circ} B$ by $((A,\circ),(B,
\cdot),\rho_{\circ},\rho_{\cdot})$ and $A
\bowtie_{\rho_{\cdot}}^{\rho_{\circ}} B$, respectively. In this
case, $((A,\circ),(B, \cdot),\rho_{\circ},\rho_{\cdot})$ is called
a \textbf{ matched pair of commutative associative algebras}.

\begin{defi}{\rm(\cite{NB})}
    \label{ppipei}
    Let $(P,\{,\},\circ)$ and $(Q,[,],\cdot)$ be Poisson algebras, $\rho_{\{,\}},\rho_{\circ}: P\to \mathrm{End}_{\bfk}(Q)$ and $\rho_{[,]},\rho_{\cdot}: Q\to \mathrm{End}_{\bfk}(P)$ be linear maps.
    For all $x, y\in P, u,v\in Q$, define binary operations $\llbracket,\rrbracket$ and $\odot$ on the direct sum of vector spaces $P \oplus Q$ by
\begin{equation}
\begin{split}
\llbracket(x,u),(y,v)\rrbracket:=&(\{x, y\}+\rho_{[,]}(u)y-\rho_{[,]}(v)x,\rho_{\{,\}}(x)v-\rho_{\{,\}}(y)u+[u,
v]),\\
(x,u)\odot(y,v):=&(x\circ y+\rho_{\cdot}(u)y+\rho_{\cdot}(v)x,\rho_{\circ}(x)v+\rho_{\circ}(y)u+u\cdot v).
\end{split}
\label{ppod}
\end{equation}
If $(P\oplus Q,\llbracket,\rrbracket,\odot)$ is a Poisson algebra,
then we call
$((P,\{,\},\circ),(Q,[,],\cdot),\rho_{\{,\}},\rho_{\circ},\rho_{[,]},\rho_{\cdot})$
a \textbf{matched pair of Poisson algebras} and denote the Poisson
algebra $(P\oplus Q,\llbracket,\rrbracket,\odot)$ by $P
\bowtie_{\rho_{[,]},\rho_{\cdot}}^{\rho_{\{,\}},\rho_{\circ}} Q$.
\end{defi}
\begin{defi}
Let $(A_{h},\circ_{h})$ and $(B_{h},\cdot_{h})$ be
topologically free associative algebras. Let $l_{\circ_h},r_{\circ_h}:
    A_{h}\rightarrow \mathrm{End}_{\mathbf{K}}(B_{h})$ and $l_{\cdot_h},r_{\cdot_h}: B_{h}\rightarrow \mathrm{End}_{\mathbf{K}}(A_{h})$ be $\mathbf{K}$-linear maps.
    If $((A_{h},\circ_{h}),(B_{h},\cdot_{h}),l_{\circ_h},r_{\circ_h},l_{\cdot_h},r_{\cdot_h})$ is a matched pair of associative algebras over $\mathbf{K}$, then we call it a \textbf{matched pair of topologically free associative algebras} and denote the topologically free associative algebra structure on $A_{h}\oplus B_{h}$ by $A_{h}\bowtie_{l_{\cdot_h},r_{\cdot_h}}^{l_{\circ_h},r_{\circ_h}} B_{h}$.

    Let $((A,\circ), (B, \cdot), l_\circ,r_\circ, l_\cdot, r_\cdot)$ be a matched pair of associative algebras and $((A_{h},\circ_{h}),(B_{h},\cdot_{h})$, $l_{\circ_h},r_{\circ_h},l_{\cdot_h},r_{\cdot_h})$ be a matched pair of topologically free associative algebras. If $A_{h}\bowtie_{l_{\cdot_h},r_{\cdot_h}}^{l_{\circ_h},r_{\circ_h}} B_{h}$ is an associative deformation of $A \bowtie_{l_\cdot, r_\cdot}^{l_\circ,r_\circ} B$, then we call $(A_{h},\circ_{h},V_{h},\cdot_{h},l_{\circ_h},r_{\circ_h},l_{\cdot_h},r_{\cdot_h})$ an \textbf{associative matched pair deformation} of
    $((A,\circ), (B, \cdot), l_\circ,r_\circ, l_\cdot, r_\cdot)$
\end{defi}

\begin{rmk}\label{ids}
    Let $\imath:A_{h}\oplus B_{h}\to(A\oplus B)_{h}$ be the map defined by
    \[\imath\,\Big(\sum_{s=0}^{\infty}x_{s}~h^{s},\sum_{s=0}^{\infty}v_{s}~h^{s}\Big):=\sum_{s=0}^{\infty}(x_{s},v_{s})~h^{s},\;\;\forall s\in\NN,x_{s}\in A, v_{s}\in B.\]
    Then $\imath$ is a $\bf K$-linear isomorphism.
With $A_{h}\oplus B_{h}$ identified with $(A\oplus B)_{h}$ this way, it makes sense to say that $A_{h}\bowtie_{l_{\cdot_h},r_{\cdot_h}}^{l_{\circ_h},r_{\circ_h}} B_{h}$ is an associative deformation of $A \bowtie_{l_\cdot, r_\cdot}^{l_\circ,r_\circ} B$.
\end{rmk}
\begin{pro}\label{edm}\label{proedm}
  Let $((A,\circ), (B, \cdot), l_\circ,r_\circ, l_\cdot, r_\cdot)$ be a matched pair of associative algebras and $((A_{h},\circ_{h}),(B_{h},\cdot_{h}),l_{\circ_h},r_{\circ_h},l_{\cdot_h},r_{\cdot_h})$ be a matched
  pair of topologically free associative algebras. Then $((A_{h},\circ_{h}),(B_{h},\cdot_{h}),l_{\circ_h},r_{\circ_h},l_{\cdot_h},r_{\cdot_h})$ is an associative matched pair deformation of
    $((A,\circ)$, $(B, \cdot), l_\circ,r_\circ$, $l_\cdot$, $r_\cdot)$
  if and only if the following congruences
    hold$:$
    \begin{align}
        x\circ_{h}y&\equiv x\circ y\pmod h,&u\cdot_{h}v&\equiv u\cdot v\pmod h,\label{e1}\\
        l_{\circ_{h}}(x)v&\equiv l_\circ(x)v\pmod h,&r_{\circ_{h}}(y)u&\equiv r_\circ(y)u\pmod h,\label{e2}\\
        l_{\cdot_{h}}(u)y&\equiv l_\cdot(u)y\pmod h,&r_{\cdot_{h}}(v)x&\equiv r_\cdot(v)x\pmod h,\;\;\forall x,y\in A, u,v\in B.\label{e3}
    \end{align}
\end{pro}
\begin{proof}
    Let $\odot_h$ denote the associative product of the topologically free associative algebra $A_{h}\bowtie_{l_{\cdot_h},r_{\cdot_h}}^{l_{\circ_h},r_{\circ_h}} B_{h}$. Then $((A_{h},\circ_{h}),(B_{h},\cdot_{h}),
    l_{\circ_h},r_{\circ_h},l_{\cdot_h},r_{\cdot_h})$ is an associative matched pair deformation of $((A,\circ), (B, \cdot)$, $l_\circ,r_\circ$, $l_\cdot, r_\cdot)$ if and only if
    \begin{equation*}
    (x,u)\odot_{h}(y,v)\equiv (x,u)\odot(y,v)\pmod h ,\;\;\forall x,y\in A, u,v\in B.
    \end{equation*}
    It holds if and only if for all
    $x,y\in A, u,v\in B$,
      \begin{equation}\label{em}
      \begin{split}
    (x\circ_h y+l_{\cdot_{h}}(u)y+r_{\cdot_{h}}(v)x,l_{\circ_{h}}(x)v+r_{\circ_{h}}(y)u+u\cdot_h v)
    \equiv\\\; (x\circ y+l_\cdot(u)y+r_\cdot(v)x,l_\circ(x)v+r_\circ(y)u+u\cdot v)\pmod h.
    \end{split}
    \end{equation}
    By suitably choosing $x$, $y$, $u$ and $v$, we find that \eqref{em} holds if and only if
    (\ref{e1}), (\ref{e2}) and (\ref{e3}) hold.  \end{proof}
\begin{rmk}
Condition (\ref{e1}) holds if and only if
    $(A_{h},\circ_{h})$ is an
    associative deformation of $(A,\circ)$ and  $(B_{h},\cdot_{h})$ is an
    associative deformation of $(B,\cdot)$.
\end{rmk}

\begin{defi}
 Let $((A,\circ),(B, \cdot),\rho_{\circ},\rho_{\cdot})$ be a matched pair of commutative associative algebras and $((A_{h},\circ_{h}),(B_{h},\cdot_{h}),l_{\circ_h},r_{\circ_h},l_{\cdot_h},r_{\cdot_h})$ be an associative matched pair deformation. Denote the corresponding \qcls of $(A_{h},\circ_{h})$ and $(B_{h},\cdot_{h})$ by $(A,\{,\},\circ)$ and $(B,[,],\cdot)$, respectively. Define linear maps $\rho_{\{,\}}: A\rightarrow
\mathrm{End}_{\bfk}(B),\rho_{[,]}: B\to
\mathrm{End}_{\bfk}(A)$ by
\begin{eqnarray}
    \rho_{\{,\}}(x)(v)&\equiv&\frac{l_{\circ_h}(x)v-r_{\circ_h}(x)v}{h}\pmod{h}, \label{cmod}\\
\rho_{[,]}(v)(x)&\equiv&\frac{l_{\cdot_h}(v)x-r_{\cdot_h}(v)x}{h}\pmod{h}, \;\; \forall x\in A, v\in B\label{cmod2}.
\end{eqnarray}
Then $((A,\{,\},\circ),(B,[,]
,\cdot),\rho_{\{,\}},\rho_{\circ},\rho_{[,]},\rho_{\cdot})$ is
called the \textbf{quasiclassical limit (\qcl) of the associative
matched pair deformation}
$(A_{h},\circ_{h},V_{h},\cdot_{h},l_{\circ_h},r_{\circ_h},l_{\cdot_h},r_{\cdot_h})$.
\end{defi}

\begin{thm}\label{qclpipei}
 Let $((A,\circ),(B, \cdot),\rho_{\circ},\rho_{\cdot})$ be a matched pair of commutative associative algebras, $((A_{h},\circ_{h}),(B_{h},\cdot_{h}),l_{\circ_h},r_{\circ_h},l_{\cdot_h},r_{\cdot_h})$ be an
  associative matched pair deformation and $((A,\{,\},\circ),(B,[,]$, $\cdot)$, $\rho_{\{,\}},\rho_{\circ},\rho_{[,]},\rho_{\cdot})$ be the QCL.
 Then $((A,\{,\},\circ),(B,[,] ,\cdot),\rho_{\{,\}},\rho_{\circ},\rho_{[,]},\rho_{\cdot})$ is a matched pair of Poisson algebras. Moreover, the \qcl of the associative deformation
$A_{h}\bowtie_{l_{\cdot_h},r_{\cdot_h}}^{l_{\circ_h},r_{\circ_h}} B_{h}$ of the commutative associative algebra $A\bowtie_{\rho_{\cdot}}^{\rho_{\circ}} B$ is exactly the Poisson algebra $A\bowtie_{\rho_{[,]},\rho_{\cdot}}^{\rho_{\{,\}},\rho_{\circ}} B$.
\end{thm}
\begin{proof}
By Theorem~\ref{thacl}, $(A,\{,\},\circ)$ and $(B,[,],\cdot)$ are Poisson algebras.
Set $(A\oplus B,\odot):=A\bowtie_{\rho_{\cdot}}^{\rho_{\circ}} B$ and $(A_h\oplus B_h,\odot_h):=A_{h}\bowtie_{l_{\cdot_h},r_{\cdot_h}}^{l_{\circ_h},r_{\circ_h}} B_{h}$.
For all $x, y\in A, u,v\in B$, define binary operations $\llbracket,\rrbracket$ and $\odot$  on $A \oplus B$ by (\ref{ppod}). Then we have
    \begin{align*}
    &(x,u)\odot_{h}(y,v)-(y,v)\odot_{h}(x,u)\\
    &=(x\circ_{h}y-y\circ_{h}x+l_{\cdot_{h}}(u)y-r_{\cdot_{h}}(u)y+r_{\cdot_{h}}(v)x-l_{\cdot_{h}}(v)x, 0)\\
    &\qquad+(0,l_{\circ_{h}}(x)v-r_{\circ_{h}}(x)v+r_{\circ_{h}}(y)u-l_{\circ_{h}}(y)u+u\cdot_{h}v-v\cdot_{h}u) \\
    &\equiv(\{x, y\}+\rho_{[,]}(u)y-\rho_{[,]}(v)x,\rho_{\{,\}}(x)v-\rho_{\{,\}}(y)u+[u,
    v])\pmod{h^2}.
    \end{align*}
    Thus $(A\oplus B,\llbracket,\rrbracket,\odot)$ coincides with the QCL of $(A_h\oplus B_h,\odot_h)$.
    By Theorem~\ref{thacl}, $(A\oplus B,\llbracket,\rrbracket,\odot)$ is a Poisson algebra. Hence $((A,\{,\},\circ),(B,[,] ,\cdot),\rho_{\{,\}},\rho_{\circ},\rho_{[,]},\rho_{\cdot})$ is a matched pair of Poisson algebras. The last statement holds since the Poisson algebras $A\bowtie_{\rho_{[,]},\rho_{\cdot}}^{\rho_{\{,\}},\rho_{\circ}} B$ and $(A\oplus B,\llbracket,\rrbracket,\odot)$ coincide.
\end{proof}

\subsection{Associative matched pair deformations and antisymmetric infinitesimal bialgebra deformations}

Let $V$ be a vector space and set $V^{\ast}:=\Hom_{\bfk}(V,\bfk)$.
Let $\langle,\rangle:V^{\ast}\ot V\to\bfk$ be the natural pairing
between the dual space of a vector space and the vector space,
that is,
\[\langle u^{\ast},v \rangle:=u^{\ast}(v),\;\;\forall u^{\ast}\in V^{\ast},v\in V.\]
Let $W$ be a vector space and $T:V\to W$ be a linear map.
Define a linear map $T^{\ast}:W^{\ast}\to V^{\ast}$ by
\begin{equation*}
\langle T^{\ast}(w^{\ast}), v \rangle = \langle w^{\ast}, T(v) \rangle, \;\;
\forall v \in V, \; w^{\ast} \in W^{\ast},
\end{equation*}
called the linear dual map of $T$.

Let $A$ and $V$ be vector spaces. For a linear map $\varrho:A\rightarrow \textrm{End}_{\bfk}(V)$, we let $\varrho^{\ast}:A\rightarrow \textrm{End}_{\bfk}(V^{\ast})$ be the unique linear map satisfying
\begin{equation*}
\langle\varrho^{\ast}(x)u^{\ast},v\rangle=\langle u^{\ast},\varrho(x)v\rangle, \quad \forall x\in A, u^{\ast}\in V^{\ast},v\in V.
\end{equation*}
\begin{pro}
    \label{mpasi} {\rm{(\cite{Bai})}}
Let $A$ be a finite-dimensional vector space, $\circ$ be a binary operation on $A$, and $\Delta:A\to A\ot A$ be a linear map. Define a binary operation
$\cdot$  on $A^{\ast}$  by
    \begin{equation*}
    \langle a^{\ast}\cdot b^{\ast},x\rangle=\langle a^{\ast}\ot b^{\ast},\Delta(x)\rangle, \quad\forall  x\in A, a^{\ast},b^{\ast}\in A^{\ast},
    \end{equation*}
called the dual operation of $\Delta$. Then $(A,\circ,\Delta)$ is
an ASI bialgebra if and only if
$((A,\circ),(A^{\ast},\cdot),R_{\circ}^{\ast},L_{\circ}^{\ast}$,
$R_{\cdot}^{\ast},L_{\cdot}^{\ast})$ is a matched pair of
associative algebras.  Moreover, $(A,\circ,\Delta)$
is a commutative and cocommutative ASI bialgebra if and only if
$((A,\circ),(A^{\ast},\cdot),R_{\circ}^{\ast},R_{\cdot}^{\ast})$
is a matched pair of commutative associative algebras.
\end{pro}
In the above proposition, the ASI bialgebra $(A,\circ,\Delta)$ is called the \textbf{ASI bialgebra corresponding to} the matched pair of associative algebras $((A,\circ),(A^*,\cdot),R_\circ^*,L_\circ^*,R_\cdot^*,L_\cdot^*)$,
and conversely the latter is called the \textbf{matched pair of associative algebras corresponding to} $(A,\circ,\Delta)$. Similarly, in the commutative and cocommutative case, $(A,\circ,\Delta)$ is called the \textbf{commutative and cocommutative ASI bialgebra corresponding to} $((A,\circ),(A^*,\cdot),R_\circ^*,R_\cdot^*)$,
and conversely the latter is called the  \textbf{matched pair of commutative associative algebras corresponding to} $(A,\circ,\Delta)$. Similar notions are applied to Propositions~\ref{eqvtam}, \ref{eqvmpp}, \ref{eqvbasi}, \ref{eqvbasih} and \ref{eqvbpb}.

Let $V$ be a vector space. Then by
\cite[\uppercase\expandafter{\romannumeral16}.7, Exercise 2]{Ka},
there is the $\bfK$-module isomorphism
$\textrm{Hom}_{\mathbf{K}}(V_{h},\mathbf{K})\cong\textrm{Hom}_{\bfk}(V,\bfk)[[h]]=(V^{\ast})_{h}$.
Set $V^{\star}_{h}:=(V^{\ast})_{h}$ and define a pairing $\langle\
,\ \rangle_{\bf K}:V_{h}^{\star}\hat{\ot}V_{h}\to \bfK$ by
\[\langle \sum_{s=0}^{\infty}u_{s}^{\ast}~h^{s},\sum_{s=0}^{\infty}v_{s}~h^{s}\rangle_{\bfK}:=\sum_{s,l\geq 0}\langle u_{s}^{\ast},v_l \rangle~h^{s+l},\;\;\forall u_{s}^{\ast}\in
V^{\ast}, v_s\in V.\] For topologically free $\bf K$-modules $A_h$
and $V_h$, and a $\bf K$-linear map $\varrho_{h}:A_h\to
\textrm{End}_{\bf K}(V_{h})$, let $\varrho^{\star}_{h}:A_{h}\to
\textrm{End}_{\bf K}(V^{\star}_h)$ be the unique $\bf K$-linear
map such that
\begin{equation*}
\langle
\varrho^{\star}_{h}(x^{\prime})u^{\star},v^{\prime}\rangle_{\bf
K}=\langle u^{\star},\varrho_{h}(x^{\prime})v^{\prime}\rangle_{\bf
K}, \;\;\forall x^{\prime}\in A_h , v^{\prime}\in V_h ,
u^{\star}\in V^{\star}_{h}.
\end{equation*}

Let $V$ and $W$ be finite-dimensional vector spaces. The pairing $\langle\ ,\ \rangle_{\bf K}$ can be turned into
a pairing between $W_{h}^{\star}\hat{\ot}V_{h}^{\star}$ and $W_{h}\hat{\ot}V_{h}$ by
\[\langle w^{\star}\hat{\ot}v^{\star},w_{h}\hat{\ot}v_{h} \rangle_{\bf K}=\langle w^{\star},w_{h}\rangle_{\bf K}\langle v^{\star},v_{h}\rangle_{\bf K},\quad\forall w^{\star}\hat{\ot}v^{\star}\in W_{h}^{\star}\hat{\ot}V_{h}^{\star},w_{h}\hat{\ot}v_{h} \in W_{h}\hat{\ot}V_{h}.\]
\begin{pro}
    \label{eqvtam}
    Let $A$ be a finite-dimensional vector space, $\circ_h$ be a $\bf K$-bilinear operation on $A_h$ and $\Delta_h:A_h\to A_h\hat{\ot} A_h$ be a $\bf K$-map. Define an operation $\cdot_h$ on $A_h^{\star}$ by
    \[\langle a^{\star}\cdot_h b^{\star},x_h\rangle_{\bf K}=\langle a^{\star}\hat{\ot}b^{\star},\Delta_h(x_h)\rangle_{\bf K},\;\;\forall a^{\star},b^{\star}\in A_h^{\star},\;x_h\in A_h,\]
    called the topological dual operation of $\Delta_h$. Then $(A_h,\circ_h,\Delta_h)$ is a topologically free ASI bialgebra if and only if $((A_h,\circ_h),(A^{\star}_h,\cdot_h),R^{\star}_{\circ_h},L^{\star}_{\circ_h},R^{\star}_{\cdot_h},
    L^{\star}_{\cdot_h})$ is a matched pair of topologically free associative algebras. \delete{In this case, we call $(A_h,\circ_h,\Delta_h)$ the \textbf{topologically free ASI bialgebra corresponding to}
    the matched pair of topologically free associative algebras $((A_h,\circ_h),(A^{\star}_h,\cdot_h),R^{\star}_{\circ_h},L^{\star}_{\circ_h},R^{\star}_{\cdot_h},L^{\star}_{\cdot_h})$
    and $((A_h,\circ_h),(A^{\star}_h,\cdot_h),R^{\star}_{\circ_h},L^{\star}_{\circ_h},R^{\star}_{\cdot_h},L^{\star}_{\cdot_h})$ the \textbf{matched pair of topologically free associative algebras corresponding to}
    the topologically free ASI bialgebra $(A_h,\circ_h,\Delta_h)$.} 
\end{pro}
\begin{proof}
    Suppose that the dimension of $A$ is $n\in\mathbb{N}$ and $\{e_1,...,e_n\}$ is a basis of $A$. For $1\leq i,j,k\leq n$, let $\lambda^{ijk}_h$ be the elements of $\bf K$ such that
    $$\Delta_h(e_i)=\sum_{1\leq j,k\leq n}\lambda^{ijk}_he_j\hat{\ot} e_k.$$
Denote the quotient field of the domain $\bf K$ by $\mathbb{F}$ and set $A_{(h)}:=A_h\otimes_{\mathbf{K}}\mathbb{F}$. The $A_{(h)}$ is an $n$-dimensional vector space over the field $\mathbb{F}$. For $x_h\in A_h$, we identify it with the element $x_h\otimes_{\mathbf{K}}1\in A_{(h)}$. So $A_h$ becomes a subset of $A_{(h)}$. With this identification, $A_{(h)}=\{\sum_{i=1}^{n}\xi_ie_i|\xi_1,...,\xi_n\in\mathbb{F}\}$. Denote the $\mathbb{F}$-bilinear extension of the operation $\circ_h$ to $A_{(h)}$ by $\circ_{(h)}$. Define $\Delta_{(h)}:A_{(h)}\ot_{\mathbb{F}}A_{(h)}\to A_{(h)}$ by
    $$\Delta_{(h)}(e_i)=\sum_{1\leq j,k\leq n}\lambda^{ijk}_he_j\ot_{\mathbb{F}} e_k,\;\;1\leq i\leq n.$$
    Thus $(A_h,\circ_h,\Delta_h)$ is a topologically free ASI bialgebra if and only if $(A_{(h)},\circ_{(h)},\Delta_{(h)})$ is an ASI bialgebra over the field $\mathbb{F}$.

    Set $A_{(h)}\spcheck:=\Hom_{\mathbb{F}}(A_{(h)},\mathbb{F})$ and denote the dual operation of $\Delta_{(h)}$ relative to the field $\mathbb{F}$ by $\cdot_{(h)}$.
    Denote the linear dual map of $L_{\circ_{(h)}}(x_{(h)}),\;R_{\circ_{(h)}}(x_{(h)}),\;L_{\cdot_{(h)}}(a\spcheck),\; R_{\cdot_{(h)}}(a\spcheck)$ relative to the filed
     $\mathbb{F}$ by $L_{\circ_{(h)}}\spcheck(x_{(h)}),\;R_{\circ_{(h)}}\spcheck(x_{(h)}),\;L_{\cdot_{(h)}}\spcheck(a\spcheck),\; R_{\cdot_{(h)}}\spcheck(a\spcheck)$, where
     $x_{(h)}\in A_{(h)},\;a\spcheck\in A_{(h)}\spcheck$. By Proposition~{\ref{mpasi}}, $(A_{(h)},\circ_{(h)},\Delta_{(h)})$ is an ASI bialgebra over the field $\mathbb{F}$ if and only if
      $((A_{(h)},\circ_{(h)}),(A_{(h)}\spcheck,\cdot_{(h)}),R_{\circ_{(h)}}\spcheck,L_{\circ_{(h)}}\spcheck,R_{\cdot_{(h)}}\spcheck$, $L_{\cdot_{(h)}}\spcheck)$ is a matched pair of associative algebras over the field $\mathbb{F}$.

    For $a^{\star}\in A_h^{\star}$, we denote its $\mathbb{F}$-linear extension by $a^{\star}_{\mathbb{F}}$, that is, $a^{\star}_{\mathbb{F}}(\sum_{i=1}^{n}\xi_ie_i):=\sum_{i=1}^{n}\xi_ia^{\star}(e_i)$, where
    $\xi_1,...,\xi_n\in\mathbb{F}$. Note that for all $x_h\in A_h,\;a^{\star}\in A_h^{\star},\;1\leq i\leq
    n$, we  have
    \begin{eqnarray*}
    &&(R_{\circ_{(h)}}\spcheck(x_h)a^{\star}_{\mathbb{F}})(e_i)=a^{\star}_{\mathbb{F}}(e_i\circ_{(h)}x_h)\\
        &&=a^{\star}_{\mathbb{F}}(e_i\circ_{h}x_h)
        =\langle a^{\star},e_i\circ_{h}x_h\rangle_{\bf
        K}=(R_{\circ_{h}}^{\star}(x_h)a^{\star})(e_i).
    \end{eqnarray*}
Hence
$(R_{\circ_{h}}^{\star}(x_h)a^{\star})_{\mathbb{F}}=R_{\circ_{(h)}}\spcheck(x_h)a^{\star}_{\mathbb{F}}$.
Similarly, we obtain
\begin{align*}
    (L_{\circ_{h}}^{\star}(x_h)a^{\star})_{\mathbb{F}}&=L_{\circ_{(h)}}\spcheck(x_h)a^{\star}_{\mathbb{F}},&R_{\cdot_{h}}^{\star}(a^{\star})x_h&=R_{\cdot_{(h)}}\spcheck(a^{\star}_{\mathbb{F}})x_h,\\
    L_{\cdot_{h}}^{\star}(a^{\star})x_h&=L_{\cdot_{(h)}}\spcheck(a^{\star}_{\mathbb{F}})x_h,&(a^{\star}\cdot_h b^{\star})_{\mathbb{F}}&=a^{\star}_{\mathbb{F}}\cdot_{(h)} b^{\star}_{\mathbb{K}},\;\;\forall x_h\in A_h,\;a^{\star},b^{\star}\in A_h^{\star}.
\end{align*}
Consider the $\bf K$-linear map $$A_h\oplus A_h^{\star}\to A_{(h)}\oplus A_{(h)}\spcheck,\;\;(x_h,a^{\star})\mapsto(x_h,a^{\star}_{\mathbb{F}}).$$The image of
$$(x_h\circ_{h}y_h+R_{\cdot_{h}}^{\star}(a^{\star})y_h+L_{\cdot_{h}}^{\star}(b^{\star})x_h,R_{\circ_{h}}^{\star}(x_h)b^{\star}+L_{\circ_{h}}^{\star}(y_h)a^{\star}+a^{\star}\cdot_h b^{\star})$$ is
$$(x_h\circ_{(h)}y_h+R_{\cdot_{(h)}}\spcheck(a^{\star}_{\mathbb{F}})y_h+L_{\cdot_{(h)}}\spcheck(b^{\star}_{\mathbb{F}})x_h,R_{\circ_{(h)}}\spcheck(x_h)b^{\star}_{\mathbb{F}}+
L_{\circ_{(h)}}\spcheck(x_h)a^{\star}_{\mathbb{F}}+a^{\star}_{\mathbb{F}}\cdot_{(h)}
b^{\star}_{\mathbb{F}}),$$ where $x_h,y_h\in
A_h,\;a^{\star},b^{\star}\in A_h^{\star}$. It is clear that this
map is injective. So
$((A_h,\circ_h),(A^{\star}_h,\cdot_h),R^{\star}_{\circ_h},L^{\star}_{\circ_h}$,
$R^{\star}_{\cdot_h},L^{\star}_{\cdot_h})$ is a matched pair of
topologically free associative algebras if and only if
$((A_{(h)},\circ_{(h)})$,
$(A_{(h)}\spcheck$,$\cdot_{(h)}),R_{\circ_{(h)}}\spcheck,
L_{\circ_{(h)}}\spcheck,R_{\cdot_{(h)}}\spcheck,L_{\cdot_{(h)}}\spcheck)$
a matched pair of topologically free associative algebras over the
field $\mathbb{F}$. Then the conclusion follows.
\end{proof}

\begin{pro}\label{prodc}
   Let $A_h$ and $V_h$ be topologically free $\bf K$-modules.
   Let $\varrho_{h},l_h,r_h:A_h\to \textrm{End}_{\bf K}(V_h)$ be $\bf K$-linear maps.
   \begin{enumerate}
       \item\label{dedu} Let $\varrho:A\to {\rm End}_{\bf k}(V)$ be the linear map that satisfies $\varrho_{h}(x)\equiv\varrho(x)\pmod h$ for all $x\in A$. Then $\varrho^{\star}_{h}(x)\equiv\varrho^{\ast}(x)\pmod h$.
       \item \label{cldu} If $l_{h}\equiv r_h\pmod h$ and define a linear map $\rho:A\to {\rm End}_{\bf k}(V)$ by
       \[\rho(x)v\equiv\frac{l_h(x)v-r_h(x)v}{h}\pmod{h},\;\;\forall x\in A,v\in V,\]
then the following congruence holds:
       \[\rho^{\ast}(x)v\equiv\frac{l^{\star}_h(x)v-r^{\star}_h(x)v}{h}\pmod{h},\;\;\forall x\in A,v\in V.\]
   \end{enumerate}
\end{pro}
\begin{proof}
    (\ref{dedu}). For all $x\in A$, $u_{h}\in V_h$, and $v^{\ast}\in V^{\ast}$, we have
    \begin{eqnarray*}
        \langle \varrho_{h}^{\star}(x)v^{\ast},u_{h}\rangle_{\bf K}
        &&=\langle v^{\ast},\varrho_{h}(x)u_{h}\rangle_{\bf K}\\
        &&\equiv\langle v^{\ast},\varrho(x)u\rangle\pmod{h}\\
        &&\equiv\langle \varrho^{\ast}(x)v^{\ast},u\rangle\pmod{h}\\
        &&\equiv\langle \varrho^{\ast}(x)v^{\ast},u_{h}\rangle_{\bf K}\pmod{h},
    \end{eqnarray*}
    where $u$ is the element in $V$ that satisfies $u_{h}\equiv u\pmod h$. Thus, the conclusion holds.

    (\ref{cldu}) follows from (\ref{dedu}) by putting $\varrho_{h}=\frac{l_h-r_h}{h}$ and $\varrho=\rho$.
\end{proof}
\begin{thm}\label{eqvdadm}
    Let $(A,\circ,\Delta)$ be a finite-dimensional ASI bialgebra and $(A_h,\circ_h,\Delta_h)$ be a topologically free ASI bialgebra. Let $\cdot$ and $\cdot_h$ be the dual operation and the topological dual operation of $\Delta$ and $\Delta_h$, respectively. Then $(A_h,\circ_h,\Delta_h)$ is an ASI bialgebra deformation of $(A,\circ,\Delta)$ if and only of
$((A_h,\circ_h),(A^{\star}_h,\cdot_h),R^{\star}_{\circ_h},L^{\star}_{\circ_h},R^{\star}_{\cdot_h},L^{\star}_{\cdot_h})$
is an associative matched pair deformation of
$((A,\circ),(A^{\ast},\cdot),R_{\circ}^{\ast},L_{\circ}^{\ast},R_{\cdot}^{\ast},L_{\cdot}^{\ast})$.
\end{thm}

\begin{proof}
    Assume that $((A_h,\circ_h),(A^{\star}_h,\cdot_h),R^{\star}_{\circ_h},L^{\star}_{\circ_h},R^{\star}_{\cdot_h},L^{\star}_{\cdot_h})$ is an associative matched pair deformation of $((A,\circ),(A^{\ast},\cdot),R_{\circ}^{\ast},L_{\circ}^{\ast},R_{\cdot}^{\ast},L_{\cdot}^{\ast})$. Thus, by Proposition~\ref{proedm}, $(A_h,\circ_h)$ and $(A^{\star}_h,\cdot_h)$ are associative deformations of $(A,\circ)$ and $(A^{\ast},\cdot)$, respectively. For all $a^{\star},b^{\star}\in A^{\star}_h,\;x\in A$,
    \begin{equation}\label{cc}
    \begin{split}
        \langle a^{\star}\hat{\otimes}b^{\star},\Delta_h(x)\rangle_{\bf K}
    &=\langle a^{\star}\cdot_h b^{\star},x\rangle_{\bf K}\\
    &\equiv
    \langle a^{\ast}\cdot b^{\ast},x\rangle\pmod h\\
    &\equiv\langle a^{\ast}\otimes b^{\ast},\Delta(x)\rangle\pmod h,
    \end{split}
    \end{equation}
    where $a^{\ast},b^{\ast}\in A^{\ast}$ satisfy $$a^{\star}\equiv a^{\ast}\pmod h,\;\;b^{\star}\equiv b^{\ast}\pmod h.$$
    Hence, $(A_h,\Delta_h)$ is a coassociative deformation of $(A,\Delta)$. Therefore, $(A_h,\circ_h,\Delta_h)$ is an ASI bialgebra deformation of $(A,\circ,\Delta)$.

    Conversely, assume that $(A_h,\circ_h,\Delta_h)$ is an ASI bialgebra deformation of $(A,\circ,\Delta)$. Then we have
    \begin{equation}\label{cimh}
        x\circ_{h}y\equiv x\circ y\pmod h,\;\;\forall x,y\in A.
    \end{equation}
    Hence we obtain the following congruences
    \[L_{\circ_h}(x)\equiv L_{\circ}(x)\pmod{h},\;\;R_{\circ_h}(x)\equiv R_{\circ}(x)\pmod{h},\;\;\forall x\in A.\]
    By Proposition~\ref{prodc}~(\ref{dedu}), we have
    \begin{equation}\label{sci}
        R^{\star}_{\circ_h}(x)a^{\ast}\equiv R_{\circ}^{\ast}(x)a^{\ast}\pmod h,\;\;L^{\star}_{\circ_h}(x)a^{\ast}\equiv
        L_{\circ}^{\ast}(x)a^{\ast}\pmod h,
        \;\;\forall x\in A,\;a^{\ast}\in A^{\ast}.
    \end{equation}
Since $(A_h,\Delta_h)$ is a coassociative deformation of $(A,\Delta)$, similar to (\ref{cc}), we show that
\begin{equation}\label{cdmh}
    a^{\ast}\cdot_h b^{\ast}\equiv a^{\ast}\cdot b^{\ast}\pmod h,\;\;\forall a^{\ast},b^{\ast}\in A^{\ast}.
\end{equation}
Similar to (\ref{sci}),  we have
\begin{equation}\label{scd}
R^{\star}_{\cdot_h}(a^{\ast})x\equiv
R_{\cdot}^{\ast}(a^{\ast})x\pmod
h,\;\;L^{\star}_{\cdot_h}(a^{\ast})x\equiv
R_{\cdot}^{\ast}(a^{\ast})x\pmod h,\;\;\forall x\in
A,\;a^{\ast}\in A^{\ast}.
\end{equation}
Finally, by (\ref{cimh}), (\ref{sci}), (\ref{cdmh}), (\ref{scd}) and Proposition~\ref{proedm}, we complete the proof.
\end{proof}

There is also a similar characterization of a finite-dimensional Poisson bialgebra.
\begin{pro}
    {\rm{(\cite{NB})}}\label{eqvmpp}
Let $A$ be a finite-dimensional vector space, $\{,\},\circ$ be binary
operations on $A$ and $\delta,\Delta:A\to A\ot A$ be linear maps.
Let $[,]$ and $\cdot$ be the dual operations of $\delta$ and
$\Delta$, respectively. Then $(A,\{,\},\circ,\delta, \Delta)$ is a
Poisson bialgebra if and only if
$((A,\{,\},\circ),(A^{\ast},[,],\cdot),-\ad_{\{,\}}^{\ast},R_{\circ}^{\ast},-\ad_{[,]}^{\ast},R_{\cdot}^{\ast})$
is a matched pair of Poisson algebras. \delete{ In this case,
$(A,\{,\},\circ,\delta, \Delta)$ is called the \textbf{Poisson
bialgebra corresponding to} the matched pair of Poisson algebras
$((A,\{,\},\circ),(A^{\ast},[,],\cdot),-\ad_{\{,\}}^{\ast},R_{\circ}^{\ast},-\ad_{[,]}^{\ast},R_{\cdot}^{\ast})$,
and
$((A,\{,\},\circ),(A^{\ast},[,],\cdot),-\ad_{\{,\}}^{\ast},R_{\circ}^{\ast},-\ad_{[,]}^{\ast},R_{\cdot}^{\ast})$
is called the \textbf{matched pair of Poisson algebras
corresponding to } the Poisson bialgebra $(A,\{,\},\circ,\delta,
\Delta)$.} 
\end{pro}
\begin{thm}
    \label{qcldadm}
    Let $A$ be a finite-dimensional vector space, $\circ$ be a binary operation on $A$, $\Delta:A\to A\ot A$ be a linear map and $\cdot$ be the dual operation of $\Delta$. Let $\circ_h$ be a $\bf K$-bilinear operation on $A_h$, $\Delta_h:A_h\to A_h\hat{\ot} A_h$ be a $\bf K$-linear map and $\cdot_h$ be its topological dual operation.
    \begin{enumerate}
        \item\label{qcldadm1} Assume that $(A,\circ,\Delta)$ is a commutative and cocommutative ASI bialgebra, $(A_{h},\circ_{h},\Delta_{h})$ is an ASI bialgebra deformation and $(A,\{,\},\circ,\delta,\Delta)$ is the QCL. Let $[,]$ denote the dual operation of $\delta$. Then $((A_h,\circ_h),(A^{\star}_h,\cdot_h),R^{\star}_{\circ_h},L^{\star}_{\circ_h},R^{\star}_{\cdot_h},L^{\star}_{\cdot_h})$, which is the matched pair of topologically free associative algebras corresponding to $(A_{h},\circ_{h},\Delta_{h})$, is an associative matched pair deformation of  $((A,\circ),(A^{\ast},\cdot),R_{\circ}^{\ast},R_{\cdot}^{\ast})$, which is the matched pair of commutative associative algebras corresponding to $(A,\circ,\Delta)$. Moreover, $((A,\{,\},\circ),(A^{\ast},[,],\cdot),-\ad_{\{,\}}^{\ast},R_{\circ}^{\ast},-\ad_{[,]}^{\ast},R_{\cdot}^{\ast})$, which is the matched pair of Poisson algebras corresponding to $(A,\{,\},\circ,\delta,\Delta)$ coincides with the QCL of $((A_h,\circ_h),(A^{\star}_h,\cdot_h),R^{\star}_{\circ_h},L^{\star}_{\circ_h},R^{\star}_{\cdot_h},L^{\star}_{\cdot_h})$.
         \item\label{qcldadm2} Assume that $((A,\circ),(A^{\ast},\cdot),R_{\circ}^{\ast},R_{\cdot}^{\ast})$ is a matched pair of commutative associative algebras, $((A_h,\circ_h),(A^{\star}_h,\cdot_h),R^{\star}_{\circ_h},L^{\star}_{\circ_h},R^{\star}_{\cdot_h},L^{\star}_{\cdot_h})$ is an associative matched pair deformation and the QCL is $(A,\{,\},\circ,A^{\ast},[,],\cdot,\rho_{\{,\}},R_{\circ}^{\ast},\rho_{[,]},R_{\cdot}^{\ast})$. Then $(A_{h},\circ_{h},\Delta_{h})$, which is the topologically free ASI bialgebra corresponding to $((A_h,\circ_h),(A^{\star}_h,\cdot_h),R^{\star}_{\circ_h},L^{\star}_{\circ_h},R^{\star}_{\cdot_h},L^{\star}_{\cdot_h})$, is an ASI bialgebra deformation of $(A,\circ,\Delta)$, which is the commutative and cocommutative ASI bialgebra corresponding to $((A,\circ),(A^{\ast},\cdot),R_{\circ}^{\ast},R_{\cdot}^{\ast})$. Moreover, the Poisson bialgebra corresponding to the matched pair of Poisson algebras $((A,\{,\},\circ),(A^{\ast},[,],\cdot),-\ad_{\{,\}}^{\ast},R_{\circ}^{\ast},-\ad_{[,]}^{\ast},R_{\cdot}^{\ast})$ coincides with the QCL of $(A_{h},\circ_{h},\Delta_{h})$.
         \end{enumerate}
         That is, we have the following commutative
         diagram:
         \begin{equation}
         \begin{gathered}
\xymatrixcolsep{5pc}\xymatrixrowsep{5pc}
\xymatrix{
    *\txt{commutative and\\
        cocommutative\\ ASI bialgebra\\$(A,\circ,\Delta)$}\ar@<-3pt>[d]|-{\txt{\small Proposition~{\rm\ref{mpasi}}}}\ar[r]^{\txt{\small deformation}}&*\txt{topologically free\\ASI bialgebra\\$(A_{h},\circ_{h},\Delta_{h})$}\ar[r]^{\txt{\small QCL}}_{\txt{\small Theorem~{\rm\ref{thm-limasi}}}}\ar@<-3pt>[d]|-{\txt{\small Proposition~{\rm\ref{eqvtam}}}}&*\txt{Poisson bialgebra\\$(A,\{,\},\circ,\delta,\Delta)$}\ar@<-3pt>[d]|-{\txt{\small Proposition~{\rm\ref{eqvmpp}}}}\\
    *\txt{matched pair\\of commutative\\associative algebras\\$((A,\circ),(A^{\ast},\cdot),R_{\circ}^{\ast},R_{\cdot}^{\ast})$}\ar[r]^{\txt{\small deformation}}\ar@<-3pt>[u]|-{\phantom{\txt{\small mt~{\rm\ref{mpasi}}}}}
    & *\txt{matched pair\\of topologically free\\associative algebras\\$((A_h,\circ_h),(A^{\star}_h,\cdot_h),$\\$R^{\star}_{\circ_h},L^{\star}_{\circ_h},R^{\star}_{\cdot_h},L^{\star}_{\cdot_h})$}\ar@<-3pt>[u]|-{\phantom{\txt{\small Proposition~\ref{eqvtam}}}}\ar[r]^{\txt{\small QCL}}_{\txt{\small Theorem~{\rm\ref{qclpipei}}}}
    & *\txt{matched pair\\of Poisson algebras\\$((A,\{,\},\circ),(A^{\ast},[,],\cdot),$\\$-\ad_{\{,\}}^{\ast},R_{\circ}^{\ast},-\ad_{[,]}^{\ast},R_{\cdot}^{\ast})$.}\ar@<-3pt>[u]|-{\phantom{\txt{\small Proposition~\ref{eqvmpp}}}}
}\label{diag1}
 \end{gathered}
\end{equation}
\end{thm}
\begin{proof}
        (\ref{qcldadm1}). The first conclusion follows from Theorem~{\ref{eqvdadm}}.
    Since the QCL of $(A_h,\circ_h)$ is $(A,\{,\},\circ)$, we have
    \[\ad_{\{,\}}(x)y\equiv\frac{L_{\circ_h}(x)y-R_{\circ_h}(x)y}{h}\pmod{h},\;\;\forall x,y\in A.\]
    By Proposition~\ref{prodc}~(\ref{cldu}),
      \begin{equation}\label{qclads1}
         \ad^{\ast}_{\{,\}}(x)a^{\ast}\equiv\frac{L^{\star}_{\circ_h}(x)a^{\ast}-R^{\star}_{\circ_h}(x)a^{\ast}}{h}\pmod{h},\;\;\forall x\in A,a^{\ast}\in  A^{\ast}.
      \end{equation}
    On the other hand, for all $a^{\ast},b^{\ast}\in A^{\ast},\;x_h\in A_h$,
    \begin{equation}\label{qclads2}
        \begin{split}
            \langle a^{\ast}\cdot_h b^{\ast}- b^{\ast}\cdot_h a^{\ast},x_h\rangle_{\bf K}
        &=\langle a^{\ast}\hat{\ot}b^{\ast}- b^{\ast}\hat{\ot}a^{\ast},\Delta_h(x_h)\rangle_{\bf K}\\
            &=\langle a^{\ast}\hat{\ot}b^{\ast},\Delta_h(x_h)-\sigma_{\bf K}(\Delta_h(x_h))\rangle_{\bf K}\\
        &\equiv\langle a^{\ast}\hat{\ot}b^{\ast},\Delta_h(x)-\sigma_{\bf K}(\Delta_h(x))\rangle_{\bf K}\pmod{h^2}\\
    &\equiv\langle a^{\ast}\ot b^{\ast},\delta(x)\rangle~h\pmod{h^2}\\
        &\equiv\langle [a^{\ast},b^{\ast}],x\rangle~h\pmod{h^2}\\
            &\equiv\langle [a^{\ast},b^{\ast}]~h,x_h\rangle_{\bf K}\pmod{h^2},
        \end{split}
    \end{equation}
where $x\in A$ satisfies $x_h\equiv x\pmod h$. Thus, the QCL of the associative deformation $(A^{\star}_h,\cdot_h)$ of the commutative algebra $(A^{\ast},\cdot)$ is exactly $(A^{\ast},[,],\cdot)$. So by Proposition~\ref{prodc}~(\ref{cldu}) again, we deduce that
 \begin{equation}\label{qclads3}
\ad^{\ast}_{[,]}(a^{\ast})x\equiv\frac{L^{\star}_{\cdot_h}(a^{\ast})x-R^{\star}_{\cdot_h}(a^{\ast})x}{h}\pmod{h},\;\;\forall x\in A,a^{\ast}\in A^{\ast},
\end{equation}
completing the proof of (\ref{qcldadm1}).

The proof of (\ref{qcldadm2}) is similar to the proof of (\ref{qcldadm1}).
\end{proof}
\begin{rmk}
    As an application of~{\rm \ref{qcldadm}}, we give a simple proof of Theorem~{\rm \ref{thm-limasi}}~in the case that the commutative and cocommutative ASI bialgebra $(A,\circ,\Delta)$ is finite-dimensional. Let $(A_{h},\circ_{h},\Delta_{h})$ be an ASI bialgebra deformation of $(A,\circ,\Delta)$ and $(A,\{,\},\circ,\delta, \Delta)$ be the QCL. Denote by $[,]$ and $\cdot$ the dual operations of $\delta$ and $\Delta$, respectively. Let $\cdot_h$ be the topological dual operation of $\Delta_h$.
    Then by Theorem~{\rm \ref{qcldadm}~(\ref{qcldadm1})}, $((A_h,\circ_h),(A^{\star}_h,\cdot_h),R^{\star}_{\circ_h},L^{\star}_{\circ_h},R^{\star}_{\cdot_h},L^{\star}_{\cdot_h})$ is an associative matched pair
     deformation of $((A,\circ),(A^{\ast},\cdot),R_{\circ}^{\ast},R_{\cdot}^{\ast})$. By Theorem~{\rm \ref{qcldadm}~(\ref{qcldadm2})}, the QCL is $((A,\{,\},\circ),(A^{\ast},[,],\cdot),-\ad_{\{,\}}^{\ast},
     R_{\circ}^{\ast}$, $-\ad_{[,]}^{\ast}$, $R_{\cdot}^{\ast})$, which is a matched pair of Poisson algebras. Thus, by Proposition~\ref{eqvmpp}, $(A,\{,\},\circ,\delta,\Delta)$
     is a Poisson bialgebra.
\end{rmk}

\subsection{Associative matched pair deformations via derivations}\label{sec3.3}

We extend Lemma~\ref{Dda}~to associative matched pair
deformations.
\begin{pro}\label{dmpa}
    Let $((A,\circ), (B, \cdot), l_\circ,r_\circ, l_\cdot, r_\cdot)$ be a matched pair of associative algebras and $D_1,D_2$ be commuting derivations on the
    associative algebra $(A\oplus B,\odot):=A \bowtie_{l_\cdot, r_\cdot}^{l_\circ,r_\circ} B$ such that $D_i(A)\subset A,D_i(B)\subset B$ where $i\in\{1,2\}$. Let $(A_{h},\circ_{h})$ and
     $(B_{h},\cdot_{h})$ be the associative deformations induced by $(A,\circ,D_1,D_2)$ and $(B,\cdot,D_1,D_2)$, respectively. For all $x\in A, v\in B$, define
    \begin{align}
    l_{\circ_h}(x)v &:= \sum_{s=0}^{\infty}l_{\circ}(D_1^{s}(x))D_2^{s}(v)\frac{h^{s}}{s!},
    &r_{\circ_h}(x)v &:= \sum_{s=0}^{\infty}r_{\circ}(D_2^{s}(x))D_1^{s}(v)\frac{h^{s}}{s!},\label{dod}\\
    l_{\cdot_h}(v)x &:= \sum_{s=0}^{\infty}l_{\cdot}(D_1^{s}(v))D_2^{s}(x)\frac{h^{s}}{s!},
    &r_{\cdot_h}(v)x &:= \sum_{s=0}^{\infty}r_{\cdot}(D_2^{s}(v))D_1^{s}(x)\frac{h^{s}}{s!}\label{dcd}.
    \end{align}
    Extend $\mathbf{K}$-linearly the above maps to $A_{h}$ and $B_h$. Then $((A_{h},\circ_{h}),(B_{h},\cdot_{h}),l_{\circ_h},r_{\circ_h},l_{\cdot_h},r_{\cdot_h})$ is
     an associative matched pair deformation of $((A,\circ), (B, \cdot), l_\circ,r_\circ, l_\cdot, r_\cdot)$, called the \textbf{associative matched pair deformation induced by} $((A,\circ),(B, \cdot), l_\circ,r_\circ, l_\cdot, r_\cdot,D_1,D_2)$. Moreover, the associative deformation induced by $(A \bowtie_{l_\cdot, r_\cdot}^{l_\circ,r_\circ} B,D_1,D_2)$ is exactly $A_{h}\bowtie_{l_{\cdot_h},r_{\cdot_h}}^{l_{\circ_h},r_{\circ_h}} B_{h}$.
\end{pro}
\begin{proof}
    Let  $(A_{h}\oplus B_{h},\odot_{h}^{\prime})$ be the associative deformation induced by $(A\oplus B,\odot,D_1,D_2)$.
    Thus, $(A_{h}\oplus B_{h},\odot_{h}^{\prime})$ is a topologically free associative algebra.
    Note that for all $x,y\in A$, $u,v\in B$,
     \begin{eqnarray*}
        (x,u)\odot_{h}^{\prime}(y,v)
        & =&
        \sum_{s=0}^{\infty}(D_1^{s}((x,u))\odot D_2^{s}((y,v)))\frac{h^{s}}{s!}\\
        &=& \sum_{s=0}^{\infty}((D_1^s(x),D_1^s(u))\odot(D_2^s(y),D_2^s(v)))\frac{h^{s}}{s!}\\
        & =&\sum_{s=0}^{\infty}(D_1^{s}(x)\circ D_2^{s}(y)+l_{\cdot}(D_1^{s}(u))D_2^{s}(y)+r_{\cdot}(D_2^{s}(v))D_1^{s}(x) ,0)\frac{h^{s}}{s!}\\
        &&+\sum_{s=0}^{\infty}(0,l_{\circ}(D_1^{s}(x))D_2^{s}(v)+r_{\circ}(D_2^{s}(y))D_1^{s}(u)+D_1^{s}(u)\circ D_2^{s}(v))\frac{h^{s}}{s!} \\
            &=&(x\circ_{h} y+l_{\cdot_h}(u)y+r_{\cdot_h}(v)x,l_{\circ_h}(x)v+r_{\circ_h}(y)u+u\cdot_{h} v).
    \end{eqnarray*}
So
$((A_{h},\circ_{h}),(B_{h},\cdot_{h}),l_{\circ_h},r_{\circ_h},l_{\cdot_h},r_{\cdot_h})$
is a matched pair of topologically free associative algebras. Note
that (\ref{e1}) and (\ref{e2}) hold, then the first conclusion
holds. The last conclusion also follows from the above study.
\end{proof}

Let $V_1$, $V_2$, $W_1$ and $W_2$ be vector spaces. Let $f:V_1\to
W_1$, $g:V_2\to W_2$ be linear maps. Define a linear map
$f\dotplus g:V_1\oplus V_2\to W_1\oplus W_2$ by
\[f\dotplus g((v_1,v_2)):=(f(v_1),g(v_2)),\;\;\forall v_1\in V_1, v_2\in V_2.\]
\delete{
\begin{lem}{\rm{(\cite{Lin})}}\label{led}
    Let $((A,\circ), (B, \cdot), l_\circ,r_\circ, l_\cdot, r_\cdot)$ be a matched pair of associative algebras, $d$ and $d^\prime$ be derivations on associative algebras $(A,\circ)$ and $(B,\cdot)$,respectively. Denote $d\dotplus d^\prime$ by $D$. Then $D$ is a derivation on the associative algebra $A \bowtie_{l_\cdot, r_\cdot}^{l_\circ,r_\circ} B$ if and only if for all $x\in A$, $v\in B$,
    \begin{align}
        d^{\prime}(l_{\circ}(x)v)&=l_{\circ}(d(x))v+l_{\circ}(x)d^{\prime}(v),\;\;&d^{\prime}(r_{\circ}(x)v)&=r_{\circ}(d(x))v+r_{\circ}(x)d^{\prime}(v),\\
            d(l_{\cdot}(v)x)&=l_{\cdot}(d^{\prime}(v))x+l_{\cdot}(v)d^\prime(x),\;\;    &d(r_{\cdot}(v)x)&=r_{\cdot}(d^{\prime}(v))x+r_{\cdot}(v)d^\prime(x).
    \end{align}
\end{lem}}
We give a proof of the following result using the method of
deformations-QCLs.
\begin{cor}{\rm{(\cite{Lin})}}\label{qclddmp}
    Let $((A,\circ),(B, \cdot),\rho_{\circ},\rho_{\cdot})$ be a matched pair of commutative associative algebras, $d_1,d_2$ and $d_1^\prime,d_2^\prime$ be commuting derivations on $(A,\circ)$ and $(B,\cdot)$, respectively. Let $(A,\{,\},\circ)$ and $(B, [,],\cdot)$ denote the Poisson algebras induced by $(A,\circ,d_1,d_2)$ and $(B,\cdot,d_1^\prime,d_2^\prime)$, respectively.
    For all $x\in A$, $v\in B$, define linear maps $\rho_{\{,\}}:A\to\End_{\bf k}(B)$, $\rho_{[,]}:B\to\End_{\bf k}(A)$ by
    \begin{eqnarray}
    \rho_{\{,\}}(x)v&:=&\rho_{\circ}(d_1(x))d_{2}^{\prime}(v)-\rho_{\circ}(d_2(x))d_{1}^{\prime}(v),\label{qcldo}\\
    \rho_{[,]}(v)x&:=&\rho_{\cdot}(d_{1}^{\prime}(v))d_2(x)-\rho_{\cdot}(d_{2}^{\prime}(v))d_1(x).
    \end{eqnarray}
     If $d_i\dotplus d_i^{\prime}$ is a derivation on $A \bowtie_{\rho_\cdot}^{\rho_\circ} B$ for all $i\in\{1,2\}$, then $((A,\{,\},\circ),(B,[,] ,\cdot),\rho_{\{,\}},\rho_{\circ},\rho_{[,]},\rho_{\cdot})$ is a matched pair of Poisson algebras, called the \textbf{matched pair of Poisson algebras induced by} $(A,\circ,d_1,d_2, B,\cdot,d_1^\prime,d_2^\prime, \rho_{\circ},\rho_{\cdot})$ .
    Moreover, the Poisson algebra $A\bowtie_{\rho_{[,]},\rho_{\cdot}}^{\rho_{\{,\}},\rho_{\circ}} B$ is exactly the Poisson algebra induced by $(A \bowtie_{\rho_\cdot}^{\rho_\circ} B,D_1,D_2)$, where $D_i:=d_i\dotplus d_i^\prime$, $i=1,2$.
\end{cor}
\begin{proof}
Denote by $(A_{h},\circ_{h})$ and $(B_{h},\cdot_{h})$ the
associative deformations induced by $(A,\circ,d_1,d_2)$ and
$(B,\cdot,d_1^\prime,d_2^\prime)$, respectively. Define $\bf
K$-linear maps $l_{\circ_h},r_{\circ_h}:
    A_{h}\rightarrow \mathrm{End}_{\mathbf{K}}(B_{h})$ and $l_{\cdot_h},r_{\cdot_h}:
    B_{h}\rightarrow \mathrm{End}_{\mathbf{K}}(A_{h})$ by~(\ref{dod})~and~(\ref{dcd}), respectively.
    Then by Proposition~\ref{dmpa}, $((A_{h},\circ_{h}),(B_{h},\cdot_{h}),l_{\circ_h},r_{\circ_h},l_{\cdot_h},r_{\cdot_h})$ is an associative matched pair deformation of $((A,\circ),(B, \cdot),\rho_{\circ},\rho_{\cdot})$. Note that
\begin{eqnarray*}
    \frac{l_{\circ_h}(x)v-r_{\circ_h}(x)v}{h}&=&\sum_{s=0}^{\infty}\frac{(\rho_{\circ}(D_1^{s}(x))D_2^{s}(v)-\rho_{\circ}(D_2^{s}(x))D_1^{s}(v))h^{s}}{s!h}\\
    &\equiv&\rho_{\circ}(d_1(x))d_{2}^{\prime}(v)-\rho_{\circ}(d_2(x))d_{1}^{\prime}(v)\pmod h, \forall x\in A, v\in B.
\end{eqnarray*}
Thus, we obtain (\ref{cmod}). Similarly, we obtain (\ref{cmod2}).
Therefore, $((A,\{,\},\circ),(B,[,]
,\cdot),\rho_{\{,\}},\rho_{\circ},\rho_{[,]},\rho_{\cdot})$ is
exactly the QCL of
$((A_{h},\circ_{h}),(B_{h},\cdot_{h}),l_{\circ_h},r_{\circ_h},l_{\cdot_h},r_{\cdot_h})$.
By Theorem~\ref{qclpipei}, we show that it is a matched pair of
Poisson algebras. The last conclusion also follows from
Theorem~\ref{qclpipei}.
\end{proof}

Recall that there is an equivalent characterization of coherent
derivations on an ASI bialgebra.
\begin{pro}
    {\rm{(\cite{Lin})}}\label{proecdd}
    Let $(A,\circ,\Delta)$ be a finite-dimensional ASI bialgebra, $d$ be a derivation on $(A,\circ)$, $\dc$ be a coderivation on $(A,\Delta)$ and  $\cdot:A\ot A\to A$ be the dual operation of $\Delta$.
    Then $(d,\dc)$ are coherent derivations on $(A,\circ,\Delta)$ if and only if $d\dotplus\dc^{\ast}$ is a derivation on $A\bowtie_{R^{\ast}_{\cdot},L^{\ast}_{\cdot}}^{R^{\ast}_{\circ},L^{\ast}_{\circ}}A^{\ast}$.
\end{pro}

Let $(A,\circ,\Delta)$ be a finite-dimensional ASI bialgebra,
$(d_1,\dc_1)$ and $(d_2,\dc_2)$ be coherent derivations on $(A,
\circ, \Delta)$. Suppose that $d_1$ commutes with $d_2$ and
$\dc_1$ commutes with $\dc_2$.  Let $\cdot:A\ot A\to A$ be the
dual operation of $\Delta$. By Proposition~\ref{proecdd},
$d_1\dotplus\dc^{\ast}_1$ and $d_2\dotplus\dc^{\ast}_2$ are
derivations on
$A\bowtie_{R^{\ast}_{\cdot},L^{\ast}_{\cdot}}^{R^{\ast}_{\circ},L^{\ast}_{\circ}}A^{\ast}$.
It is clear that they commute. Then, by Proposition~\ref{dmpa},
$((A,\circ),(A^{\ast},\cdot),R_{\circ}^{\ast},L_{\circ}^{\ast},R_{\cdot}^{\ast},L_{\cdot}^{\ast},d_1\dotplus\dc^{\ast}_1,d_2\dotplus\dc^{\ast}_2)$
induces an associative matched pair deformation of
$((A,\circ),(A^{\ast},\cdot),R_{\circ}^{\ast},L_{\circ}^{\ast},R_{\cdot}^{\ast},L_{\cdot}^{\ast})$.
We denote it by
$((A_h,\circ_h),(A^{\star}_h,\cdot_h),l_{\circ_h},r_{\circ_h},l_{\cdot_h},r_{\cdot_h})$.
\begin{pro}
    \label{prodmcdasi}
    Keep the above assumptions and notations. We suppose that {\rm(\ref{ed1})}, {\rm(\ref{ed2})}, {\rm(\ref{ed3})} and {\rm(\ref{ed4})} hold. Let $(A_h, \circ_h, \Delta_h)$ be the
    ASI bialgebra deformation induced by $(A,\circ, d_1,d_2,\Delta,\dc_1,\dc_2)$. Then the matched pair of topologically free associative algebras corresponding to $(A_h, \circ_h, \Delta_h)$
    coincides with  $((A_h,\circ_h),(A^{\star}_h,\cdot_h),l_{\circ_h},r_{\circ_h},l_{\cdot_h},r_{\cdot_h})$.
\end{pro}
\begin{proof}
   Let $\cdot_h^{\prime}:A_h\hat{\ot} A_h\to A_h$ be the topological dual operation of $\Delta_h$. Let $a^{\ast},b^{\ast}\in A^{\ast},\;x\in
   A$. Then we have
    \begin{eqnarray*}
        \langle a^{\ast}\cdot^{\prime}_h b^{\ast},x\rangle_{\bf K}
        &=&\langle a^{\ast}\hat{\ot}b^{\ast},\Delta_h(x)\rangle_{\bf K}
        =\langle a^{\ast}\hat{\ot}b^{\ast},\sum_{s=0}^{\infty}(\dc_1^{s}\ot\dc^{s}_{2})\Delta(x)\frac{h^s}{s!}\rangle_{\bf K}\\
        &=&\langle \sum_{s=0}^{\infty} \dc_1^{\ast s}(a^{\ast})\ot\dc_2^{\ast s}(b^{\ast})\frac{h^s}{s!},\Delta(x)\rangle_{\bf K}
        =\langle \sum_{s=0}^{\infty} \dc_1^{\ast s}(a^{\ast})\ot\dc_2^{\ast s}(b^{\ast}),\Delta(x)\rangle\frac{h^s}{s!}\\
        &=&\langle \sum_{s=0}^{\infty} \dc_1^{\ast s}(a^{\ast})\cdot\dc_2^{\ast s}(b^{\ast}),x\rangle\frac{h^s}{s!}
        =\langle \sum_{s=0}^{\infty} \dc_1^{\ast s}(a^{\ast})\cdot\dc_2^{\ast s}(b^{\ast})\frac{h^s}{s!},x\rangle_{\bf K}
            =\langle a^{\ast}\cdot_h b^{\ast},x\rangle_{\bf K}.
    \end{eqnarray*}
Thus, by the $\bfK$-bilinearity of $\langle,\rangle_{\bf K}$,
$\cdot_h$ coincides with $\cdot_h^{\prime}$. So, the matched pair
of topologically free associative algebras corresponding to
$(A_h,\circ_h, \Delta_h)$ is
$((A_h,\circ_h),(A^{\star}_h,\cdot_h),R^{\star}_{\circ_h},L^{\star}_{\circ_h},R^{\star}_{\cdot_h},L^{\star}_{\cdot_h})$.

Moreover, we have
\begin{eqnarray*}
    \langle l_{\circ_h}(x)a^{\ast},y\rangle_{\bf K}
    &=&\langle \sum_{s=0}^{\infty}R^{\ast}_{\circ}(d_1^{s}(x))\dc_2^{\ast s}(a^{\ast})\frac{h^{s}}{s!},y\rangle_{\bf K}
    =\langle \dc_2^{\ast s}(a^{\ast}),\sum_{s=0}^{\infty}R^{\ast}_{\circ}(d_1^{s}(x))y\frac{h^{s}}{s!}\rangle_{\bf K}\\
    &=&\langle a^{\ast},\sum_{s=0}^{\infty}\dc_{2}^{s}(y\circ d_1^{s}(x))\frac{h^{s}}{s!}\rangle_{\bf K}
    \overset{(\ref{ped1})}{=}\langle a^{\ast},\sum_{s=0}^{\infty}\dc_{2}^{s}(y\circ d_1^{s}(x))\frac{h^{s}}{s!}\rangle_{\bf K}\\
    &=&\langle a^{\ast},\sum_{s=0}^{\infty}d_{1}^{s}(y)\circ d_2^{s}(x)\frac{h^{s}}{s!}\rangle_{\bf K}
    =\langle a^{\ast},y\circ_h x\rangle_{\bf K}
    =\langle R^{\star}_{\circ_h}(x)a^{\ast},y\rangle_{\bf K}.
\end{eqnarray*}
By the $\bfK$-bilinearity of $\langle,\rangle_{\bf K}$, we have
$l_{\circ_h}=R^{\star}_{\circ_h}$. After a similar discussion,
by~(\ref{ped2}), we have $r_{\circ_h}=L^{\star}_{\circ_h}$;
by~(\ref{ped3}), we have $l_{\cdot_h}=R^{\star}_{\cdot_h}$;
 by~(\ref{ped4}),
we have $r_{\cdot_h}=L^{\star}_{\cdot_h}$. Hence the conclusion
holds. \delete{ For all $x\in A,\;a^{\ast},b^{\ast}\in A^{\ast}$,
\begin{eqnarray*}
    \langle b^{\ast},l_{\cdot_h}(a^{\ast})x\rangle_{\bf K}
    &=&\langle b^{\ast},\sum_{s=0}^{\infty}R^{\ast}_{\cdot}(\dc_1^{\ast s}(a^{\ast})d_2^{s}(x)\frac{h^{s}}{s!}\rangle_{\bf K}
    =\sum_{s=0}^{\infty}\langle b^{\ast}\cdot\dc_1^{\ast s}(a^{\ast})\frac{h^{s}}{s!},d_2^{s}(x)\rangle_{\bf K}\\
    &=&\sum_{s=0}^{\infty}\langle b^{\ast}\cdot\dc_1^{\ast s}(a^{\ast}),d_2^{s}(x)\rangle\frac{h^{s}}{s!}
    =\sum_{s=0}^{\infty}\langle b^{\ast}\ot\dc_1^{\ast s}(a^{\ast}),\Delta(d_2^{s}(x))\rangle\frac{h^{s}}{s!}\\
        &=&\sum_{s=0}^{\infty}\langle b^{\ast}\ot a^{\ast},(\id_{A}\ot\dc_1^{s})\Delta(d_2^{s}(x))\rangle\frac{h^{s}}{s!}
    \overset{(\ref{ped3})}{=}\langle b^{\ast}\ot a^{\ast},\sum_{s=0}^{\infty}(\dc_1^{s}\ot\dc_2^{s})\Delta(x)\frac{h^{s}}{s!}\rangle_{\bf K}\\
    &=&\langle b^{\ast}\hat{\ot} a^{\ast},\Delta_h(x)\rangle_{\bf K}
    =\langle b^{\ast}\cdot_h a^{\ast}, x\rangle_{\bf K}
    =\langle b^{\ast},R_{\cdot_h}^{\star}(a^{\ast})x\rangle_{\bf K}.
\end{eqnarray*}
Therefore, $l_{\cdot_h}=R^{\star}_{\cdot_h}$. Similarly,
by~(\ref{ped4}), we deduce that $r_{\cdot_h}=L^{\star}_{\cdot_h}$,
which completes the proof.}
\end{proof}
\begin{rmk}
    By the above conclusion, we see that $((A_h,\circ_h),(A^{\star}_h,\cdot_h),R^{\star}_{\circ_h},L^{\star}_{\circ_h},R^{\star}_{\cdot_h},L^{\star}_{\cdot_h})$ coincides with the matched pair of topologically free associative algebras $((A_h,\circ_h),(A^{\star}_h,\cdot_h),l_{\circ_h},r_{\circ_h},l_{\cdot_h},r_{\cdot_h})$. Then we complete another proof of Theorem~{\rm\ref{ddasi}} in the case of the ASI bialgebra therein being finite-dimensional by Proposition~{\rm\ref{eqvtam}}.
\end{rmk}
By Proposition~\ref{prodmcdasi}~and Theorem~\ref{qcldadm}, we obtain the following result.
\begin{cor}
    Let $(d_1,\dc_1)$ and $(d_2,\dc_2)$ be coherent derivations on a finite-dimensional commutative and cocommutative ASI bialgebra $(A, \circ, \Delta)$. Suppose that $d_1$ commutes with $d_2$, $\dc_1$ commutes with $\dc_2$ and {\rm(\ref{ed1})}, {\rm(\ref{ed2})}, {\rm(\ref{ed3})} and {\rm(\ref{ed4})} hold. Then, let $(A_h, \circ_h, \Delta_h)$ be the ASI bialgebra deformation induced by $(A,\circ, d_1,d_2,\Delta,\dc_1,\dc_2)$ and $(A,\{,\},\circ,\delta,\Delta)$ be the QCL.
     Let $[,],\cdot:A\ot A\to A$ be the dual maps of $\delta,\Delta$. Let $((A_h,\circ_h),(A^{\star}_h,\cdot_h),l_{\circ_h},r_{\circ_h},l_{\cdot_h},r_{\cdot_h})$ be the
     associative matched pair deformation induced by $((A,\circ),(A^{\ast},\cdot),R_{\circ}^{\ast},L_{\circ}^{\ast},R_{\cdot}^{\ast},L_{\cdot}^{\ast},d_1\dotplus\dc^{\ast}_1,d_2\dotplus\dc^{\ast}_2)$.
Then the matched pair of topologically free associative algebras
corresponding to $(A_h, \circ_h, \Delta_h)$ coincides with
$((A_h,\circ_h),(A^{\star}_h,\cdot_h),l_{\circ_h},r_{\circ_h},l_{\cdot_h},r_{\cdot_h})$.
Moreover, the matched pair of Poisson algebras corresponding to
$(A,\{,\},\circ,\delta,\Delta)$ is exactly the QCL of
$((A_h,\circ_h),(A^{\star}_h,\cdot_h),l_{\circ_h},r_{\circ_h},l_{\cdot_h},r_{\cdot_h})$.
In other words, the following diagram commutes.
    \begin{displaymath}
    \xymatrixcolsep{6pc}\xymatrixrowsep{6pc}
    
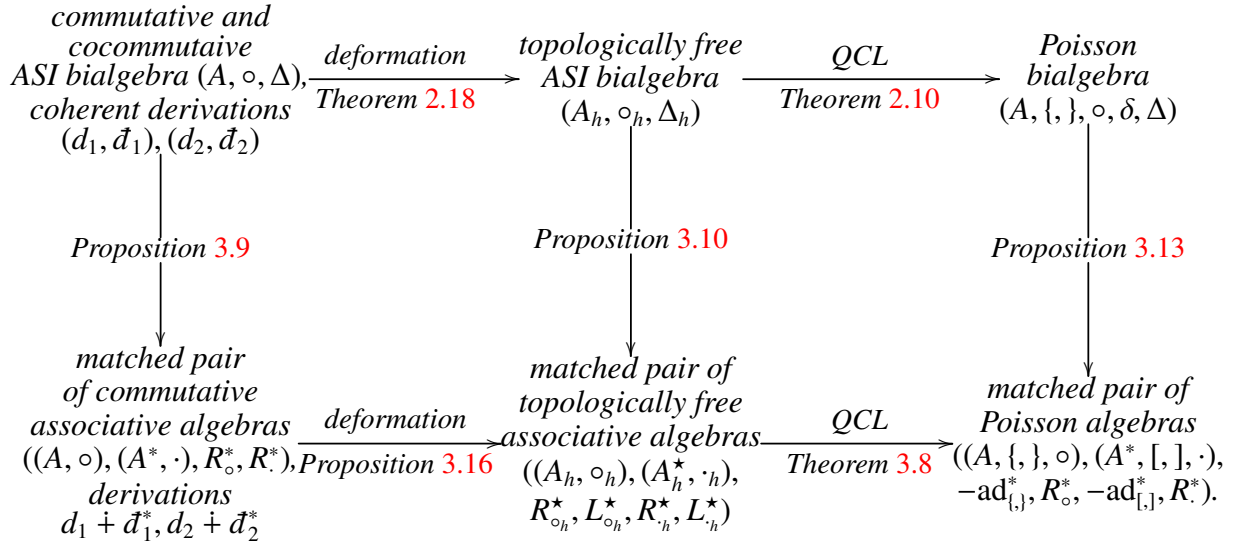
\xymatrix{
        *\txt{commutative and\\cocommutaive\\ ASI bialgebra $(A,\circ,\Delta)$,\\coherent derivations\\$(d_1,\dc_1),(d_2,\dc_2)$}\ar[d]|-{\txt{\small Proposition~{\rm\ref{mpasi}}}}\ar[r]^{\txt{\small deformation}}_{\txt{\small Theorem~{\rm\ref{ddasi}}}}&*\txt{topologically free\\ASI bialgebra\\$(A_{h},\circ_{h},\Delta_{h})$}\ar[r]^{\txt{\small QCL}}_{\txt{\small Theorem~{\rm\ref{thm-limasi}}}}\ar[d]|-{\txt{\small Proposition~{\rm\ref{eqvtam}}}}&*\txt{Poisson\\bialgebra\\$(A,\{,\},\circ,\delta,\Delta)$}\ar[d]|-{\txt{\small Proposition~{\rm\ref{eqvmpp}}}}\\
        *\txt{matched pair\\of commutative\\associative algebras\\$((A,\circ),(A^{\ast},\cdot),R_{\circ}^{\ast},R_{\cdot}^{\ast})$,\\derivations\\$d_1\dotplus\dc^{\ast}_1,d_2\dotplus\dc^{\ast}_2$}\ar[r]^{\txt{\small deformation}}_{\txt{\small Proposition~{\rm\ref{dmpa}}}}
        & *\txt{matched pair of\\topologically free\\associative algebras\\$((A_h,\circ_h),(A^{\star}_h,\cdot_h),$\\$R^{\star}_{\circ_h},L^{\star}_{\circ_h},R^{\star}_{\cdot_h},L^{\star}_{\cdot_h})$}\ar[r]^{\txt{\small QCL}}_{\txt{\small Theorem~{\rm\ref{qclpipei}}}}
        & *\txt{matched pair of\\Poisson algebras\\$((A,\{,\},\circ),(A^{\ast},[,],\cdot),$\\$-\ad_{\{,\}}^{\ast},R_{\circ}^{\ast},-\ad_{[,]}^{\ast},R_{\cdot}^{\ast})$.}
    }
    \end{displaymath}
\end{cor}


\begin{rmk}
    Conversely, starting from a matched pair of commutative associative algebras $((A,\circ),(A^{\ast},\cdot),R_{\circ}^{\ast},L_{\circ}^{\ast},R_{\cdot}^{\ast},L_{\cdot}^{\ast})$ and commuting derivations $d_1\dotplus\dc^{\ast}_1,d_2\dotplus\dc^{\ast}_2$ on $A\bowtie_{R^{\ast}_{\cdot}}^{R^{\ast}_{\circ}}A^{\ast}$ where $d_1,\dc_1,d_2,\dc_2:A\to A$ are linear maps and $\dc^{\ast}_1,\dc^{\ast}_2$ are linear dual maps of $\dc_1,\dc_2$, respectively, and assuming that {\rm(\ref{ed1})}, {\rm(\ref{ed2})}, {\rm(\ref{ed3})} and {\rm(\ref{ed4})} hold, we can establish upward arrows in the above diagram.
\end{rmk}
\section{Associative Manin triple deformations}\label{S4}
We study the associative Manin triple deformation and show that
the corresponding quasiclassical limit is a standard Manin triple
of Poisson algebras. Thus, we obtain another equivalent
characterization of the deformations-quasiclassical limits process
for a Poisson bialgebra.
\subsection{Associative Manin triple deformations and the quasiclassical limits}
\begin{defi}
    Let $(A, \cdot)$ be an associative algebra. Then a bilinear form $\mathfrak{B}$ on $(A, \cdot)$ is \textbf{invariant} (or \textbf{associative} \cite{SY}) if
    \begin{equation*}
    \mathfrak{B}(x\cdot y,z)=\mathfrak{B}(x,y\cdot z),\;\;\forall x,y,z\in A.
    \end{equation*}

    Let $(A_h,\cdot_h)$ be a topologically free associative algebra. Then a $\bfK$-bilinear form $\mathfrak{B}_h$ on $(A_h,\cdot_h)$ is \textbf{invariant} if
    \begin{equation*}
    \mathfrak{B}_h(x_h\cdot_h y_h,z_h)=\mathfrak{B}_h(x_h, y_h\cdot_h z_h),\;\;\forall x_h, y_h, z_h\in A_h.
    \end{equation*}

   Let $(P,\{,\},\circ)$ be a Poisson algebra. Then a bilinear form $\mathfrak{B}$ on $(P,\{,\},\circ)$ is \textbf{invariant} if
    \begin{equation*}
    \mathfrak{B}(x\circ y,z)=\mathfrak{B}(x,y\circ z),\;\;\mathfrak{B}(\{x, y\},z)=\mathfrak{B}(x,\{y,z\}),\;\;\forall x,y,z\in P.
    \end{equation*}
\end{defi}

    Let $\mathfrak{B}$ be a bilinear form on a vector space $V$. Then the $\bfK$-bilinear form $\mathfrak{B}_{\bfK}:V_h\hat{\ot} V_h\to\bfK$ is defined by
\begin{equation}\label{scb}
    \mathfrak{B}_{\bfK}(\sum_{s=0}^{\infty}x_s~h^s,\sum_{s=0}^{\infty}y_s~h^s):=\sum_{l=0}^{\infty}\sum_{s=0}^{l}\mathfrak{B}(x_s,y_{l-s})~h^l,\;\; \forall\sum_{s=0}^{\infty}x_s~h^s,\sum_{s=0}^{\infty}y_s~h^s\in V_h.
\end{equation}
\begin{thm}
    \label{dinv}
    Let $(A_h,\circ_h)$ be an associative deformation of a commutative associative algebra $(A,\circ)$ and $(A,\{,\},\circ)$ be the QCL. Let $\mathfrak{B}$ be a symmetric invariant bilinear form on $(A,\circ)$. If $\mathfrak{B}_{\bfK}$ is an invariant $\bfK$-bilinear form on $(A_h,\circ_h)$, then $\mathfrak{B}$ is an invariant bilinear form on $(A,\{,\},\circ)$.
\end{thm}
\begin{proof}
        Let $x,y,z\in A$. Then we have
        \begin{equation}\label{inv1}
  \mathfrak{B}_{\bfK}(x\circ_{h}y,z)=\mathfrak{B}_{\bfK}(x,y\circ_{h}z).
    \end{equation}
Interchanging the two sides of the equality in (\ref{inv1}) and then interchanging $x$ and $z$, we obtain
    \begin{equation}\label{inv2}
\mathfrak{B}_{\bfK}(z,y\circ_{h}x)=\mathfrak{B}_{\bfK}(z\circ_{h}y,x).
\end{equation}
    Since $\mathfrak{B}$ is symmetric, by (\ref{scb}), we show that
\begin{equation}\label{dcb}
\mathfrak{B}_{\mathbf{K}}(x_h,y_h)=\mathfrak{B}_{\mathbf{K}}(y_h,x_h),\;\;\forall x_h,y_h\in A_h,
\end{equation}
that is, $\mathfrak{B}_{\bfK}$ is a symmetric $\bfK$-bilinear form on $(A_h,\circ_h)$.
Applying the symmetry of $\mathfrak{B}_{\bfK}$ to (\ref{inv2}) yields
\begin{equation}\label{inv3}
\mathfrak{B}_{\bfK}(y\circ_{h}x,z)=\mathfrak{B}_{\bfK}(x,z\circ_{h}y).
\end{equation}
Subtracting (\ref{inv3}) from (\ref{inv1}) and taking the first order terms of $h$, we get
\begin{equation*}
    \mathfrak{B}(\{x, y\},z)=\mathfrak{B}(x,\{y,z\}),\;\;\forall x,y,z\in A.
\end{equation*}
We complete the proof by noting that $\mathfrak{B}$ is an invariant bilinear form on $(A,\circ)$.
\end{proof}
We recall the notion of a standard Manin triple of associative
algebras, which is exactly  the notion of a double construction of
Frobenius algebra in \cite{Bai}.
\begin{defi}
    Let $(A,\circ)$ and $(A^{\ast},\cdot)$ be subalgebras of an associative algebra $(A\oplus A^{\ast},\odot)$. Define a bilinear form $\mathfrak{B}$ on $(A\oplus A^{\ast},\odot)$ by
    \begin{equation}\label{dsdb}
        \mathfrak{B}(x+a^{\ast},y+b^{\ast}):=\langle a^{\ast},y\rangle+\langle b^{\ast},x\rangle,\;\;\forall x,y\in A,\;a^{\ast},b^{\ast}\in A^{\ast}.
    \end{equation}
    If $\mathfrak{B}$ is an invariant bilinear form on $(A\oplus A^{\ast},\odot)$, then we call the triple $((A\oplus A^{\ast},\odot),A, A^{\ast})$ a \textbf{standard Manin triple of associative algebras}. Moreover, if $(A\oplus A^{\ast},\odot)$ is commutative, then we call
    $((A\oplus A^{\ast},\odot),A, A^{\ast})$ a \textbf{standard Manin triple of commutative associative algebras}.
\end{defi}
\begin{defi}
        Let $(A_h,\circ_h)$ and $(A^{\star}_h,\cdot_h)$ be subalgebras of a topologically free associative algebra $(A_h\oplus A^{\star}_h,\odot_h)$. Define a bilinear form $\mathfrak{B}_h$ on $(A_h\oplus A^{\star}_h,\odot_h)$ by
    \begin{equation}\label{dsdbh}
    \mathfrak{B}_h(x_h+a^{\star},y_h+b^{\star}):=\langle a^{\star},y_h\rangle_{\bfK}+\langle b^{\star},x_h\rangle_{\bfK},\;\;\forall x_h,y_h\in A_h,\;a^{\star},b^{\star}\in A^{\star}_h.
    \end{equation}
    If $\mathfrak{B}_{h}$ is an invariant $\bfK$-bilinear form on $(A_h\oplus A^{\star}_h,\odot_h)$, then we call $((A_h\oplus A^{\star}_h,\odot_h),A_h, A^{\star}_h)$ a \textbf{standard Manin triple of topologically free associative algebras}.

        Let $((A\oplus A^{\ast},\odot),A, A^{\ast})$ be a standard Manin triple of associative algebras and let $((A_h\oplus A^{\star}_h,\odot_h)$, $A_h, A^{\star}_h)$ be a standard Manin triple of topologically free associative algebras. If $(A_h\oplus A^{\star}_h,\odot_h)$ is an associative deformation of $(A\oplus A^{\ast},\odot)$, then $((A_h\oplus A^{\star}_h,\odot_h),A_h, A^{\star}_h)$ is called an \textbf{associative Manin triple deformation} of $((A\oplus A^{\ast},\odot),A, A^{\ast})$. In this case, if $(A\oplus A^{\ast},\odot)$ is commutative and $(A\oplus A^{\ast},\llbracket,\rrbracket,\odot)$ is the QCL of $(A_h\oplus A^{\star}_h,\odot_h)$, then we call $((A\oplus A^{\ast},\llbracket,\rrbracket,\odot),A, A^{\ast})$ the \textbf{quasiclassical limit (QCL) of the associative Manin triple deformation} $((A_h\oplus A^{\star}_h,\odot_h),A_h, A^{\star}_h)$.
\end{defi}
We then recall the notion of a standard Manin triple of Poisson algebras.
\begin{defi}{\rm(\cite{NB})}
 Let $(A,\{,\},\circ)$ and $(A^{\ast},[,],\cdot)$ be subalgebras of a Poisson algebra $(A\oplus A^{\ast},\llbracket,\rrbracket,\odot)$. Define a bilinear form $\mathfrak{B}$ on $(A\oplus A^{\ast},\llbracket,\rrbracket,\odot)$ by {\rm(\ref{dsdb})}.
    If $\mathfrak{B}$ is an invariant bilinear form on $(A\oplus A^{\ast},\llbracket,\rrbracket,\odot)$, then we call
    $((A\oplus A^{\ast},\llbracket,\rrbracket,\odot),A, A^{\ast})$ a \textbf{standard Manin triple of Poisson algebras}.
\end{defi}
\begin{thm}
    \label{qcldmt}
   Let $((A_h\oplus A^{\star}_h,\odot_h),A_h, A^{\star}_h)$ be an associative Manin triple deformation of a standard Manin triple of  commutative associative algebras $((A\oplus A^{\ast},\odot),A, A^{\ast})$ and $((A\oplus A^{\ast},\llbracket,\rrbracket,\odot),A, A^{\ast})$ be the QCL. Then $((A\oplus A^{\ast},\llbracket,\rrbracket,\odot),A, A^{\ast})$ is a standard Manin triple of Poisson algebras.
\end{thm}
\begin{proof}
    Since $(A_h,\odot_h)$ and $(A^{\star}_h,\odot_h)$ are subalgebras of
 $(A_h\oplus A^{\star}_h,\odot_h)$, $(A,\llbracket,\rrbracket,\odot)$ and $(A^{\ast},\llbracket,\rrbracket,\odot)$ are subalgebras of $(A\oplus A^{\ast},\llbracket,\rrbracket,\odot)$. Define a bilinear form $\mathfrak{B}$ on $(A\oplus A^{\ast},\odot)$ and a $\bfK$-bilinear form $\mathfrak{B}_h$ on $(A_h\oplus A^{\star}_h,\odot_h)$ by (\ref{dsdb}) and (\ref{dsdbh}), respectively. Then for all $x,y\in A,\;a^{\ast},b^{\ast}\in A^{\ast}$,
    \[\mathfrak{B}_h(x+a^{\ast},y+b^{\ast})=\langle a^{\ast},y\rangle_{\bfK}+\langle b^{\ast},x\rangle_{\bfK}=\langle a^{\ast},y\rangle+\langle b^{\ast},x\rangle=  \mathfrak{B}(x+a^{\ast},y+b^{\ast}).\]
    Thus $\mathfrak{B}_h$ and  $\mathfrak{B}_{\bfK}$  coincide on $(A\oplus A^\ast)^{\ot 2}$. By their $\bfK$-bilinearity, $\mathfrak{B}_h=\mathfrak{B}_{\bfK}$. Due to Theorem~\ref{dinv}, $\mathfrak{B}$
    is an invariant bilinear form on $(A\oplus A^{\ast},\llbracket,\rrbracket,\odot)$, proving the conclusion.
\end{proof}
\subsection{Associative Manin triple deformations and antisymmetric infinitesimal bialgebra deformations}
\begin{pro}
    {\rm(\cite{Bai})}\label{eqvbasi}
    Let $A$ be a finite-dimensional vector space, $\circ$ be a binary operation on $A$, $\Delta:A\to A\ot A$ be a linear map and
$\cdot$  be the dual operation of $\Delta$. Define a binary operation $\odot$ on $A\oplus A^{\ast}$ by
    \begin{equation}\label{bod}
        (x+a^{\ast})\odot(y+b^{\ast}):=x\circ y+R_{\cdot}^{\ast}(a^{\ast})y+L_{\cdot}^{\ast}(b^{\ast})x+R_{\circ}^{\ast}(x)b^{\ast}+L_{\circ}^{\ast}(y)a^{\ast}+ a^{\ast}\cdot b^{\ast},
    \end{equation}
   for all $x,y\in A,\;a^{\ast},b^{\ast}\in A^{\ast}$. Then $(A,\circ,\Delta)$ is an ASI bialgebra if and only if $((A\oplus A^{\ast},\odot),A,A^{\ast})$ is a standard Manin triple
    of associative algebras. \delete{In this case, we call $(A,\circ,\Delta)$ the \textbf{ASI bialgebra corresponding to} the standard Manin triple of associative algebras $((A\oplus A^{\ast},\odot),A,A^{\ast})$
    and  $((A\oplus A^{\ast},\odot),A,A^{\ast})$ the \textbf{standard Manin triple of associative algebras corresponding to} the ASI bialgebra $(A,\circ,\Delta)$.}
    Moreover, $(A,\circ,\Delta)$ is a commutative and cocommutative ASI bialgebra if and only if $((A\oplus A^{\ast},\odot),A,A^{\ast})$ is a standard Manin triple of commutative associative algebras.
    \delete{In this case, we call $(A,\circ,\Delta)$ the \textbf{commutative and cocommutative ASI bialgebra corresponding to} the standard Manin triple of commutative associative algebras
     $((A\oplus A^{\ast},\odot),A,A^{\ast})$ and $((A\oplus A^{\ast},\odot),A,A^{\ast})$ the \textbf{standard Manin triple of commutative associative algebras corresponding to} the
      commutative and cocommutative ASI bialgebra $(A,\circ,\Delta)$.} 
\end{pro}

Using the method that proves Theorem~\ref{eqvtam}, we obtain the following result by Proposition~\ref{eqvbasi}.
\begin{pro}
\label{eqvbasih}
Let $A$ be a finite-dimensional vector space, $\circ_h$ be a $\bf K$-bilinear operation on $A_h$, $\Delta_h:A_h\to A_h\hat{\ot} A_h$ be a $\bf K$-linear map and $\cdot_h$ be the topological dual operation of $\Delta_h$. Define a binary operation $\odot_h$ on $A_h\oplus A^{\star}_h$ by
    \begin{equation}\label{bodh}
    (x_h+a^{\star})\odot_h(y_h+b^{\star}):=x_h\circ_h y_h+R_{\cdot_h}^{\star}(a^{\star})y_h+L_{\cdot_h}^{\star}(b^{\star})x_h+R_{\circ_h}^{\star}(x_h)b^{\star}+L_{\circ_h}^{\star}(y_h)a^{\star}+ a^{\star}\cdot_h b^{\star},
    \end{equation}
   for all $x_h,y_h\in A_h,\;a^{\star},b^{\star}\in A^{\star}_h$.
Then $(A_h,\circ_h,\Delta_h)$ is a topologically free ASI
bialgebra if and only if $((A_h\oplus A^{\star}_h,\odot_h),A_h,
A^{\star}_h)$ is a standard Manin triple of topologically free
associative algebras. \delete{In this case, we call
$(A_h,\circ_h,\Delta_h)$ the \textbf{topologically free ASI
bialgebra corresponding to} the standard Manin triple of
topologically free associative algebras $((A_h\oplus
A^{\star}_h,\odot_h),A_h, A^{\star}_h)$ and $((A_h\oplus
A^{\star}_h,\odot_h),A_h, A^{\star}_h)$ the \textbf{standard Manin
triple of topologically free associative algebras corresponding
to} the topologically free ASI bialgebra
$(A_h,\circ_h,\Delta_h)$.} 
\end{pro}
\begin{thm}
    \label{eqvdadb}
    Let $(A,\circ,\Delta)$ be a finite-dimensional ASI bialgebra and $(A_h,\circ_h,\Delta_h)$ be a topologically free ASI bialgebra. Let $\cdot$ and $\cdot_h$ be the dual operation and the topological dual operation of $\Delta$ and $\Delta_h$, respectively. Define a binary operation $\odot$ on $A\oplus A^{\ast}$ and a binary operation $\odot_h$ on $A_h\oplus A^{\star}_h$ by {\rm(\ref{bod})} and {\rm(\ref{bodh})}, respectively. Then $(A_h,\circ_h,\Delta_h)$ is an ASI bialgebra deformation of $(A,\circ,\Delta)$ if and only if
$((A_h\oplus A^{\star}_h,\odot_h),A_h, A^{\star}_h)$ is an associative Manin triple deformation of $((A\oplus A^{\ast},\odot),A,A^{\ast})$.
\end{thm}
\begin{proof}
Assume that $((A_h\oplus A^{\star}_h,\odot_h),A_h, A^{\star}_h)$ is an associative Manin triple deformation of $((A\oplus A^{\ast},\odot),A,A^{\ast})$. Then $(A_h\oplus A^{\star}_h,\odot_h)$ is an associative deformation of $(A\oplus A^{\ast},\odot)$. It is direct to see that $(A_h,\circ_h)$ and $(A^{\star}_h,\cdot_h)$ are associative deformations of $(A,\circ)$ and $(A^{\ast},\cdot)$, respectively. Similar to (\ref{cc}), we show that $(A_h,\Delta_h)$ is a coassociative deformation of $(A,\Delta)$. Therefore, $(A_h,\circ_h,\Delta_h)$ is an ASI bialgebra deformation of $(A,\circ,\Delta)$.

Assume that $(A_h,\circ_h,\Delta_h)$ is an ASI bialgebra deformation of $(A,\circ,\Delta)$. Then (\ref{cimh}) holds. Thus, similar to (\ref{sci}), we obtain the following congruences.
    \begin{equation}\label{sci2}
        R^{\star}_{\circ_h}(x)b^{\ast}\equiv R_{\circ}^{\ast}(x)b^{\ast}\pmod h,\;\;L^{\star}_{\circ_h}(y)a^{\ast}\equiv
        L_{\circ}^{\ast}(y)a^{\ast}\pmod h,
        \;\;\forall x,y\in A,\;a^{\ast},b^{\ast}\in A^{\ast}.
    \end{equation}
Since $(A_h,\Delta_h)$ is a coassociative deformation of $(A,\Delta)$, similar to (\ref{cc}), we obtain (\ref{cdmh}). Thus, similar to (\ref{sci}),  we have
\begin{equation}\label{scd2}
R^{\star}_{\cdot_h}(a^{\ast})y\equiv
R_{\cdot}^{\ast}(a^{\ast})y\pmod
h,\;\;L^{\star}_{\cdot_h}(b^{\ast})x\equiv
R_{\cdot}^{\ast}(b^{\ast})x\pmod h,\;\;x,y\in A,\;a^{\ast},b^{\ast}\in
A^{\ast}.
\end{equation} Adding (\ref{cimh}), (\ref{sci2}), (\ref{cdmh}) and (\ref{scd2}) together, we show that $(A_h\oplus A^{\star}_h,\odot_h)$ is an associative
deformation of $(A\oplus A^{\ast},\odot)$, completing the proof.
\end{proof}



\begin{pro}
    \label{eqvbpb}{\rm(\cite{NB})}
    Let $A$ be a finite-dimensional vector space, $\{,\},\circ$ be binary operations on $A$ and $\delta,\Delta:A\to A\ot A$ be linear maps. Let $[,]$ and $\cdot$ be the dual operations of $\delta$ and $\Delta$, respectively. Define binary operations $\odot$ and $\llbracket,\rrbracket$ on $A\oplus A^{\ast}$ by {\rm(\ref{bod})} and the following equation, respectively.
        \begin{equation}\label{lrb}
    \llbracket x+a^{\ast},y+b^{\ast}\rrbracket:=\{x, y\}+\ad_{[,]}^{\ast}(b^{\ast})x-\ad_{[,]}^{\ast}(a^{\ast})y+\ad_{\{,\}}^{\ast}(y)a^{\ast}-\ad_{\{,\}}^{\ast}(x)b^{\ast}+ [a^{\ast},b^{\ast}],
    \end{equation}
for all $x,y\in A,\;a^{\ast},b^{\ast}\in A^{\ast}$. Then
$(A,\{,\},\circ,\delta, \Delta)$ is a Poisson bialgebra if and
only if $((A\oplus A^{\ast},\llbracket,\rrbracket,\odot),A,
A^{\ast})$ is a standard Manin triple of Poisson algebras. \delete{In this
case, $(A,\{,\},\circ,\delta, \Delta)$ is called the
\textbf{Poisson bialgebra corresponding to} the standard Manin
triple of Poisson algebras $((A\oplus
A^{\ast},\llbracket,\rrbracket,\odot),A, A^{\ast})$, and
$((A\oplus A^{\ast},\llbracket,\rrbracket,\odot),A, A^{\ast})$ is
called the \textbf{standard Manin triple of Poisson algebras
corresponding to } the Poisson bialgebra $(A,\{,\},\circ,\delta,
\Delta)$.} 
\end{pro}
\begin{thm}
  \label{qcldadb}
    Let $A$ be a finite-dimensional vector space, $\circ$ be a binary operation on $A$, $\Delta:A\to A\ot A$ be a linear map and $\cdot$ be the dual operation of $\Delta$. Let $\circ_h$ be a $\bf K$-bilinear operation on $A_h$, $\Delta_h:A_h\to A_h\hat{\ot} A_h$ be a $\bf K$-linear map and $\cdot_h$ be its topological dual operation.
    \begin{enumerate}
        \item\label{qcldadb1} Assume that $(A,\circ,\Delta)$ is a commutative and cocommutative ASI bialgebra, $(A_{h},\circ_{h},\Delta_{h})$ is an ASI bialgebra deformation and $(A,\{,\},\circ,\delta,\Delta)$ is the QCL.  Then $((A_h\oplus A^{\star}_h,\odot_h),A_h, A^{\star}_h)$, which is the standard Manin triple of topologically free associative algebras corresponding to $(A_{h},\circ_{h},\Delta_{h})$, is an associative Manin triple deformation of $((A\oplus A^{\ast},\odot),A,A^{\ast})$, which is the standard Manin triple of commutative associative algebras corresponding to $(A,\circ,\Delta)$. Moreover, $((A\oplus A^{\ast},\llbracket,\rrbracket,\odot),A, A^{\ast})$, which is the standard Manin triple of Poisson algebras corresponding to $(A,\{,\},\circ,\delta,\Delta)$, coincides with the QCL of $((A_h\oplus A^{\star}_h,\odot_h),A_h, A^{\star}_h)$.
         \item\label{qcldadb2}
         Define a binary operation $\odot$ on $A\oplus A^{\ast}$ and a binary operation $\odot_h$ on $A_h\oplus A^{\star}_h$ by {\rm(\ref{bod})} and {\rm(\ref{bodh})}, respectively.
         Assume that $((A_h\oplus A^{\star}_h,\odot_h),A_h, A^{\star}_h)$ is an associative Manin triple deformation of
         a standard Manin triple of commutative associative algebras $((A\oplus A^{\ast},\odot),A,A^{\ast})$ and the QCL is $((A\oplus A^{\ast},\llbracket,\rrbracket,\odot),A, A^{\ast})$.
          Then $(A_{h},\circ_{h},\Delta_{h})$, which is the topologically free ASI bialgebra corresponding to $((A_h\oplus A^{\star}_h,\odot_h),A_h, A^{\star}_h)$, is an ASI bialgebra deformation of $(A,\circ,\Delta)$, which is the commutative and cocommutative ASI bialgebra corresponding to $((A\oplus A^{\ast},\odot),A,A^{\ast})$. Moreover, the Poisson bialgebra corresponding to the standard Manin triple of Poisson algebras $((A\oplus A^{\ast},\llbracket,\rrbracket,\odot),A, A^{\ast})$ coincides with the QCL of $(A_{h},\circ_{h},\Delta_{h})$.
         \end{enumerate}
         That is, we have the following commutative diagram.
         \begin{equation}
         \begin{gathered}
\xymatrixcolsep{5pc}\xymatrixrowsep{5pc}
\xymatrix{
    *\txt{commutative and\\
        cocommutative\\ ASI bialgebra\\$(A,\circ,\Delta)$}\ar@<-3pt>[d]|-{\txt{\small Proposition~{\rm\ref{eqvbasi}}}}\ar[r]^{\txt{\small deformation}}&*\txt{topologically free\\ASI bialgebra\\$(A_{h},\circ_{h},\Delta_{h})$}\ar[r]^{\txt{\small QCL}}_{\txt{\small Theorem~{\rm\ref{thm-limasi}}}}\ar@<-3pt>[d]|-{\txt{\small Proposition~{\rm\ref{eqvbasih}}}}&*\txt{Poisson bialgebra\\$(A,\{,\},\circ,\delta,\Delta)$}\ar@<-3pt>[d]|-{\txt{\small Proposition~{\rm\ref{eqvbpb}}}}\\
    *\txt{standard Manin triple\\of commutative\\associative algebras\\$((A\oplus A^{\ast},\odot),A,A^{\ast})$}\ar[r]^{\txt{\small deformation}}\ar@<-3pt>[u]|-{\phantom{\txt{\small Proposition~{\rm\ref{mpasi}}}}}
    & *\txt{standard Manin triple\\of topologically free\\associative algebras\\$((A_h\oplus A^{\star}_h,\odot_h),A_h, A^{\star}_h)$}\ar@<-3pt>[u]|-{\phantom{\txt{\small Proposition~\ref{eqvtam}}}}\ar[r]^{\txt{\small QCL}}_{\txt{\small Theorem~{\rm\ref{qcldmt}}}}
    & *\txt{standard Manin \\triple of\\ Poisson algebras\\$((A\oplus A^{\ast},\llbracket,\rrbracket,\odot),$\\$A, A^{\ast})$.}\ar@<-3pt>[u]|-{\phantom{\txt{\small Proposition~\ref{eqvmpp}}}}
}\label{diag2}
 \end{gathered}
\end{equation}
\end{thm}
\begin{proof}
    (\ref{qcldadb1}). The first conclusion follows from Theorem~{\ref{eqvdadb}}. Since the QCL of $(A_h,\circ_h)$ is $(A,\{,\},\circ)$, we have
    \begin{equation}\label{qclbb}
        \{x,y\}\equiv\frac{x\circ_{h}y-y\circ_{h}x}{h}\pmod{h},\;\;\forall  x, y\in A.
    \end{equation}
Similar to (\ref{qclads1}), we obtain
      \begin{equation}\label{qclads12}
         \ad^{\ast}_{\{,\}}(y)a^{\ast}\equiv\frac{L^{\star}_{\circ_h}(y)a^{\ast}-R^{\star}_{\circ_h}(y)a^{\ast}}{h}\pmod{h},\;-\ad^{\ast}_{\{,\}}(x)b^{\ast}\equiv\frac{R^{\star}_{\circ_h}(x)b^{\ast}-L^{\star}_{\circ_h}(x)b^{\ast}}{h}\pmod{h},
      \end{equation}
      where $x,y\in A,a^{\ast},b^{\ast}\in  A^{\ast}$. Let $\cdot_h$ be the topological dual operation of $\Delta_h$ and $[,]$ be the dual operation of $\delta$. Then, similar to (\ref{qclads2}), the following congruence holds.
    \begin{equation}\label{qclads22}
       [a^{\ast},b^{\ast}]\equiv
           \frac{a^{\ast}\cdot_h b^{\ast}- b^{\ast}\cdot_h a^{\ast}}{h} \pmod{h},\;\;\forall a^{\ast},b^{\ast}\in  A^{\ast}.
    \end{equation}
Similar to (\ref{qclads3}), we have
 \begin{equation}\label{qclads32}
\ad^{\ast}_{[,]}(b^{\ast})x\equiv\frac{L^{\star}_{\cdot_h}(b^{\ast})x-R^{\star}_{\cdot_h}(b^{\ast})x}{h}\pmod{h},\;\;-\ad^{\ast}_{[,]}(a^{\ast})y\equiv\frac{R^{\star}_{\cdot_h}(a^{\ast})y-L^{\star}_{\cdot_h}(a^{\ast})y}{h}\pmod{h},
\end{equation}
where $x,y\in A,a^{\ast},b^{\ast}\in  A^{\ast}$.
Adding (\ref{qclbb}), (\ref{qclads12}), (\ref{qclads22}) and (\ref{qclads32}) together, we show that the QCL of $(A_h\oplus A^{\star}_h,\odot_h)$ coincides with $(A\oplus A^{\ast},\llbracket,\rrbracket,\odot)$,
completing the proof of (\ref{qcldadm1}).

       (\ref{qcldadb2}). The first conclusion follows from Theorem~{\ref{eqvdadb}}. Denote by $(A,\{,\},\circ,\delta,\Delta)$ the QCL of $(A_{h},\circ_{h},\Delta_{h})$.  Let $[,],\cdot$ be the restrictions of $\llbracket,\rrbracket,\odot$ on $A^{\ast}$, respectively. Let $\cdot_h$ be the restrictions of $\odot_h$ on $A^{\star}_h$.
       Then $\cdot$ is the dual operation of $\Delta$ because $(A,\circ,\Delta)$ is the commutative and cocommutative ASI bialgebra corresponding to $((A\oplus A^{\ast},\odot),A,A^{\ast})$. On the other hand, similar to (\ref{qclads2}), we show that $[,]$ coincides with the dual operation of $\delta$, yielding the conclusion.
\end{proof}


At the end of this paper, we would like to point out that,
similarly, we can also illustrate the deformations-quasiclassical
 limits process for associative Manin triple deformations
 explicitly in terms of derivations, corresponding to the
 antisymmetric infinitesimal bialgebra deformations via coherent
 derivations in Section~\ref{sec2.3} as well as the associative  matched
 pair  deformations via derivations in Section~\ref{sec3.3}.

\noindent {\bf Acknowledgements.} This work is supported by NSFC
(12271265, 12261131498, W2412041, 12671037), Fundamental Research Funds for
the Central Universities and Nankai Zhide Foundation. S. Chen is supported by the Scientific Research Foundation of the Education Department of Hunan Province (Grant No. 25C0411). The authors thank
the referees for valuable suggestions. 

\noindent
{\bf Declaration of interests. } The authors have no conflicts of interest to disclose.

\noindent
{\bf Data availability. } Data sharing is not applicable as no new data were created or analyzed.


\end{document}